\documentclass{amsart}

\newif\ifcomments
\commentstrue  

\usepackage{amsmath,amsthm,amssymb,amscd,url,enumerate}
\usepackage[utf8]{inputenc}
\usepackage{hyperref}
\usepackage{pdflscape}
\usepackage{caption}
\usepackage[noadjust]{cite}
\usepackage{xypic}
\usepackage{multirow}

\usepackage{tikz}
\usetikzlibrary{decorations.markings,arrows}
\usepackage{tikz-cd}
\usepackage{amsfonts}
\usepackage{amsmath}
\usepackage{graphicx}
\usepackage{amsthm}
\usepackage{xcolor}
\usepackage{mathrsfs}
\usepackage[margin=1in]{geometry}
\usepackage[ruled,vlined,linesnumbered,algosection]{algorithm2e}
\usepackage{algpseudocode}
\usepackage[inline]{enumitem}
\usepackage{csquotes}
\usepackage[all,cmtip]{xy} 
\usepackage{float}
\usepackage[caption = false]{subfig}

\definecolor{refkey}{gray}{.25}
\definecolor{labelkey}{gray}{.75}

\SetEndCharOfAlgoLine{}
\DecMargin{2.2mm}
\SetKwInput{Input}{Input}
\SetKwInput{Output}{Output}
\SetKwFor{While}{While}{do}{}
\SetKwFor{ForEach}{For each}{do}{}
\SetKwIF{If}{ElseIf}{Else}{If}{then}{else if}{else}{endif}
\SetKw{Return}{Return}
\SetKw{KwRet}{return}
\SetKw{Recurse}{Recurse}
\SetKwFor{For}{For}{do}{}
\SetKw{Break}{break}

\newcommand{\poly}{\operatorname{poly}}

\theoremstyle{plain} 
\newtheorem{theorem}{Theorem} 
\newtheorem{theorem*}{Theorem} 

\newtheorem{proposition}[theorem]{Proposition}
\newtheorem{lemma}[theorem]{Lemma}
\newtheorem{corollary}[theorem]{Corollary}
\newtheorem{heuristic}[theorem]{Heuristic}

\theoremstyle{definition}
\newtheorem{definition}[theorem]{Definition}

\theoremstyle{remark}
\newtheorem{remark}[theorem]{Remark}

\newtheorem{example}[theorem]{Example}

\newtheorem*{lemma*}{Lemma}

\numberwithin{theorem}{section}

\newcommand{\Gal}{\operatorname{Gal}}

\newcommand{\CC}{\mathbb{C}}
\newcommand{\ZZ}{\mathbb{Z}}
\newcommand{\QQ}{\mathbb{Q}}
\newcommand{\FF}{\mathbb{F}}

\newcommand{\Fpbar}{{\overline{\mathbb{F}}_p}}

\newcommand{\End}{\operatorname{End}}
\newcommand{\Aut}{\operatorname{Aut}}

\newcommand{\disc}{\operatorname{disc}}

\newcommand{\Einit}{E_{\operatorname{init}}}

\newcommand{\SSO}{\operatorname{SS}_\mathcal{O}^{\text{pr}}}
\newcommand{\ClO}{\Cl(\mathcal{O})}

\newcommand{\hO}{h_\mathcal{O}}

\newcommand{\Frob}{\operatorname{Frob}}
\newcommand{\Stab}{\operatorname{Stab}}

\newcommand{\Np}{\# \mathcal{G}_\ell }
\newcommand{\Sp}{\# \mathcal{S} }

\DeclareMathOperator{\Cl}{Cl}
\DeclareMathOperator{\Ell}{\mathcal{E}\ell\ell}

\let\SS\relax   \DeclareMathOperator{\SS}{SS}
\definecolor{Bittersweet}{rgb}{1.0, 0.44, 0.37}
\definecolor{KateColour}{rgb}{0.3, 0.6, 0.3}

\ifcomments
\newcommand{\RS}[1]{\textcolor{violet}{{\sf (Renate:} {\sl{#1})}}}
\newcommand{\KL}[1]{\textcolor{cyan}{{\sf (Kristin:} {\sl{#1})}}}
\newcommand{\KS}[1]{\textcolor{KateColour}{{\sf (Kate:} {\sl{#1})}}}
\newcommand{\HT}[1]{\textcolor{cyan}{{\sf (Ha:} {\sl{#1})}}}
\newcommand{\MC}[1]{\textcolor{Bittersweet}{{\sf (Mingjie:} {\sl{#1})}}}
\newcommand{\SA}[1]{\textcolor{blue}{{\sf (Sarah:} {\sl{#1})}}}
\else
\newcommand{\RS}[1]{}
\newcommand{\KL}[1]{}
\newcommand{\KS}[1]{}
\newcommand{\HT}[1]{}
\newcommand{\MC}[1]{}
\newcommand{\SA}[1]{}
\fi

\title{Orientations and cycles in supersingular isogeny graphs}
\author[Arpin, Chen, Lauter, Scheidler, Stange, Tran]{Sarah Arpin, Mingjie Chen, Kristin E. Lauter, Renate Scheidler, Katherine E. Stange, Ha T. N. Tran}

\date{\today}

\address{%
Mathematics Institute,
Universiteit Leiden,
Leiden, The Netherlands}
\email{S.A.Arpin@math.leidenuniv.nl}
\address{School of Computer Science, University of Birmingham, University Road West, Birmingham, UK B15 2TT}
\email{m.chen.1@bham.ac.uk}
\address{Facebook AI Research, Meta, Seattle, WA}
\email{klauter@fb.com}
\address{%
Department of Mathematics and Statistics, University of Calgary,
2500 University Drive NW, Calgary, Alberta, Canada T2N 1N4}
\email{rscheidl@ucalgary.ca}
\address{%
Department of Mathematics, University of Colorado,
Campus Box 395, Boulder, Colorado 80309-0395}
\email{kstange@math.colorado.edu}
\address{Department of Mathematical and Physical Sciences, Concordia University of Edmonton, 
7128 Ada Blvd NW, Edmonton, AB T5B 4E4, Canada}
\email{hatran1104@gmail.com}



\thanks{
Katherine E. Stange and Sarah Arpin were supported by NSF-CAREER CNS-1652238.  Katherine E. Stange was also supported by Simons Fellowship 822143. Ha T. N. Tran was supported by the Natural Sciences and Engineering Research Council of Canada (NSERC) (funding RGPIN-2019-04209 and DGECR-2019-00428). R. Scheidler was supported by the Natural Sciences and Engineering Research Council of Canada (NSERC) (funding RGPIN-2019-04844).  Mingjie Chen was supported by NSF grants DMS-1844206, DMS-1802161.
}

\begin{document}

\maketitle

\begin{abstract}
The paper concerns several theoretical aspects of  oriented supersingular $\ell$-isogeny volcanoes and their relationship to closed walks in the supersingular $\ell$-isogeny graph.  Our main result is a  bijection between 
the rims of the union of all oriented supersingular $\ell$-isogeny volcanoes over $\overline{\FF}_p$ (up to conjugation of the orientations), and 
\emph{isogeny cycles} (non-backtracking closed walks which are not powers of smaller walks) of the supersingular $\ell$-isogeny graph over $\overline{\mathbb{F}}_p$.  The exact proof and statement of this bijection are made more intricate by special behaviours arising from extra automorphisms and the ramification of $p$ in certain quadratic orders.  We use the bijection to count isogeny cycles of given length in the supersingular $\ell$-isogeny graph exactly as a sum of class numbers of these orders, and also give an explicit upper bound by estimating the class numbers.
\end{abstract}

\section{Introduction}

Fix primes $\ell < p$.  The supersingular $\ell$-isogeny graph $\mathcal{G}_\ell$ is the directed graph whose vertices are the $\overline{\FF}_p$-isomorphism classes of supersingular elliptic curves, and whose edges are isogenies, up to post-composition by an automorphism.  The graph is $(\ell+1)$-regular, has around $p/12$ vertices, and consists of one connected component.  However, it has few obvious symmetries:  its structure more closely resembles that of a random regular graph in its statistical behaviour, including its graph spectrum.  It is a Ramanujan graph.

The graph $\mathcal{G}_\ell$ was introduced into cryptography in 2006 in ~\cite{CharlesGorenLauter}, where the hard problem of finding paths in the graph was proposed.  In the cryptographic setting, $p$ is always taken to be a large prime of cryptographic size and $\ell$ is a small prime like $\ell = 2$ or $3$.  A key exchange protocol called SIDH based on the path-finding problem in $\mathcal{G}_\ell$ was proposed in~\cite{dFJP} and was considered for standardization in the NIST Post-Quantum Cryptography process (2017--2022).  Although SIDH has been broken \cite{castryck-sidh-attack,maino-attack-sidh, damien-attack}, so far there do not exist polynomial time algorithms (polynomial in $\log p$) to solve either the \emph{path finding problem} (to find a path joining two given supersingular curves), or the \emph{endomorphism ring problem} (to compute the endomorphism ring, either as a maximal order in a quaternion algebra or by giving a basis of endomorphisms) of a supersingular curve.  These problems are closely related to one another \cite{EHLMP_reductions,Wesolowski_IsogPathandEndoRing}, and to other cryptographic protocols such as CSIDH \cite{CSIDH} and SQISign \cite{sqisign}.

The graph $\mathcal{G}_\ell$, like any superhero worth its salt, has an `alter ego,' which is as the graph of maximal orders in a quaternion algebra $B_{p,\infty}$ ramified at $p$ and $\infty$.  For each elliptic curve vertex, the endomorphism ring of the curve is such a maximal order.  However, even viewed in this way, the graph is still its apparently disordered self.  To reveal some familiar structure with which to navigate the graph, one can focus, one at a time, on the quadratic rings embedded in the quaternion algebra.

The aforementioned structure recalls to mind the analogous graph for ordinary elliptic curves, whose endomorphism rings are merely quadratic orders.  One can form an ordinary $\ell$-isogeny graph just as one does a supersingular one:  vertices are isomorphism classes of ordinary elliptic curves, and edges are $\ell$-isogenies \cite{FouquetMorain} \cite{Isogeny_volcanoes}.  In this case, the endomorphism rings of the curves in a connected component are all orders in the same quadratic field.  Stratifying the graph according to the conductors of these orders, with larger orders at higher `altitude,' we see at the top a cycle, or `rim,' from which trees descend down the sides of the `volcano' (see Figure~\ref{fig:oriented_volcano}).  This simple, organized picture is in contrast to the complexity of the supersingular graph.  

The storyline of the present paper is that the (finite) supersingular graph $\mathcal{G}_\ell$ can be understood as the result of superimposing infinitely many volcanoes, each obtained by focusing on a single quadratic field embedded inside the quaternion algebra $B_{p,\infty}$ and its intersection with the various maximal orders.  The volcanoes in this case are not the volcanoes of the ordinary $\ell$-isogeny graph, but rather the connected components of the \emph{oriented supersingular $\ell$-isogeny graph associated to a quadratic field $K$}, denoted by $\mathcal{G}_{K,\ell}$.  This graph has recently been studied by Col\`o-Kohel \cite{colo2019orienting} and Onuki \cite{onuki2021} in the cryptographic context.  The vertices of $\mathcal{G}_{K,\ell}$ are pairs $(E, \iota)$ consisting of an isomorphism class of supersingular curves, together with an embedding $\iota\colon K \rightarrow \End^0(E)$ of a quadratic field into the endomorphism algebra of the curve, called a \emph{$K$-orientation}.  Edges are again $\ell$-isogenies.  (For the precise definitions, see Section~\ref{sec:background}.)  
  
Upon forgetting orientations, each of these volcanoes covers the supersingular $\ell$-isogeny graph.  In particular, each rim maps to a closed walk in the $\ell$-isogeny graph.

 \begin{figure}
    \centering
    \scalebox{.6}{
    \begin{tikzpicture}
\node[circle,minimum size=.8cm,draw,fill=black] (E1) at (1,3) {};
\node[circle,minimum size=.8cm,draw,fill=black] (E2) at (2,1) {};
\node[circle,minimum size=.8cm,draw,fill=black] (E3) at (3.5,4) {};
\node[circle,minimum size=.8cm,draw,fill=black] (E4) at (5,1) {};
\node[circle,minimum size=.8cm,draw,fill=black] (E5) at (6,3) {};

\node[circle,minimum size=.6cm,draw,fill=black!50] (E6) at (1,.4) {};
\node[circle,minimum size=.6cm,draw,fill=black!50] (E7) at (0.3,1.4) {};
\node[circle,minimum size=.6cm,draw,fill=black!50] (E8) at (1.2,-1.5) {};
\node[circle,minimum size=.6cm,draw,fill=black!50] (E9) at (2.8,-1.5) {};
\node[circle,minimum size=.6cm,draw,fill=black!50] (E10) at (4.2,-1.5) {};
\node[circle,minimum size=.6cm,draw,fill=black!50] (E11) at (5.8,-1.5) {};
\node[circle,minimum size=.6cm,draw,fill=black!50] (E12) at (6,.4) {};
\node[circle,minimum size=.6cm,draw,fill=black!50] (E13) at (6.7,1.4) {};
\node[circle,minimum size=.6cm,draw,fill=black!50] (E14) at (3,1.9) {};
\node[circle,minimum size=.6cm,draw,fill=black!50] (E15) at (4,1.9) {};

\draw[-, line width=1.5mm, dashed] (E1) to (E2);
\draw[-, line width=1.5mm, dashed] (E3) to (E1);
\draw[-, line width=1.5mm, dashed] (E3) to (E5);
\draw[-, line width=1.5mm, dashed] (E4) to (E5);
\draw[-, line width=1.5mm, dashed] (E4) to (E2);
\draw[-, line width=.6mm] (E1) to (E6);
\draw[-, line width=.6mm] (E1) to (E7);
\draw[-, line width=.6mm] (E2) to (E8);
\draw[-, line width=.6mm] (E2) to (E9);
\draw[-, line width=.6mm] (E4) to (E10);
\draw[-, line width=.6mm] (E4) to (E11);
\draw[-, line width=.6mm] (E5) to (E12);
\draw[-, line width=.6mm] (E5) to (E13);
\draw[-, line width=.6mm] (E3) to (E14);
\draw[-, line width=.6mm] (E3) to (E15);
\draw[dash pattern={on 12pt off 2pt on 1pt off 2pt on 1pt off 2pt on 1pt}, line width=.3mm] (E8) -- (.8,-2.5);
\draw[dash pattern={on 12pt off 2pt on 1pt off 2pt on 1pt off 2pt on 1pt}, line width=.3mm] (E8) -- (1.2,-2.5);
\draw[dash pattern={on 12pt off 2pt on 1pt off 2pt on 1pt off 2pt on 1pt}, line width=.3mm] (E8) -- (1.6,-2.5);

\draw[dash pattern={on 12pt off 2pt on 1pt off 2pt on 1pt off 2pt on 1pt}, line width=.3mm] (E9) -- (2.4,-2.5);
\draw[dash pattern={on 12pt off 2pt on 1pt off 2pt on 1pt off 2pt on 1pt}, line width=.3mm] (E9) -- (2.8,-2.5);
\draw[dash pattern={on 12pt off 2pt on 1pt off 2pt on 1pt off 2pt on 1pt}, line width=.3mm] (E9) -- (3.2,-2.5);

\draw[dash pattern={on 12pt off 2pt on 1pt off 2pt on 1pt off 2pt on 1pt}, line width=.3mm] (E10) -- (3.8,-2.5);
\draw[dash pattern={on 12pt off 2pt on 1pt off 2pt on 1pt off 2pt on 1pt}, line width=.3mm] (E10) -- (4.2,-2.5);
\draw[dash pattern={on 12pt off 2pt on 1pt off 2pt on 1pt off 2pt on 1pt}, line width=.3mm] (E10) -- (4.6,-2.5);

\draw[dash pattern={on 12pt off 2pt on 1pt off 2pt on 1pt off 2pt on 1pt}, line width=.3mm] (E11) -- (6.2,-2.5);
\draw[dash pattern={on 12pt off 2pt on 1pt off 2pt on 1pt off 2pt on 1pt}, line width=.3mm] (E11) -- (5.8,-2.5);
\draw[dash pattern={on 12pt off 2pt on 1pt off 2pt on 1pt off 2pt on 1pt}, line width=.3mm] (E11) -- (5.4,-2.5);

\draw[dash pattern={on 10pt off 2pt on 1pt off 2pt on 1pt off 2pt on 1pt}, line width=.3mm] (E15) -- (4.3,1.1);
\draw[dash pattern={on 10pt off 2pt on 1pt off 2pt on 1pt off 2pt on 1pt}, line width=.3mm] (E15) -- (4,1.1);
\draw[dash pattern={on 10pt off 2pt on 1pt off 2pt on 1pt off 2pt on 1pt}, line width=.3mm] (E15) -- (3.7,1.1);

\draw[dash pattern={on 10pt off 2pt on 1pt off 2pt on 1pt off 2pt on 1pt}, line width=.3mm] (E14) -- (3.3,1.1);
\draw[dash pattern={on 10pt off 2pt on 1pt off 2pt on 1pt off 2pt on 1pt}, line width=.3mm] (E14) -- (3,1.1);
\draw[dash pattern={on 10pt off 2pt on 1pt off 2pt on 1pt off 2pt on 1pt}, line width=.3mm] (E14) -- (2.7,1.1);

\draw[dash pattern={on 10pt off 2pt on 1pt off 2pt on 1pt off 2pt on 1pt}, line width=.3mm] (E7) -- (0,.6);
\draw[dash pattern={on 10pt off 2pt on 1pt off 2pt on 1pt off 2pt on 1pt}, line width=.3mm] (E7) -- (.2,.5);
\draw[dash pattern={on 10pt off 2pt on 1pt off 2pt on 1pt off 2pt on 1pt}, line width=.3mm] (E7) -- (.4,.4);

\draw[dash pattern={on 10pt off 2pt on 1pt off 2pt on 1pt off 2pt on 1pt}, line width=.3mm] (E6) -- (.7,-.4);
\draw[dash pattern={on 10pt off 2pt on 1pt off 2pt on 1pt off 2pt on 1pt}, line width=.3mm] (E6) -- (.9,-.5);
\draw[dash pattern={on 10pt off 2pt on 1pt off 2pt on 1pt off 2pt on 1pt}, line width=.3mm] (E6) -- (1.1,-.6);

\draw[dash pattern={on 10pt off 2pt on 1pt off 2pt on 1pt off 2pt on 1pt}, line width=.3mm] (E12) -- (6.3,-.4);
\draw[dash pattern={on 10pt off 2pt on 1pt off 2pt on 1pt off 2pt on 1pt}, line width=.3mm] (E12) -- (6.1,-.5);
\draw[dash pattern={on 10pt off 2pt on 1pt off 2pt on 1pt off 2pt on 1pt}, line width=.3mm] (E12) -- (5.9,-.6);
\draw[dash pattern={on 10pt off 2pt on 1pt off 2pt on 1pt off 2pt on 1pt}, line width=.3mm] (E13) -- (6.6,.4);
\draw[dash pattern={on 10pt off 2pt on 1pt off 2pt on 1pt off 2pt on 1pt}, line width=.3mm] (E13) -- (6.8,.5);
\draw[dash pattern={on 10pt off 2pt on 1pt off 2pt on 1pt off 2pt on 1pt}, line width=.3mm] (E13) -- (7,.6);
\end{tikzpicture}}
\caption{An oriented 3-isogeny volcano. The rim is comprised of the five black vertices. These vertices are primitively oriented by a 3-fundamental order $\mathcal{O}$. The grey vertices at altitude~1 of the volcano are primitively oriented by the suborder of $\mathcal{O}$ of index 3. The five thick dotted edges along the rim indicate horizontal oriented 3-isogenies, and the remaining black edges are ascending/descending.}
\label{fig:oriented_volcano}
\end{figure}
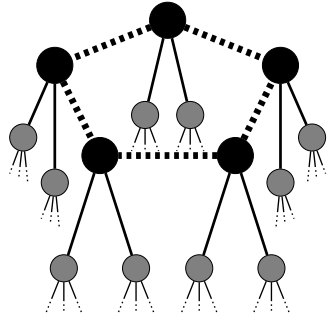

We concern ourselves with a particular type of closed walk in $\mathcal{G}_\ell$: one with no backtracking and which is not a repeat of a smaller closed walk.  We keep its direction (traversed `forward' or `backward'), but forget any choice of basepoint, and call this an \emph{isogeny cycle} (Definition~\ref{defn:isogeny-cycle}). 
The main result of this paper is that each isogeny cycle is obtained, by forgetting orientations, from exactly one oriented volcano rim (up to conjugation of the rim; see Section~\ref{sec:bij-statement} for definitions).

\begin{theorem*}[See precise version in Theorem~\ref{thm:mainbij-nobase}]
\label{thm:mainbij-intro}
Fix distinct primes $\ell < p$.  Let $r > 2$.
The isogeny cycles of length $r$ in $\mathcal{G}_\ell$ are in bijection with the rims of length $r$ of the union of all oriented supersingular $\ell$-isogeny volcanoes over $\overline{\FF}_p$, up to conjugation of the orientations.
\end{theorem*}

The exact statement is discussed in Section~\ref{sec:bij-statement}, and a detailed example is given in Section~\ref{sec:examples}.  At the highest level, the bijection is given by a simple process.  Given a rim, one can forget the orientation information to obtain an isogeny cycle in the supersingular graph.  Conversely, given an isogeny cycle, one can compose the sequence of isogenies it comprises, to obtain an endomorphism.  This endomorphism induces an orientation, placing the initial curve on a rim.  The proof, however, spans around a third of the paper, in large part because of the existence of special cases and issues spawned by extra automorphisms and quadratic fields in which $p$ ramifies.  In particular, \emph{the bijection is not canonical} due to these considerations.  See also Remark~\ref{remark:ssgraph}, where we make some philosophical comments on the obstacles presented by extra automorphisms in the supersingular isogeny graph.

We use the bijection of Theorem~\ref{thm:mainbij-intro} to count isogeny cycles of length $r$ in $\mathcal{G}_\ell$.  The essential ingredient is that the rims of $K$-oriented $\ell$-isogeny volcanoes correspond to cycles of length $r$ in the permutation on the oriented rim vertices arising from the action of an element of order $r$ in the class group of an order of $K$.

\begin{theorem*}[See precise versions in Theorem~\ref{thm:numcycles-randwalk}, Corollary~\ref{cor:h} and Corollary~\ref{cor:cyclesbound}]
The number of isogeny cycles of length $r$ in $\mathcal{G}_\ell$ is given by 
\begin{equation}
    \label{eqn:sumintro}
    \frac{1}{r} \sum_{\mathcal{O} \in \mathcal{I}_r} \epsilon_{\mathcal{O},\ell} h_\mathcal{O},
\end{equation}
where $\mathcal{I}_r$ is a set of quadratic orders satisfying certain splitting properties for $\ell$ and $p$ specified in Section~\ref{sec:cycle-order}, and $\epsilon_{\mathcal{O},r}$ is defined in Theorem~\ref{thm:mainbij-fibre}, and satisfies $1 \le \epsilon_{\mathcal{O},r} \le 2$, with $\epsilon_{\mathcal{O},r} = 2$ whenever $p$ is unramified in $\mathcal{O}$.
The quantity \eqref{eqn:sumintro} is bounded above by
\[
\frac{2\pi e^{0.578} \log (4\ell)}{3} \, \ell^r (\log r + c) + O(\ell^{3r/4} \log r)
\]
as $r \rightarrow \infty$, where the constant $c$ is explicit, and the implied constant in the big $O$ notation can be made explicit using Corollary~\ref{cor:cyclesbound}. For $p\equiv 1\pmod{12}$, the quantity \eqref{eqn:sumintro} asymptotically approaches $\frac{\ell^{r}}{2r}$ as $r\to\infty$.
\end{theorem*}

The asymptotic for the number of isogeny cycles is as expected for Ramanujan graphs (see Section~\ref{sec:asymptotic}) and we give a standard but self-contained proof for expander graphs.  However, the class number formula and the explicit upper bound are consequences of the main bijection between rims and isogeny cycles. Figure~\ref{fig:cycledata} and Section~\ref{sec:examples} contain experimental data confirming our results.

Our strategy for counting isogeny cycles leads immediately to an algorithm to list all isogeny cycles of length $r$ in $\mathcal{G}_\ell$.  In particular, the discriminants of elements of $\mathcal{I}_r$ are bounded by $\ell^r$.  The $j$-invariants of the isogeny cycles are 
roots of the Hilbert class polynomials of the orders $\mathcal{O} \in \mathcal{I}_r$. The cycle in $\mathcal{G}_\ell$ can be determined by use of the modular polynomial $\Phi_\ell$. This generalizes a simple congruence condition check on $p$ used in \cite[Section 5.3.4]{CharlesGorenLauter} to rule out the existence of small cycles for the Charles-Goren-Lauter hash function.  The exact value of $\epsilon_{\mathcal{O},r}$ in the ramified cases depends upon the theory of the ring class field of $K$, and is given in terms of the class number and genus number (Section~\ref{sec:classfieldtheory}).

In the course of the proof, we need to clarify a point about the number of orientations of a curve.
When $p$ does not split in $\mathcal{O}$, Onuki \cite{onuki2021} considered the set $\SSO$ of pairs $(E,\iota)$ of a supersingular isogeny graph together with a primitive $\mathcal{O}$-orientation (meaning a $K$-orientation $\iota$ such that $\iota(K) \cap \End(E) = \mathcal{O}$), and the set $\rho(\Ell(\mathcal{O}))$ of pairs $(E,\iota) \in \SSO$ obtained by reducing an elliptic curve over $\CC$ with normalized complex multiplication by $\mathcal{O}$ (where normalized means that $[\alpha]^* \omega = \alpha \omega$ for the invariant differential $\omega$; see Section~\ref{sec:back:sso}.)
Onuki showed that these two sets are not always equal; sometimes $\#\SSO = 2 \# \rho(\Ell(\mathcal{O}))$.  This raises the natural question of when this case occurs.
In this paper, we answer that it occurs if and only if $p$ is inert in $K$.  (In the ramified case, $\#\SSO = \#\rho(\Ell(\mathcal{O}))$ exactly.)  We prove this using the ring class field and the action of conjugation (Theorem~\ref{thm:sso}), but note that the result is also proved using Deuring lifting in \cite[Lemma 3.2]{EOY}.

We also explicitly find the in- and out-degrees of the vertices of all possible $K$-oriented $\ell$-isogeny graphs.  The $K$-oriented graph is $(\ell+1)$-regular except at vertices corresponding to curves with extra $K$-automorphisms (i.e.\ automorphisms which preserve the orientation), where the out-degree may be slightly less than $\ell+1$.  This is made explicit in Remark~\ref{prop:gkl-auto} and Proposition  \ref{prop:ellplusone}.  In connection with this, we provide a new direct proof of the volcano structure of the graph, and the 
count of ascending, descending and horizontal edges even in exceptional cases, in Proposition~\ref{prop:volcanostructure}.

The motivation for this paper is prior work by the same authors, summarized in Section~\ref{sec:pathfinding}, in which the oriented volcano structure is used to navigate the supersingular $\ell$-isogeny graph \cite{paperone}.  In particular, given an initial vertex in the supersingular $\ell$-isogeny graph, and given an endomorphism of that curve, one asks to navigate to a target curve (typically $j=1728$).  The endomorphism provides an orientation of the initial curve, and hence a location in an oriented supersingular $\ell$-isogeny volcano.  Then, using methods of determining ascending and descending directions, one can navigate the oriented volcano.  The target curve must be given an orientation on the same volcano (algorithms are provided in \cite{paperone}).  From both curves, one ascends to the rim and hopes to connect the paths.  Having found a path on an oriented volcano, one takes the image in the original supersingular $\ell$-isogeny graph.

The algorithms provided in \cite{paperone} explicitly relied on heuristics about oriented volcanoes.  They also implicitly hinted at the sort of bijection given in our first theorem above.  In Section~\ref{sec:randwalk} of the present paper, we revisit the explicit heuristics of \cite{paperone} and provide some partial results.  In particular, we do the following.

\begin{enumerate}
    \item We give an explicit bound on the distance from uniform of the distribution of endpoints of non-backtracking random walks in $\mathcal{G}_\ell$ from an initial set of vertices when $p\equiv1\pmod{12}$ (Proposition~\ref{prop:randwalk2}). 
    \item We estimate the expected distance of any fixed vertex to a set of $j$-invariants of a given size (Corollary~\ref{cor:allcurves}).
    \item We fix an oriented volcano and bound the depth at which any $j$-invariant first occurs (Section~\ref{sec:depth-of-j}).
    \item We fix a $j$-invariant and consider all the possible orientations by a given field, and the distribution of this collection of oriented curves on the various oriented volcanoes (Section~\ref{sec:volcanoes-upon-which}).  We are able to prove a result that differs slightly from the heuristic used in \cite{paperone}.
\end{enumerate}

As an additional minor point of interest, Lemma~\ref{lem:loopsandauts} shows that the kernel of an endomorphism is fixed by an extra automorphism if and only if its field of definition lies in the field of definition of the automorphism. As a corollary, we can classify the behaviour of loops at $j=0$ and $j=1728$ in Corollary~\ref{cor:loops}; this is a new method of proof, and slight generalization, of a known result \cite{AAM, YZ, LOX2}.

The topic of closed walks\footnote{often simply called cycles in the isogeny literature; we use the term \emph{closed walk} to more accurately follow the graph theory literature} in the supersingular $\ell$-isogeny graph was first studied in \cite{CharlesGorenLauter}, where it was observed that closed walks give endomorphisms, and so the splitting behaviour of $p$ in various extensions controlled the existence of small cycles.  The endomorphisms created by closed walks were studied in the thesis of Kohel \cite{KohelThesis}, and further in \cite{bank2019cycles}, with the goal of understanding when such cycles are independent and generate an endomorphism ring.  These sources also discuss the computation of the trace of a cycle.  The neighbourhoods of $0$ and $1728$ are particularly intricate, and their loops and cycles are discussed in \cite{AAM, YZ, LOX, LOX2, XLD}.  Cycles in the SIKE graph which pass through the secret key are discussed in \cite{Onuki2020TheEO}.

Further back in the literature, Gross \cite{Gross_Heights} considered counting endomorphisms of a fixed degree by computing the traces of matrices derived from looking at certain modular forms of weight 2 for the group $\Gamma_0(N)$.  Thus Gross counts closed walks of a different sort, and in particular allows backtracking (which we do not).  As the degree increases, the proportion of the count due to backtracking increases as well.  We provide a comparison at the end of our extended example in Section~\ref{sec:examples}.

The study of $K$-orientations is in some sense as old as the study of optimal embeddings in quaternion algebras, and has naturally appeared in the study of isogeny graphs, for example in \cite{BonehLove}.  In the context of isogeny-based cryptography, the graph $\mathcal{G}_{K,\ell}$ was studied by Col\`o and Kohel in order to propose OSIDH, a new key exchange protocol \cite{colo2019orienting} (see also \cite{CPV}).  Onuki considered the difference between $\SSO$ and $\rho(\Ell(\mathcal{O}))$ \cite{onuki2021}.  In cryptography, the class group action on oriented curves has been used constructively \cite{chenu2021higherdegree}, and the hardness of the group action problem has been considered both classically and quantumly \cite{CHVW} \cite{WesolowskiOrientations}; a new variation called the $\mathcal{O}$-uber isogeny problem appears in \cite{SETA} (see also \cite{CPV}).  The volcano structure provides some algorithmic approaches to hard problems, which includes the recent work of the present authors \cite{paperone} and related work \cite{WesolowskiOrientations}. Recent work also asks about the number of orientable curves \cite{Antonin_lowerbound2022}; see also \cite{BonehLove}.  Finally, specific orientations underlie the line of work in \cite{deQuehenEtAl_ImprovedTorPt} and \cite{SETA}.

\textbf{Outline of the paper.}
Section~\ref{sec:background} contains the background on oriented and unoriented supersingular $\ell$-isogeny graphs. 
In Section~\ref{sec:bij}, we state and prove  Theorem~\ref{thm:mainbij-intro}, and this section is
the heart of the paper.   Section~\ref{sec:classfieldtheory} settles the question of when $\rho(\Ell(\mathcal{O})) = \SSO$ and discusses the action of conjugation on $\SSO$.  Section~\ref{sec:cycle-order} specifies exactly the number of rims associated to a given quadratic order, in terms of class and genus numbers.  Section~\ref{sec:examples} provides an extended example of the main bijection of Theorem~\ref{thm:mainbij-intro}.   Section~\ref{sec:counting} counts the number of isogeny cycles in $\mathcal{G}_\ell$ asymptotically, as well as exactly in terms of class numbers, and gives an explicit upper bound.  Section~\ref{sec:pathfinding} gives an overview of our previous work and how it motivates the work in this paper, as well as the heuristics it depends upon.  Section~\ref{sec:randwalk} partially addresses some of these heuristic questions by use of methods of expander graphs.

\subsection*{Acknowledgements}  Thank you to David Kohel, Greg Martin, Eli Orvis, Christophe Petit, Lillian B. Pierce, Jan Vonk, and Jonathan Wise for helpful discussions and feedback. We would also like to thank the conference \emph{Women in Numbers 5} for the opportunity to form this research group.  We also appreciate the detailed comments of our anonymous referees.

\section{Mathematical preliminaries on oriented supersingular $\ell$-isogeny graphs and their volcanoes}
\label{sec:background}

\subsection{The supersingular $\ell$-isogeny graph $\mathcal{G}_\ell$}
\label{sec:back:ss}

Let $\ell < p$ be distinct primes.  We let $\mathcal{G}_\ell$ denote the supersingular $\ell$-isogeny graph over $\overline{\FF}_p$, defined as the directed graph whose vertices are isomorphism classes of supersingular elliptic curves over $\overline{\FF}_p$, and whose directed edges are isogenies of degree $\ell$ up to post-composition by an automorphism.  The number $\Np$ of supersingular $j$-invariants over $\Fpbar$ (see \cite{silverman2009arithmetic}) is
\begin{equation}\label{eq:number_supersingular_js}
    \Np = \left\lfloor \frac{p}{12} \right\rfloor + \begin{cases}
    0 &\text{ if }p\equiv 1\pmod{12},\\
    1 &\text{ if }p\equiv5 \text{ or }7\pmod{12},\\
    2&\text{ if }p\equiv 11\pmod{12}.
    \end{cases}
\end{equation}

Supersingular elliptic curves with $j$-invariants equal to $0$ and $1728$ have extra automorphisms, beyond the usual automorphisms $[\pm1]$.  Let $E_0: y^2 = x^3 - 1$ and $E_{1728}: y^2 = x^3 - x$ with $j(E_0) = 0, j(E_{1728}) = 1728$. The curve $E_0/\Fpbar$ is supersingular for any $p\equiv 2\pmod{3}$, and $E_{1728}/\Fpbar$ is supersingular for any $p\equiv 3\pmod{4}$. For the sake of this discussion, we assume these curves are supersingular. We provide their explicit endomorphism rings below, referenced from \cite{mcmurdy2014explicit,kate1728note}.

The automorphism group for $E_{1728}$ is $\Aut(E_{1728}) = \{[\pm 1], [\pm i]\}$, where $[i](x,y) := (-x,iy)$ for $i\in\mathbb{F}_{p^2}$ with $i^2 = -1$.  Write $\pi_p$ for the $p$-power Frobenius endomorphism: $\pi_p(x,y):= (x^p, y^p)$. 
In $B_{p,\infty}:= \End(E_{1728})\otimes_\mathbb{Z} \mathbb{Q}$, we have:
\begin{equation}\label{eq:1728EndoRing}
    \End(E_{1728}) = \left\langle 1, [i], \frac{1 + \pi_p}{2}, \frac{[i] + [i]\pi_p}{2} \right\rangle.
\end{equation}

The automorphism group for $E_0$ is $\Aut(E_0) = \{[\pm1], [\pm\zeta_3],[\pm\zeta_3^2]\}$, where $[\zeta_3](x,y) := (\zeta_3x, y)$ with $\zeta_3$ a root of $T^2 + T + 1$ in $\mathbb{F}_{p^2}$.  We let $[\sqrt{-3}] = 2 [\zeta_3] + 1$.  In $B_{p,\infty}:= \End(E_0)\otimes_\mathbb{Z} \mathbb{Q}$, we have
\begin{equation}\label{eq:0EndoRing}
    \End(E_0) = \left\langle 1, \frac{-1 + [\sqrt{-3}]}{2}, \pi_p, \frac{3 + [\sqrt{-3}] + 3\pi_p + [\sqrt{-3}]\pi_p}{6} \right\rangle.
\end{equation}

\begin{lemma}\label{lem:normorthogonal}
Let $E \in \{E_0, E_{1728}\}$ and let $\varphi \in \Aut(E)$ such that $\varphi \ne \pm [1]$. If $\overline{\varphi} \theta= \theta \varphi$ for $\theta \in \End(E)$, then $\deg\theta \ge p$.
\end{lemma}

\begin{proof}
Let $\eta = [i]$ or $\eta =[\sqrt{-3}]$ depending whether $E = E_{1728}$ or $E=E_0$, respectively. The condition $\overline{\varphi} \theta= \theta \varphi$ implies that $\theta$ must have the form  $\theta= s \pi_p + t (\eta \pi_p)$ where $s, t \in \QQ$.  

For $E=E_{1728}$, using the $\ZZ$-basis for $\End(E_{1728})$ in \eqref{eq:1728EndoRing}, we have $\theta= m \pi_p + n ([i] \pi_p)$ where $m, n \in \ZZ$. Thus $\deg\theta= p(m^2 + n^2) \ge p$. 

For $E=E_0$, using the $\ZZ$-basis for $\End(E_0)$ in \eqref{eq:0EndoRing}, we have $\theta= \frac{m}{2}\pi_p - \frac{n}{2} (\eta \pi_p)$ where $m, n \in \ZZ$ are of the same parity. Thus $\deg\theta= \frac{p}{4}(m^2 + 3n^2)$.  Since $m^2 + 3n^2 \equiv 0 \pmod 4$, it follows that $\deg\theta \ge p$.
 
\end{proof}

The graph $\mathcal{G}_\ell$ has regular out-degree $\ell+1$ (the number of subgroups of $E[\ell]$ of size $\ell$) \cite{Pizer_RamGraphsHecke}.  It is very nearly an undirected graph, obtained by identifying each $\ell$-isogeny with its dual.  The exception is that since isogenies are identified under post-composition with an automorphism, at curves with extra automorphisms, this identification may fail (see Figure~\ref{fig:0121}).  Nevertheless, this issue occurs for edges adjacent to at most two vertices ($j=0, 1728$), so one often considers it an $\ell+1$ regular undirected graph for the purposes of spectral graph theory, without affecting the validity of the conclusions.  Modulo this slight lie, or when $j=0$ and $j=1728$ are ordinary and hence absent from the graph (they are absent if and only if $p \equiv 1 \pmod{12})$, $\mathcal{G}_\ell$ is an $(\ell+1)$-regular Ramanujan graph.  In particular, if its 
adjacency matrix has eigenvalues $\ell+1 = \lambda_1 \ge \lambda_2 \ge \cdots \ge \lambda_n$,  then
\begin{equation}
\label{eqn:spectralgap}
\max_{i > 1} |\lambda_i| \le 2 \sqrt{\ell}.
\end{equation}

\subsection{The oriented supersingular $\ell$-isogeny graph $\mathcal{G}_{K,\ell}$}

Let $\ell<p$ be distinct primes. Fix a supersingular elliptic curve $E/\Fpbar$.  Then $\End^0(E) \cong B_{p,\infty}$, where $B_{p,\infty}$ denotes the definite quaternion algebra ramified precisely at $p$ and $\infty$ (unique up to isomorphism). Let $K$ be an imaginary quadratic field in which $p$ does not split. 
This condition on $p$ gives the existence of an embedding of $K$ into $B_{p,\infty}$. Let $\mathcal{O}$ denote an order in~$K$ and $h_{\mathcal{O}}$ the class number of this order.  Let $\Delta_\mathcal{O}$ denote its discriminant.  An order $\mathcal{O}$ of $K$ is \emph{$\ell$-fundamental} if its conductor is relatively prime to $\ell$. 

\begin{definition}[$K$-orientation]
A \emph{$K$-orientation} of $E$ is an embedding $\iota \colon K\to \End^0(E)$.
\end{definition}

\begin{definition}[$\mathcal{O}$-orientation]
A $K$-orientation $\iota \colon K\to \End^0(E)$ is an \emph{$\mathcal{O}$-orientation} if $\iota(\mathcal{O})\subseteq\End(E)$. Furthermore, we say that an $\mathcal{O}$-orientation is \emph{$\mathcal{O}$-primitive} if $\iota(\mathcal{O}) = \End(E)\cap \iota(K)$.
\end{definition}

Every $K$-orientation $\iota \colon K\to\End^0(E)$ is $\mathcal{O}$-primitive for a unique order $\mathcal{O}$ of $K$. 

In the literature, for example \cite{BonehLove}, one will find the closely related term \emph{optimal embedding}, in the context of a maximal order $\mathfrak{O}$ of a quaternion algebra $A$, for an embedding $\iota \colon K \rightarrow A$ such that $\iota(K) \cap \mathfrak{O} = \mathcal{O}$.  An optimal embedding corresponds to a primitive orientation via the Deuring correspondence.   For the theory of optimal embeddings, see \cite[Chapter 30]{voight}.

\begin{definition}[$K$-oriented elliptic curve]
A \emph{$K$-oriented elliptic curve} is a pair $(E,\iota)$, where $E$ is an elliptic curve and $\iota$ is a $K$-orientation of $E$.
\end{definition}
Isogenies of elliptic curves induce isogenies of oriented elliptic curves in the following sense: if $\phi\colon E\to E'$ is an isogeny and $\iota \colon K\to\End(E)$ is a $K$-orientation of $E$, then we define the unique induced orientation $\phi_*\iota$ on $E'$ for any $\alpha \in K$ as follows:
\begin{equation}\label{eq:induced_orientation}
    (\phi_*\iota)(\alpha) := \frac{1}{[\deg\phi]}\phi\circ\iota(\alpha)\circ \widehat{\phi}.
\end{equation}

\begin{definition}[$K$-oriented isogeny]
If $\phi_* \iota = \iota'$,  then we say that $\phi \colon (E,\iota) \rightarrow (E',\iota')$ is an isogeny of $K$-oriented curves, also known as a \emph{$K$-oriented isogeny}. 
\end{definition}

Two $K$-oriented curves $(E,\iota),(E',\iota')$ are \emph{$K$-isomorphic} if there exists an isomorphism $\eta \colon E\to E'$ such that $\eta_*\iota = \iota'$.  Two $K$-oriented isogenies $\phi\colon (E_1,\iota_1) \rightarrow (E_2,\iota_2)$ and $\phi'\colon (E_1', \iota_1') \rightarrow (E_2', \iota_2')$ are considered \emph{equivalent} if there are $K$-oriented isomorphisms $\eta$ and $\eta'$ so that the following diagram commutes:
\begin{equation}
    \label{eqn:equivalent-isogenies}
\xymatrix{
(E_1, \iota_1) \ar[r]^\phi \ar[d]_\eta & (E_2, \iota_2) \ar[d]_{\eta'} \\
(E_1', \iota_1') \ar[r]^{\phi'} & (E_2', \iota_2')
}
\end{equation}
We will frequently abuse notation by writing a representative $(E, \iota)$ or $\phi$ in place of its equivalence class.  

An automorphism $\varphi$ of $E$ is a \emph{$K$-automorphism} of $(E,\iota)$ if $\varphi_* \iota = \iota$.  This occurs if and only if $\varphi \in \iota(K)$. 

\begin{definition}[$\mathcal{G}_{K, \ell}$]
The \emph{$\ell$-isogeny graph $\mathcal{G}_{K, \ell}$ of $K$-oriented supersingular elliptic curves over $\Fpbar$} is the directed graph whose vertices $(E,\iota)$ are $K$-isomorphism classes of $K$-oriented supersingular elliptic curves over~$\Fpbar$ and whose edges are equivalence classes of $K$-isogenies of degree $\ell$ between the given oriented curves. 
\end{definition}

\begin{remark}
\label{prop:gkl-auto}
The only vertices of $\mathcal{G}_{K,\ell}$ with automorphism group larger than $\{ [\pm 1] \}$ are $(E, \iota)$ where either \begin{enumerate*}[label=(\roman *)] \item $E$ has $j$-invariant $1728$ and $\iota(i) \in \End(E)$, or
\item $E$ has $j$-invariant $0$ and $\iota(\zeta_3) \in \End(E)$.
\end{enumerate*}
\end{remark}

\begin{proposition}
\label{prop:undirected}
The directed edges of $\mathcal{G}_{K,\ell}$ can be put into equivalence classes by making an isogeny equivalent to its dual isogeny.  These classes are of size $1$ or $2$.  The only self-dual edges are loops.  By identifying isogenies with their duals, this graph can be drawn undirected.
\end{proposition}

\begin{proof}
With reference to \eqref{eqn:equivalent-isogenies}, if $\phi$ and $\phi'$ are equivalent, then their duals are equivalent.  Therefore an equivalence class consists of one directed edge and its dual, or a single self-dual edge.
\end{proof}

Although the number of  $\ell$-isogenies originating at a supersingular curve is  always $\ell+1$, some of these may be equivalent, so that in the graph $\mathcal{G}_{K,\ell}$, the out-degree may be strictly smaller.  As usual, the culprits are the extra automorphisms, as in Remark~\ref{prop:gkl-auto}. Before we state and prove the exact out-degree, we need a lemma.

\begin{lemma}\label{lem:loopsandauts}
Let $E$ be a curve with extra automorphisms. Given a non-zero isogeny $\phi\colon E \rightarrow E'$, its kernel  $\ker \phi$ is fixed by $\Aut(E)\setminus\{[\pm1]\}$ if and only if $E \cong E'$. If we further assume that $\deg\phi < p$ then $\ker \phi$ is fixed by $\Aut(E)\setminus\{[\pm1]\}$ if and only if $E \cong E'$ and $\phi$ (up to post-composition with an isomorphism $E'\to E$) is an element in $\mathbb{Q}(\Aut(E))\cap\End(E)$. 
\end{lemma}

\begin{proof}
Let $\varphi$ be a non-integer automorphism of $E$.
First, suppose $\ker\phi = \varphi\ker\phi$. Then $\ker\phi = \ker\phi\varphi$, so there exists a unique $\eta\in \Aut(E')$ such that $\phi\varphi = \eta\phi$ by \cite[III.4.11]{silverman2009arithmetic}. Then, working in $\End(E)\otimes_\mathbb{Z}\mathbb{Q} \cong \End(E')\otimes_\mathbb{Z}\mathbb{Q}$, the element $\phi\varphi\frac{\widehat{\phi}}{[\deg(\phi)]} = \eta\in \Aut(E')\subset\End(E')$ has the same minimal polynomial as $\varphi\in\Aut(E)\subset\End(E)$. The existence of non-integer elements in $\Aut(E)$ and $\Aut(E')$ with the same minimal polynomial gives $\Aut(E)\cong \Aut(E')$. Since $\Aut(E)\cong \Aut(E')\supsetneq\{[\pm1]\}$, the automorphism group determines the isomorphism class of the curve and we have $E\cong E'$. 

Next, suppose $E\cong E'$ via an isomorphism $\lambda\colon E'\to E$. In particular, $\Aut(E) \cong \Aut(E')\supsetneq\{[\pm1]\}$. Then, $\lambda\phi\in\End(E)$ and $\ker(\lambda\phi) = \ker\phi$. By abuse of notation, we replace $\lambda\phi$ with $\phi$ and may henceforth assume $\phi\in\End(E)$ in this case. The endomorphism $\phi$ has the same minimal polynomial as $\varphi^{-1}\phi\varphi$, and so up to some automorphism $\eta\in\Aut(E)$, we have $\varphi^{-1}\phi\varphi = \eta\phi$, and thus $\ker\phi = \ker\phi\varphi$. 

Now, assume $\deg\phi<p$ and suppose $E\cong E'$ and $\ker \phi$ is fixed by a generator $\varphi$ of $\Aut(E)$. Then $\phi  = \varphi^k \phi \varphi$ for some integer $0 \le k < |\Aut(E)|$. By cancellation, we know $\phi \neq \pm \phi \varphi$ since $\varphi\neq[\pm1]$,
hence $\overline{\varphi} \phi = \phi \varphi$ or else $\varphi \phi = \phi \varphi$.  For $\varphi = [\zeta_6]$ (a primitive sixth root of unity), we have $k \in \{1,2,4,5\}$.  If $k \in \{1,5\}$, then either $\overline{[\zeta_6]} \phi = \phi [\zeta_6]$ or $[\zeta_6] \phi = \phi [\zeta_6]$.  Otherwise, if $k \in \{2,4\}$, then observe that $\phi = [\zeta_6]^k \phi [\zeta_6] = [\zeta_6]^k( [\zeta_6]^k \phi [\zeta_6]) [\zeta_6] = [\zeta_6]^{2k} \phi [\zeta_6]^2 = [\zeta_3]^k \phi [\zeta_3]$, so $\overline{[\zeta_3]} \phi = \phi [\zeta_3]$ or $[\zeta_3] \phi = \phi [\zeta_3]$.  Since $\deg\phi<p$, Lemma~\ref{lem:normorthogonal} rules out the case that $\overline{\varphi} \phi = \phi \varphi$.  We are left to conclude that $\phi$ commutes with an element $\varphi$ of $\Aut(E) \setminus\{[\pm 1]\}$, which implies that $\phi \in \QQ(\varphi) \subseteq \End^0(E)$. 

On the other hand, if $\ker \phi \neq \ker \phi \varphi$, then $\phi \neq \pm \varphi^k \phi \varphi$ for any $k$, so $\phi$ does not commute with $\varphi$, and hence $\phi \notin \QQ(\Aut(E))$.  
\end{proof}

\begin{definition}[Reduced Automorphism Group]
Let $\Aut'(E): = \Aut(E)/[\pm 1]$ denote the reduced automorphism group of the elliptic curve $E$. 
\end{definition}

\begin{proposition}
\label{prop:ellplusone}
Let $(E,\iota)$ be an oriented curve in $\mathcal{G}_{K,\ell}$ where $\iota$ is an $\mathcal{O}$-orientation.  Then $(E,\iota)$ has out-degree $\ell+1$, except at the oriented curves with extra $K$-automorphisms (those of Remark~\ref{prop:gkl-auto}).  In that exceptional case, the out-degree is $(\ell+1 - s_\ell)/|\Aut'(E)| + s_\ell$ 
where $s_\ell$ is the number of $ \phi \in \End(E) \cap \iota(K)$ of degree $\ell$ up to post-composition by $\Aut(E)$.  In particular, the out-degree at $(E, \iota)$, when $(E,\iota)$ has extra $K$-automorphisms, is 
\[
\left\{ 
\begin{array}{ll}
(\ell-1)/2 + 2 & \mathcal{O} = \ZZ[i] \text{ and $\ell$ splits in $\mathcal{O}$ ($s_\ell=2$) } \\
(\ell+1)/2 & \mathcal{O} = \ZZ[i] \text{ and $\ell$ is inert in $\mathcal{O}$ ($s_\ell=0$) } \\
2 & \mathcal{O} = \ZZ[i] \text{ and $\ell=2$ ($s_\ell=1$) } \\
2(\ell-1)/3 + 2 & \mathcal{O} = \ZZ[\zeta_3] \text{ and $\ell$ splits in $\mathcal{O}$ ($s_\ell=2$) } \\
2(\ell+1)/3 & \mathcal{O} = \ZZ[\zeta_3] \text{ and $\ell$ is inert in $\mathcal{O}$ ($s_\ell=0$) } \\
3 & \mathcal{O} = \ZZ[\zeta_3] \text{ and $\ell=3$ ($s_\ell=1$) } \\
\end{array}
\right.
\]
\end{proposition}

\begin{proof}

With reference to diagram \eqref{eqn:equivalent-isogenies}, by taking $\eta\in\Aut(E)$, the two equivalent $K$-oriented isogenies $\phi$ and $\phi'$ from $(E, \iota)$ must satisfy $\iota' = \eta_* \iota = \iota$ and $\eta \ker \phi = \ker \phi'$.  The first equality implies $\eta \in \iota(K)$.  Hence $(E,\iota)$ has out-degree $(\ell+1)$ if $\Aut(E) \cap \iota(K) = [\pm 1]$. If $\Aut(E) \cap \iota(K) \supsetneq [\pm 1]$, the out-degree may be less than $(\ell + 1)$, and we investigate this case below.

If $\Aut(E) \cap \iota(K)$ contains non-trivial automorphisms, then it contains all of them, and $\End(E) \cap \iota(K) = \iota(\mathcal{O}_K) \supseteq \Aut(E)$.  The  $\ell+1$ outgoing $\ell$-isogenies fall into equivalence classes according to whether their kernels are related by an extra automorphism.  The kernel of an outgoing $\ell$-isogeny $\phi$ from $E$ is fixed by $\eta\in \Aut(E)$ if and only if  $\phi \in \End(E) \cap \iota(K)$ (as in the proof of Lemma~\ref{lem:loopsandauts}, we may  assume that $\phi \in \End(E)$ by post-composing with an isomorphism; we have assumed throughout the paper that $\ell < p$). Let $S_\ell$ be the set of equivalence classes of elements in $\End(E) \cap \iota(K)$ of degree $\ell$ up to pre-composition by $\Aut(E)$, and put $s_\ell = \#_\ell$. 
Note that $\End(E) \cap \iota(K) = \iota(\mathcal{O}_K)$, so $s_\ell = 0$, $1$ or $2$, depending on whether $\ell$ is inert, ramified or split in $\mathcal{O}_K$. By the above, a degree $\ell$-isogeny $\phi$ is in $S_\ell$ if and only if it is fixed by all $\eta \in \Aut(E)$.  Otherwise its kernel has an orbit of size $|\Aut'(E)|$.  Thus the number of out-isogenies is $(\ell+1 - s_\ell)/|\Aut'(E)| + s_\ell$. 
\end{proof}

\begin{remark}
As an immediate corollary, we obtain a new method of proof and generalization of a known result about the number of loops at a vertex \cite{AAM, YZ, LOX2}.  Specifically, the previous works concerned only large~$p$ (the last two sentences of the next corollary). 
\begin{corollary}
\label{cor:loops}
Let $j \in \{0, 1728\}$ correspond to a supersingular curve over $\overline{\FF}_{p}$.  Let $K = \QQ(\zeta_3), \QQ(i)$ for $j=0$, $1728$ respectively.  Let $n$ represent the number of loops at the vertex of $j$-invariant $j$, in $\mathcal{G}_\ell$.  Then
\[
n = 1+\left( \frac{\Delta_K}{\ell} \right) +  k|\Aut'(E)|,
\]
for some $k \in \ZZ$.
If $4\ell < p$, then the number of loops at $j=1728$ is exactly $1+\left( \frac{\Delta_K}{\ell} \right)$.  If $3\ell < p$, then the number of loops at $j=0$ is exactly $1 + \left( \frac{\Delta_K}{\ell} \right)$.
\end{corollary}

\begin{proof}
Let $\mathcal{O}$ be $\ZZ[\zeta_3]$ or $\ZZ[i]$ when $j=0$ or $1728$, respectively.  Whether $\ell$ splits or ramifies or is inert determines the number of elements of $\iota(\mathcal{O})$ of degree $\ell$, up to automorphisms:  two, one or zero, respectively.  These elements are all loops on the vertex.  All other loops fall into groups of size $|\Aut'(E)|$, as in the proof of Proposition~\ref{prop:ellplusone}.  For the last two sentences of the statement, we revisit the endomorphism rings of the vertex (Section~\ref{sec:back:ss}), to see that the smallest elements not in $\iota(\mathcal{O})$ have norm $p/4$ and $p/3$ respectively.
\end{proof}

It is possible to compute $k$ explicitly for any given value of $\ell$ in terms of $p$, by reference to the endomorphism rings as described in the proof.
\end{remark}

The graph $\mathcal{G}_{K, \ell}$ consists of connected components, each with a volcano structure as pictured in Figure~\ref{fig:oriented_volcano}. The rim vertices of any given volcano are $K$-orientations of supersingular elliptic curves which are $\mathcal{O}$-primitive for some fixed $\ell$-fundamental order $\mathcal{O}$ of $K$.  The oriented supersingular elliptic curves which appear at depths below the rim are primitively oriented by orders whose index in $\mathcal{O}$ is given by increasing powers of $\ell$: at depth one from the rim, the curves are primitively $(\mathbb{Z} + \ell\mathcal{O})$-oriented, and in general, at depth $k$, the curves are primitively $(\mathbb{Z}+ \ell^k\mathcal{O})$-oriented. 
The edges of the volcano can thus be classified according to whether they increase, decrease or preserve the depth as we move from domain to codomain.

\begin{definition}[Horizontal, descending and ascending isogenies]
Let $\phi:(E,\iota)\to (E',\iota')$ be a $K$-oriented degree $\ell$ isogeny. If $\iota$ is a primitive $\mathcal{O}$-orientation and $\iota'$ is a primitive $\mathcal{O}'$-orientation, then $\phi$ is one of the following three types of isogeny:
\begin{enumerate}
    \item $\phi$ is \textit{horizontal} if $\mathcal{O} = \mathcal{O}'$,
    \item $\phi$ is \textit{descending} if $\mathcal{O}\supsetneq\mathcal{O}'$,
    \item $\phi$ is \textit{ascending} if $\mathcal{O}\subsetneq\mathcal{O}'$.
\end{enumerate}
\end{definition}

The fact that each component has a volcano structure is the content of the following proposition; versions of this appear in various places in the literature \cite{colo2019orienting, onuki2021}.  

\begin{proposition}
\label{prop:volcanostructure}
Let $(E,\iota)$ be such that $\iota$ is $\mathcal{O}$-primitive.
If $\mathcal{O}$ is $\ell$-fundamental, then $(E,\iota)$ has no ascending $\ell$-isogeny and  $\left( \frac{\Delta_\mathcal{O}}{\ell} \right)+1$ horizontal $\ell$-isogenies.  Otherwise, $(E,\iota)$ has one ascending $\ell$-isogeny and no horizontal isogenies.  All other isogenies from $(E,\iota)$ are descending.
\end{proposition}

\begin{proof}
 This follows from \cite[Proposition 4.8]{paperone}.  Write $\mathcal{O} = \ZZ[\theta]$ satisfying the hypotheses of that proposition.  If $\mathcal{O}$ is $\ell$-fundamental, then the norm of $\theta$ is not divisible by $\ell$.  Then $\theta$ acts to permute $E[\ell]$ and therefore has two non-zero eigenvalues. By \cite[Proposition~4.8]{paperone}, there are two horizontal isogenies if and only if these eigenvalues are distinct and defined over $\FF_\ell$, i.e.\ $\ell$ splits.  If they are not distinct but defined over $\FF_\ell$, there is one horizontal isogeny, i.e.\ $\ell$ ramifies.  Otherwise there are no horizontal isogenies. If $\mathcal{O}$ is not $\ell$-fundamental, then the norm of $\theta$ is divisible by $\ell$ and there is a zero eigenvalue (but only one, since $\ell$ does not divide $\theta$).  By \cite[Proposition 4.8]{paperone}, this gives exactly one ascending isogeny.  All other isogenies are descending.
\end{proof}

Let $\mathcal{O}$ be an order in $K$. Every  primitive  $\mathcal{O}$-oriented curve  $(E,\iota)$ is on a  volcano in  $\mathcal{G}_{K, \ell}$. If $\mathcal{O}$ is $\ell$-fundamental, then $(E,\iota)$ is on the \emph{rim} of this volcano. If $\ell$  divides the conductor of $\mathcal{O}$, then the power of $\ell$ dividing the discriminant of $\mathcal{O}$ is referred to as the \emph{depth} of the curve on this volcano.

 \begin{proposition}
 A connected component of $\mathcal{G}_{K,\ell}$, when identifying $\ell$-isogenies with their duals, has at most one cycle, passing through the rim vertices.
 \end{proposition}
 
 \begin{proof}
 The fact that there is at most one ascending isogeny from any vertex below the rim implies that there are no cycles below the rim.  Restricting to the rim, we have a $2$-regular graph, hence a union of cycle graphs.  But since the graph is connected by assumption, it consists of at most one cycle.
 \end{proof}
 
 Therefore, we also use the term \emph{rim} to refer to the undirected cycle graph that one obtains as a subgraph of $\mathcal{G}_{K,\ell}$ by restricting to the rim vertices and identifying $\ell$-isogenies with their duals.  We call a rim \emph{directed} if we choose a direction of traversal (disallowing backtracking).

\begin{remark}
\label{remark:ssgraph}
We pause briefly to justify the choice of equivalence on isogenies for $\mathcal{G}_{K,\ell}$ (which differs from the notion of equivalence in $\mathcal{G}_\ell$).  In fact, when forming a graph whose vertices are isomorphism classes of any type of object, the graph is well-defined if we choose the edges to be maps up to pre- and post-composition by isomorphism, as in \eqref{eqn:equivalent-isogenies}.  In other words, it is reasonable to require of our graph definition that, under any two sets of representatives for the vertices, and any set of object-wise isomorphisms between those sets of representatives, the resulting graphs should be canonically isomorphic.  This requirement, when unravelled, is essentially the requirement than the maps be taken up to pre- and post-composition.  Historically, the supersingular isogeny graph is not defined in this robust way, instead identifying isogenies as equivalent merely up to post-composition.  This allows for the computational convenience of identifying kernels with isogenies, and allows for an identification between non-backtracking walks and cyclic isogenies.  However, it results in certain non-canonical behaviours.  These behaviours can be observed only in proximity to the curves with extra automorphisms (having $j$-invariant $j=0$ and $j=1728$), particularly the phenomenon described in the second paragraph of Section~\ref{sec:arb}.  The historical definition of the supersingular isogeny graph resulted in substantial difficulties arising from extra automorphisms in proving our main Theorem~\ref{thm:mainbij-nobase}.
\end{remark}

\subsection{The sets $\SSO$ and $\rho(\Ell(\mathcal{O}))$ and the class group action}
\label{sec:back:sso}
Having stratified each volcano into depths where the orientations are all primitive with respect to a fixed order, it is natural to consider the complete set of curves with such orientations.

\begin{definition}[$\SSO$]
Let $\SSO$ denote the set of primitively $\mathcal{O}$-oriented isomorphism classes of  supersingular elliptic curves, up to $K$-isomorphism. 
\end{definition}

Here we recall the following conditions for the set $\SSO$ to be non-empty.

\begin{proposition}[{\cite[Proposition 3.2]{onuki2021}}]
\label{prop:sso-empty}
The set $\SSO$ is not empty if and only if $p$ does not split in $K$ and does not divide the conductor of $\mathcal{O}$.
\end{proposition}

If $\mathcal{O}$ is $\ell$-fundamental, then $\SSO$ is a union of volcano rims in $\mathcal{G}_{K,\ell}$.

For the proof of our bijection in Section~\ref{sec:proofbij}, we invoke the Deuring lifting theorem as outlined in \cite[Section 3.2]{onuki2021}. For this, we will need a number of definitions and lemmas about CM elliptic curves over number fields.

There is a number field $L'$ (an extension of the ring class field $L$ of $\mathcal{O}$) and a prime ideal $\mathfrak{p}$ above $p$ in $\mathcal{O}_{L'}$ such that every elliptic curve with complex multiplication by $\mathcal{O}$ has a representative over $L'$ with good reduction at $\mathfrak{p}$ \cite[Section VII.5]{silverman2009arithmetic}.  Fix such a choice.

\begin{definition}[$\Ell(\mathcal{O})$]
Let $\Ell(\mathcal{O})$ denote the set of isomorphism classes of elliptic curves over $L'$ with endomorphism ring isomorphic to $\mathcal{O}$ and good reduction at $\mathfrak{p}$.
\end{definition}

By the theory of complex multiplication, $\# \Ell(\mathcal{O}) = h_\mathcal{O}$.  Write $[ \cdot ]_E \colon \mathcal{O} \rightarrow \End(E)$ for the isomorphism which is normalized so that $([ \alpha ]_E)^* \omega_E = \alpha \omega_E$, where $\omega_E$ denotes the invariant differential of $E$ (\cite[Section 2.3]{onuki2021}, \cite[ Prop.\ II.1.1]{SilvermanAdv}).  There is a map given by reduction modulo $\mathfrak{p}$:
\[
\rho\colon \Ell(\mathcal{O}) \rightarrow \SSO, \quad
E \mapsto (\tilde{E}, \iota_E),
\]
where $\iota_E\colon K \rightarrow \End^0(\tilde E)$ is determined by its restriction $\iota_E \colon \mathcal{O} \rightarrow \End(\tilde E)$ being the reduction modulo~$\mathfrak{p}$ of the 
isomorphism $[ \cdot ]_E$ (in other words, $\iota_E(\alpha) = [\alpha]_E \pmod{\mathfrak{p}}$ for all $\alpha \in \mathcal{O}$).  In particular, from the normalization, it holds that $\iota_E(\alpha)^* \omega_{\tilde{E}} = \alpha \omega_{\tilde E}$ by reduction modulo $\mathfrak{p}$.  

The map $\rho$ is injective on $\Ell(\mathcal{O})$, so that its image $\rho(\Ell(\mathcal{O}))$ is a subset of $\SSO$ of size $h_\mathcal{O}$.   
There is an action of the $p$-power Frobenius on $\SSO$, given by $\pi_p\cdot(E,\iota)\mapsto (E^{(p)},\iota^{(p)})$, where $E^{(p)} = \pi_p(E)$ and $\iota^{(p)} = (\pi_p)_*\iota$.

\begin{proposition}[{\cite[Proposition 3.3]{onuki2021}}]
For all $(E,\iota)\in \SSO$, at least one of $(E,\iota)$, $(E^{(p)},\iota^{(p)})$ belongs to $\rho(\Ell(\mathcal{O}))$.
\end{proposition}

As suggested by the above, surjectivity of $\rho$ may indeed fail:  $\#\SSO \in \{ \hO, 2h_\mathcal{O} \}$.  Theorem~\ref{thm:sso} describes when this failure occurs. 

We define an action of ideals of $\mathcal{O}$ on oriented elliptic curves $(E,\iota)\in \SSO$. 

\begin{definition}[Ideal action on oriented elliptic curves] \label{def:idealaction}
Let $(E,\iota)\in\SSO$ for some order $\mathcal{O}$ of $K$. Let $\mathfrak{a}$ be an integral ideal of $\mathcal{O}$ coprime to $p$. Define the intersection:
\[E[\iota(\mathfrak{a})] := \bigcap_{\alpha\in\mathfrak{a}}\ker(\iota(\alpha)).\]
This group defines an isogeny $\phi_\mathfrak{a}^{(E,\iota)}:E\to E/E[\iota(\mathfrak{a})]$. The action of $\mathfrak{a}$ on $(E,\iota)$ is defined as $\mathfrak{a}*(E,\iota):= (\phi_\mathfrak{a}^{(E,\iota)}(E),(\phi_\mathfrak{a}^{(E,\iota)})_*\iota)$.
\end{definition}
When clear from context, we will drop the superscript and just write $\phi_\mathfrak{a}$.
Note that to define the action of the ideal $\mathfrak{a}$, we require its norm  to be coprime to $p$ (see the proof of \cite[Proposition 3.6]{onuki2021}).

\begin{theorem}[{\cite[Theorem 3.4]{onuki2021}}] Assume that $p$ does not split in $K$ and $p$ does not divide the conductor of $\mathcal{O}$. Then the ideal action of Definition~\ref{def:idealaction} defines a free and transitive action of $\Cl(\mathcal{O})$ on $\rho(\Ell(\mathcal{O}))$.
\end{theorem}

This action is the `same' as that of the CM theory, in the sense that $\rho( \sigma \cdot (E,\iota)) = \sigma \cdot (\rho( (E,\iota)))$ for $\sigma \in \Cl(\mathcal{O})$ \cite[Proof of Proposition 3.6]{onuki2021}. 

Finally, we relate $\SSO$ to the volcanoes of $\mathcal{G}_{K,\ell}$.

\begin{definition}[$\mathcal{O}$-cordillera]
A collection of volcanoes in $\mathcal{G}_{K,\ell}$.  The collection of volcanoes whose rims are primitively oriented by $\mathcal{O}$ is called the \emph{$\mathcal{O}$-cordillera}.
\end{definition}
The set of vertices at the rim of the $\mathcal{O}$-cordillera is exactly the set $\SSO$.

Next, we slightly strengthen results of Col\`o-Kohel \cite{colo2019orienting} and Onuki \cite{onuki2021} to give an action on oriented isogenies by a direct product of the class group with Frobenius.

\begin{definition}[Action of Frobenius on oriented elliptic curves]\label{def:frob_action}
The two-element group $\langle \pi_p \rangle = \{1,\pi_p\}$ generated by the Frobenius automorphism $\pi_p$ of $\mathbb{F}_{p^2}$ acts on $\mathcal{G}_{K,\ell}$ by 
\[
\pi_p \cdot (E, \iota) = (E^{(p)}, \iota^{(p)} ), \quad
\pi_p \cdot \varphi = \varphi^{(p)},
\]
where $\iota^{(p)} := (\pi_p)_*(\iota)$.
\end{definition}
For any isogeny $\varphi$, we have $\pi_p \circ \varphi (x,y) = \varphi^{(p)}(x^p,y^p) = \varphi^{(p)} \circ \pi_p (x,y)$.  Hence, one has \[
\iota^{(p)}(\alpha) = (\pi_p)_*(\iota)(\alpha) = \frac{1}{p}\pi_p \circ \iota(\alpha) \circ \widehat{\pi_p} = \frac{1}{p} \iota(\alpha)^{(p)} \circ \pi_p \circ \widehat{\pi_p} = \iota(\alpha)^{(p)}.
\]
Since $\varphi \mapsto \varphi^{(p)}$ gives an isomorphism $\End(E) \cong \End(E^{(p)})$, this yields an action on $\SSO$ by $\langle \pi_p \rangle$.  In fact, it is an action on the graph $\mathcal{G}_{K,\ell}$, i.e.\ it preserves adjacency.

\begin{proposition}
\label{prop:cl-frob}
The actions described in 
Definitions~\ref{def:idealaction} and \ref{def:frob_action} above commute and hence give an action of $\Cl(\mathcal{O}) \times \langle \pi_p \rangle$ on $\SSO$. This action is transitive and its point stabilizers are either all trivial or all $\langle \pi_p \rangle$.  In particular,
$\#\SSO \in \{ \hO, 2\hO \}$.
\end{proposition}

\begin{proof}

	We have $\pi_p \cdot \varphi_\mathfrak{a} \cdot (E, \iota) = (\varphi_\mathfrak{a})^{(p)} \cdot \pi_p \cdot (E, \iota)$.
To avoid confusion we momentarily use the more specific notation $\varphi_\mathfrak{a}^{(E,\iota)}$ to denote the isogeny $\varphi_\mathfrak{a}$ with domain $(E,\iota)$ and kernel $E[\iota(\mathfrak{a})]$. Then
\begin{align*}
	\ker( (\varphi_\mathfrak{a}^{(E,\iota)})^{(p)} ) &= \ker( \varphi_\mathfrak{a}^{(E,\iota)} )^{(p)} = E[\iota(\mathfrak{a})]^{(p)} = \bigcap_{\theta \in \iota(\mathfrak{a})} \ker(\theta)^{(p)} \notag \\ &= \bigcap_{\theta \in \iota(\mathfrak{a})} \ker(\theta^{(p)}) = \bigcap_{\theta \in \iota^{(p)}(\mathfrak{a})} \ker(\theta) = E^{(p)}[\iota^{(p)}(\mathfrak{a})] = \ker( (\varphi_\mathfrak{a}^{(E^{(p)},\iota^{(p)})}).
\end{align*}

The calculation above implies that $(\varphi_\mathfrak{a}^{(E,\iota)})^{(p)} = \varphi_{\mathfrak{a}}^{(E^{(p)},\iota^{(p)})}$.
Thus,
\begin{equation*}
	\pi_p \cdot \mathfrak{a} \cdot (E, \iota)  = \mathfrak{a} \cdot \pi_p \cdot (E,\iota).
\end{equation*}
So the action of $\ClO \times \langle \pi_p\rangle$ on $\SSO$ is well-defined.

The restriction of this action to $\ClO$ acts freely and transitively on a subset of $\SSO$ which contains at least one of $(E,\iota)$ or $(E^{(p)}, \iota^{(p)})$ \cite[Theorem 3.4]{onuki2021}, from which the rest of the statement follows.  Transitivity implies that the stabilizers are all of the same size.
\end{proof}

\begin{remark}
One might expect to see the dihedral group, not a direct product.  We will see the dihedral group coming from class field theory in Section~\ref{sec:classfieldtheory}, where we consider the action of conjugation on $\SSO$.  But in our definitions here, Frobenius acts on both curve and orientation.  In other words, $\SSO$ has an action by a direct product and by a dihedral group, both extensions of the same action of $\ClO$ by $\ZZ/2\ZZ$, but these are not necessarily the same.  For an example of the action of Frobenius and the action of conjugation being different, see Figure~\ref{fig:two cycles frob conj}.
\end{remark}

\begin{figure}
    \centering
    \includegraphics[width=0.8\textwidth]{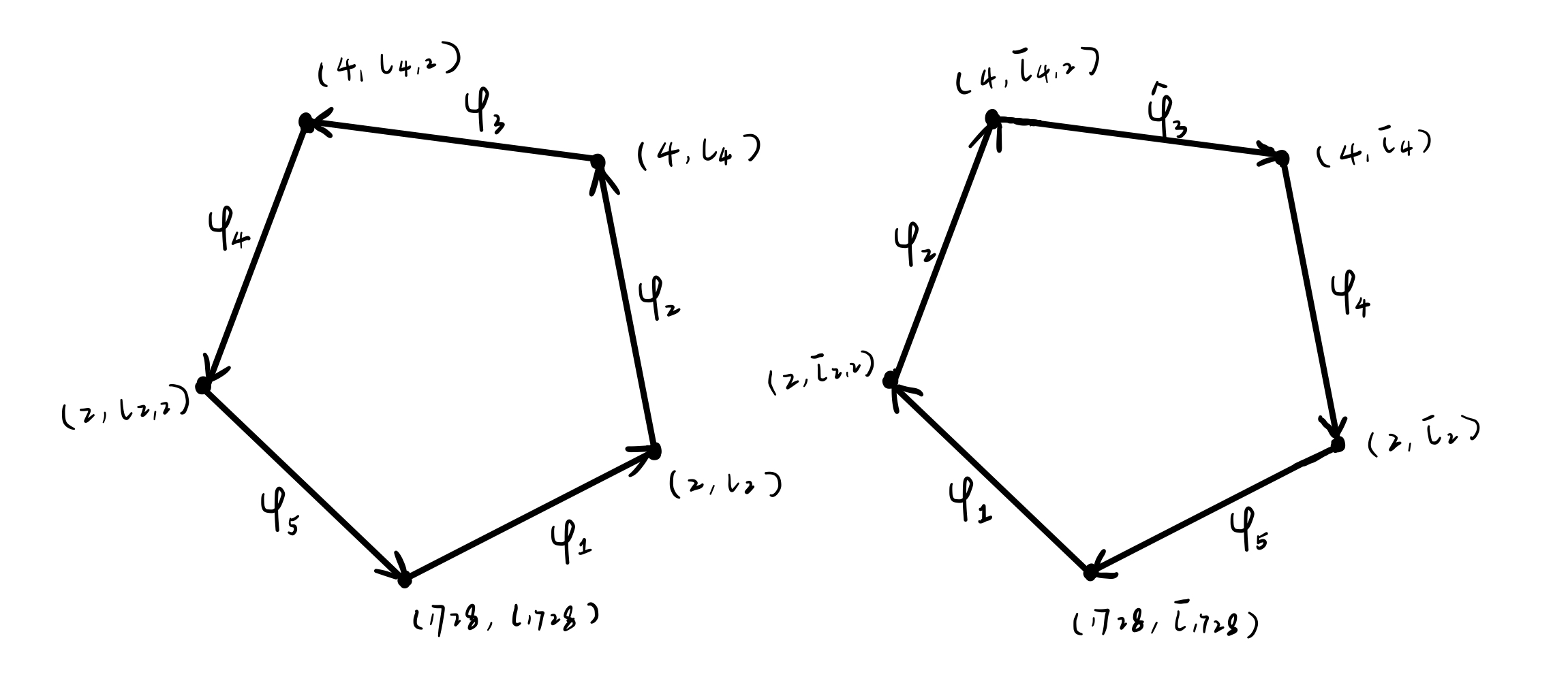}
    \caption{Here $p=31, \ell=2$ and $K= \QQ(\sqrt{-47})$. This figure shows the set $\SS_{\mathcal{O}_K}$ which consists of two cycles. We denote by $(j, \iota_j)$ the vertex with $j$-invariant $j$ and orientation $\iota_j$; if the $j$-invariant appears twice, the second orientation is denoted by $\iota_{j,2}$.  The notation $\overline{\iota}_j$ denotes the conjugate orientation.  Conjugation (Section~\ref{sec:classfieldtheory}) maps $(j, \iota_j)$ to $(j, \overline{\iota}_j)$ and  $(j, \iota_{j,2})$ to $(j, \overline{\iota}_{j,2})$. On the other hand,  $\pi_p$ maps $(1728, \iota_{1728})$ to $(1728, \overline{\iota}_{1728})$, $(2, \iota_{2})$ to $(2, \overline{\iota}_{2,2})$, and $(4, \iota_{4})$ to $(4, \overline{\iota}_{4,2})$.}
    \label{fig:two cycles frob conj}
\end{figure}

\subsection{A lemma of algebraic number theory}

We make explicit a small lemma from algebraic number theory that will be used in Section~\ref{sec:counting}.
\begin{lemma}
\label{lemma:tower-split}
Let $K$ be an imaginary quadratic field and let $\mathcal{O}$ be an order in $K$ with conductor $f$. Assume that $\ell$ is a prime that does not divide $f$. Then $\ell$ splits (ramifies, is inert, respectively) in $\mathcal{O}$ if and only if $\ell$ splits (ramifies, is inert, respectively) in the maximal order $\mathcal{O}_K$.
\end{lemma}
\begin{proof}
By \cite[Proposition 7.20]{Cox_primesoftheform}, if $\mathfrak{a}$ is an $\mathcal{O}_K$-ideal above $\ell$ prime to $f$, then $\mathfrak{a}' = \mathfrak{a} \cap \mathcal{O}$ is an $\mathcal{O}$-ideal above $\ell$ with the same norm, and if $\mathfrak{a}' $ is an $\mathcal{O}$-ideal above $\ell$ prime to $f$, then $\mathfrak{a}= \mathfrak{a}' \mathcal{O}_K$ is an $\mathcal{O}_K$-ideal above $\ell$ with the same norm. Furthermore, these maps are inverse to one another, giving a bijection between ideals of $\mathcal{O}_K$ coprime to $f$, and invertible ideals of $\mathcal{O}$.  Therefore the splitting, ramification and inertia properties of $\ell$ in $\mathcal{O}_K$ and in $\mathcal{O}$ are the same.
\end{proof}


\section{Bijection between oriented volcano rims and isogeny cycles in the supersingular $\ell$-isogeny graph}
\label{sec:bij}

\subsection{Statements of results}
\label{sec:bij-statement}
In this section, we give a bijection 
between `isogeny cycles' (closed walks satisfying a non-backtracking and non-repeating condition; see below) on the supersingular $\ell$-isogeny graph and rims (up to conjugation) of the volcanoes in all oriented supersingular $\ell$-isogeny graphs. We first give our definition of isogeny cycles.

\begin{definition}[Isogeny cycle]
\label{defn:isogeny-cycle}
An \emph{isogeny cycle} is a closed walk, forgetting basepoint, in $\mathcal{G}_\ell$ containing no backtracking (no consecutive edges compose to multiplication-by-$\ell$, Definition~\ref{defn:backtracking}) which is not a power of another closed walk (i.e.\ not equal to another closed walk repeated more than once).
\end{definition}

Isogeny cycles travelling the same vertices in opposite directions necessarily employ different directed edges and hence are considered different.

We need a few more items of terminology.  The term \emph{directed} refers to the notion of direction of traversal.  Each rim corresponds to two directed rims, since there are two possible directions of traversal.  The \emph{conjugate} of a rim $R$ is obtained by taking each vertex $(E,\iota)$ to $(E,\overline{\iota})$ and each edge $\phi\colon (E,\iota) \rightarrow (E',\iota')$ to $\phi \colon (E,\overline{\iota}) \rightarrow (E', \overline{\iota'})$.  The conjugation of a rim is again a rim.  Note that if $\phi_\mathfrak{a}^{(E,\iota)}\colon (E,\iota) \rightarrow (E', \iota')$, then $\phi_{\overline{\mathfrak{a}}}^{(E,\overline{\iota})} \colon (E,\overline{\iota}) \rightarrow (E', \overline{\iota'})$ since $E[\overline{\iota}(\overline{\mathfrak{a}})] = E[\iota(\mathfrak{a})]$.  So in particular, the direction of a rim ($E_0$, $E_1$, etc.) is preserved by conjugation, but the ideal whose action traverses the rim in that direction is conjugated.

We now state our main theorem.  
For each integer $r > 2$, define two sets:
    \[
    \mathcal{C}_r = \{ \mbox{ isogeny cycles in $\mathcal{G}_\ell$ of length $r$ } \}.
    \]
\[
    \mathcal{R}_{r} =
    \left\{ \mbox{ directed rims $R$ of size $r$ in $\mathcal{G}_{K,\ell}$: }
        \substack{\mbox{$K$ is an imaginary quadratic field }}\right \}.
\]

\begin{theorem}
\label{thm:mainbij-nobase}
Let $r > 2$.
There is a bijection between $\mathcal{C}_r$ and $\mathcal{R}_r/ \sim$, where $\sim$ denotes identifying any rim with its conjugate.  
\end{theorem}

\textbf{Sketch of proof.}  Theorem~\ref{thm:mainbij-nobase} will be proved in Section~\ref{sec:proofbij}.  At the highest level, this bijection is easy to describe:  given a rim $R \in \mathcal{R}_r$, forgetting orientations gives a cycle $C \in \mathcal{C}_r$.  Conversely, given a cycle $C \in \mathcal{C}_r$ passing through a curve $E$, composing its component isogenies in order, one obtains an endomorphism $\theta \in \End(E)$, which induces an orientation $\iota_\theta$.  Lifting $E$ to $(E,\iota_\theta)$, the isogenies of the cycle will give a rim in~$\mathcal{G}_{K, \ell}$, where $K$ is the quadratic field generated by the minimal polynomial of $\theta$.  These two maps are inverse.

However, the devil is in the details, especially the devil whose name is \emph{extra automorphisms}. In particular, if $\mathcal{G}_\ell$ has no curves with extra automorphisms, the proof is much briefer.  In what follows, we have endeavoured to clearly mark areas where we address extra automorphisms, so that at a first read, they can be skipped.  In particular, if $\mathcal{G}_\ell$ has no extra automorphisms, the reader can entirely skip Section~\ref{sec:graphcover}, and the word `safe' (Definition~\ref{defn:arb2}) can be ignored.  Section~\ref{sec:arb} should still be read, however, because its definitions and statements will be used in the proof of Theorem~\ref{thm:mainbij-nobase}.

\subsection{Extra automorphisms and `arbitrary assignment'}
\label{sec:arb}

Before we proceed, we must discuss a subtlety concerning extra automorphisms (see the related Remark~\ref{remark:ssgraph}).  In particular, the composition of the isogenies along a walk in $\mathcal{G}_\ell$ is not well-defined, because the isogenies themselves are only defined up to post-composition by an automorphism.  In the case of the automorphisms $[\pm 1]$, which commute with all the isogenies in the chain, this amounts to a single sign ambiguity.  But extra automorphisms may wreak havoc:
the endomorphism obtained by composing around a closed walk 
may not even have the same discriminant (and quadratic field) after a post-composition by an automorphism.
  
This would appear to lampoon any hope of associating endomorphisms nicely with closed walks.  However, the remedy to this issue is the following observation.
Consider two walks which differ by replacing a segment $j_1 \stackrel{\phi}{\longrightarrow} 0 \stackrel{\psi}{\longrightarrow} j_2$ with $j_1 \stackrel{[\zeta_3^k] \phi}{\longrightarrow} 0 \stackrel{\psi [\zeta_3^{-k}]}{\longrightarrow} \ j_2$, or $j_3 \stackrel{\phi}{\longrightarrow} 1728 \stackrel{\psi}{\longrightarrow} j_4$ with $j_3 \stackrel{[i] \phi} {\longrightarrow} 1728 \stackrel{\psi [-i]}{\longrightarrow} j_4$, where~$[\zeta_3]$ and~$[i]$ are extra automorphisms associated to primitive third and fourth roots of unity, respectively.  In each case, the first arrow is the same directed edge of $\mathcal{G}_\ell$, but the second arrow differs as an edge of $\mathcal{G}_\ell$.  The composition of the walk, however, is unchanged.  In fact, there are typically three distinct walks $j_1 \rightarrow 0 \rightarrow j_2$ in $\mathcal{G}_\ell$, and three distinct resulting isogenies $j_1 \rightarrow j_2$, but there is no \emph{canonical} identification between the elements of these two sets of size three.  One way to set a non-canonical identification is to consider the incoming isogeny $j_1 \rightarrow 0$ to be a particular fixed (but arbitrary) representative of its equivalence class (recall that the equivalence class consists of all the isogenies with the same kernel, i.e.\ up to post-composition by an automorphism).  
  
  \begin{definition}[Arbitrary assignment]
  \label{defn:arb}
  An \emph{arbitrary assignment} is a choice of isogeny $\phi$, up to sign, from every equivalence class of isogenies represented by a directed edge of $\mathcal{G}_\ell$.
  \end{definition}

There is no choice except when the codomain has extra automorphisms.
\textbf{In what follows, we shall assume that such a fixed arbitrary assignment has been made for every arrow entering $j=0$ or $j=1728$.}
This choice will affect the exact bijection we obtain in our main theorem, but not the fact that it is a bijection.  Figure~\ref{fig:0121} is instructive as to the subtleties that may arise.

\begin{figure}
    \centering
    \begin{tikzpicture}
\node[rectangle,draw] (E121) at (0,3) {$E_{0}$};
\node[rectangle,draw] (E0) at (0,0) {$E_{121}$};

\draw[->,violet,dashed] (E0) .. controls (-.5,1.5) .. (E121) node[pos=.5,left] {$\varphi$};
\draw[->,blue,dotted,line width=.3mm] (E0) .. controls (-1.5,1.5) .. (E121) node[pos=.5,left] {$[\zeta_3]\varphi$};
\draw[->] (E0.west) .. controls (-3.5,1.5) .. (E121.west) node[pos=.5,left] {$[\zeta_3^2]\varphi$};

\draw[->,violet,dashed] (E121) .. controls (.5,1.5) .. (E0) node[pos=.5,right] {$\widehat{\varphi}$};
\draw[->,blue,dotted,line width=.3mm] (E121) .. controls (1.5,1.5) .. (E0) node[pos=.5,right] {$\widehat{\varphi}[\zeta_3^2]$};
\draw[->] (E121.east) .. controls (3.5,1.5) .. (E0.east) node[pos=.5,right] {$\widehat{\varphi}[\zeta_3]$};

\end{tikzpicture}
\begin{center}
 \caption{This figure shows all the $2$-isogenies between $j=0$ and $j=121$ for $p=179$.  Dual pairs are indicated by arrows matching in style.  The isogenies entering $j=0$ (upward in the picture) differ by post-composition by an automorphism: either $[\pm 1]$, $[\pm\zeta_3]$ or $[\pm\zeta_3^2]$.  Therefore they are represented by a single directed edge in $\mathcal{G}_\ell$.  The downward arrows from $j=0$ to $j=121$ are represented by distinct directed edges in $\mathcal{G}_\ell$.  The two walks $\widehat{\varphi} [\zeta_3] \circ \varphi$ (up on dashed, down on solid) and $\widehat{\varphi} \circ [\zeta_3] \varphi$ (up on dotted, down on dashed) give the \emph{same} endomorphism of $j=121$, namely $\widehat{\varphi} \circ [\zeta_3] \circ \varphi$, despite being different walks in $\mathcal{G}_\ell$.  In fact, the nine distinct closed walks of length $2$ from $121$ to $0$ and back compose to only three distinct endomorphisms of $j=121$, namely $\pm [2]$, $\pm \widehat{\varphi} \circ [\zeta_3] \circ \varphi$, or $\pm \widehat{\varphi} \circ [\zeta_3^2] \circ \varphi$.  Finally, notice that \emph{any} walk of length two from $j=0$ to $j=121$ and back is considered backtracking according to Definition~\ref{defn:backtracking}.  The same is not true for walks of length two from j=121 to $j=0$ and back.}\label{fig:0121} 
\end{center}
\end{figure}
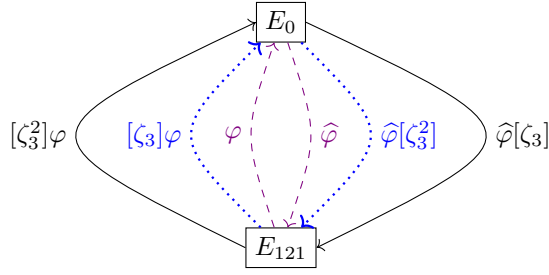

The following lemma is now immediate.

\begin{lemma}
\label{lemma:unique-endo}
Assume we have made an arbitrary assignment for $\mathcal{G}_\ell$.  Then any closed walk in $\mathcal{G}_\ell$
results in a unique endomorphism (up to sign) by composition of its component isogenies in order. 
\end{lemma}

The endomorphism obtained by composing the isogenies of a cycle in order will simply be called the \emph{composition} of the cycle.
The fact that composition is well defined allows us to define a notion of backtracking without ambiguity around curves with extra automorphisms.

\begin{definition}[Backtracking closed walk]
\label{defn:backtracking}
Given an arbitrary assignment on $\mathcal{G}_\ell$, two consecutive edges $\bullet \substack{\overset{\phi_0}{\longrightarrow}} \bullet \substack{\overset{\phi_1}{\longrightarrow}} \bullet$ are said to be \emph{backtracking} if $\phi_1 \circ \phi_0 = [\ell]$, up to possible post-composition by an automorphism.  A closed walk is considered to contain backtracking if any consecutive edges, including the last and first, are backtracking. 
\end{definition}

This definition of backtracking is the same definition as \cite[Definition 4.3]{bank2019cycles}, except those authors leave out the composition of the last and first steps, which is more appropriate in their context.  Note that this definition is not captured by the sequence of vertices alone.  To wit, see Figure~\ref{fig:0121} for an example of two length-two walks, one of which is backtracking and one of which is not backtracking, but which are taken to each other by an isomorphism of the graph fixing vertices. 

We also have a converse to Lemma~\ref{lemma:unique-endo}.

\begin{lemma}
\label{lemma:unique-cycle}
Assume we have made an arbitrary assignment for $\mathcal{G}_\ell$.
        Consider an endomorphism $\theta \in \End(E)$ of $\ell$-power degree.  If $\theta$ is not divisible by $[\ell]$, then there is a unique closed walk 
        in $\mathcal{G}_\ell$ with composition $\varphi \theta$,
        for some $\varphi \in \Aut(E)$.  Furthermore, the walk has no backtracking, and $\varphi$ is unique up to sign.
\end{lemma}

\begin{proof}
   Since $[\ell] \nmid \theta$, the kernel $\ker(\theta) \cap E[\ell]$ has size $\ell$.  Define the isogeny $\phi_1:E\to E'$ corresponding to this subgroup (using the arbitrary assignment if applicable).  Then any isogeny cycle which composes to $\theta$ must factor through this isogeny, hence this is the first step in the desired isogeny cycle: $\theta = \phi' \circ \phi_1$.  Now take $\ker \phi' \cap E'[\ell]$ to find the next kernel, and the map $\phi_2$ so that $\theta = \phi'' \circ \phi_2 \circ \phi_1$.   Continuing, we obtain a closed walk 
   in the $\ell$-isogeny graph composing to an isogeny with the same kernel as $\theta$.  Which is to say, $\pm \theta$, up to post-composition by an automorphism of $E$.  But by construction, this walk is unique.  It has no backtracking, since backtracking would produce a composition divisible by $[\ell]$.
   The same closed walk 
   is obtained if we begin with $\theta$ or a post-composition of $\theta$ by an automorphism on $E$. Thus, by Lemma~\ref{lemma:unique-endo}, the automorphism $\varphi$ of the statement is unique up to sign.
\end{proof}

Finally, there is one more small lemma we will require, that shows that any closed walk composing to a power of an endomorphism is a power of a closed walk.

\begin{lemma}
\label{lem:primitive-cycle}
Suppose we have made an arbitrary assignment for $\mathcal{G}_\ell$ and let $\iota \colon K \rightarrow \End^0(E)$ be an $\mathcal{O}$-primitive orientation.  Let $\alpha = u\beta^s$ for some $s > 1$, $\beta \in \mathcal{O}$, and some unit $u \in \mathcal{O}$.  Assume that $\ell$ does not divide $\alpha$, and that a closed walk 
$C$ of $\mathcal{G}_\ell$ composes to $\pm \iota(\alpha)$.  Then $C = C_\beta^s$ where $C_\beta$ is a closed walk of $\mathcal{G}_{\ell}$ 
that composes to $\pm \iota(\beta)$.
\end{lemma}

\begin{proof}
By Lemma~\ref{lemma:unique-cycle},  since $\ell$ does not divide $\beta$, there exists a unique closed walk 
$C_\beta$ of $\ell$-isogenies composing to $\varphi \iota( \beta)$ for some automorphism $\varphi$ of $E$.  If $E$ has no extra automorphisms, then the proof is done.

If $E$ has extra automorphisms, however, then we slightly modify $C_\beta$. Change the initial edge $\phi$ of $C_\beta$ by replacing it with $\phi\varphi^{-1}$ (having the same domain and codomain). Call this modified closed walk 
$C_\beta'$.   By the conventions of the arbitrary assignment (Definition~\ref{defn:arb}), the composition around $C_\beta'$ is $\varphi \iota( \beta) \varphi^{-1}$.  The composition of $(C_\beta')^{s}$ is $\varphi \iota( \beta)^{s} \varphi^{-1} = \varphi \iota(u^{-1}) \iota(\alpha) \varphi^{-1}$.  But then we have closed walks 
$C$ and $(C_\beta')^s$, composing to $\pm \iota(\alpha)$ and $\pm \varphi \iota(u^{-1}) \iota(\alpha) \varphi^{-1}$, respectively.   By the arbitrary assignment and Lemma~\ref{lemma:unique-cycle}, we can again modify the first edge of $(C_\beta')^s$ to obtain $C'$ which composes to $\pm \varphi \iota(u^{-1})\iota(\alpha)$.  But then $C = C'$ by Lemma~\ref{lemma:unique-cycle}, and so $\varphi = \pm  \iota(u) \in \iota(\mathcal{O})$.  Therefore $\varphi \in \iota(\mathcal{O})$, in which case $(C_\beta')^s$ composes to $\pm \iota(u^{-1}) \iota(\alpha)$, and by Lemma~\ref{lemma:unique-cycle}, $C = (C_\beta')^s$. 
\end{proof}

\subsection{Proof of Theorem~\ref{thm:mainbij-nobase}}
\label{sec:proofbij}
In order to prove Theorem~\ref{thm:mainbij-nobase}, it is easier to first discuss rims and isogeny cycles with a  basepoint (a marked point). 
Let $E$ be a fixed vertex of $\mathcal{G}_{\ell}$. 
We define the following sets.
\[
    \mathcal{R}_{E,r} =
    \left\{ \substack{\mbox{ directed rims of size $r$ in $\mathcal{G}_{K,\ell}$}\\[1pt] \mbox{passing through $(E,\,\iota)$} } \ :\
    \substack{\mbox{$K$ is an imaginary quadratic field,} \\[1pt]\mbox{$\iota$ is a $K$-orientation on $E$}
    }\right \},
\]
\[
\mathcal{I}_{E,r} =\left\{ (\iota, \mathfrak{l}) : \substack{ 
\mbox{$\iota \colon K \rightarrow \End^0(E)$ a primitive $\mathcal{O}$-orientation,} \\
\mbox{$K$ is an imaginary quadratic field,} \\[1pt] \mbox{$\mathcal{O}$ is an $\ell$-fundamental order of $K$,} \\
\mbox{$(\ell) = \mathfrak{l}\bar{\mathfrak{l}}$ splits in $\mathcal{O}$,} \\ \mbox{$[\mathfrak{l}]$ has order $r$ in $\Cl(\mathcal{O})$}, \\
}
\right\}, \text{ and}
\]
     \[
    \mathcal{C}_{E,r} = \{ \mbox{isogeny cycles of length $r$ in $\mathcal{G}_{\ell}$ starting and ending at $E$ } \}.
    \]
    
Let $\sim$ denote conjugation on elements of $\mathcal{I}_{E,r}$, i.e.\ taking a pair $(\iota, \mathfrak{l})$ to $(\overline{\iota}, \overline{\mathfrak{l}})$.  As described in Section~\ref{sec:bij-statement}, conjugation preserves the direction of the rim but conjugates the ideal whose action traverses the rim. 
We will continue to use the symbol $\sim$ to denote conjugation of a rim.

We assume a safe arbitrary assignment (Definition~\ref{defn:arb2}) has been made for $\mathcal{G}_\ell$. Depending upon this arbitrary assignment, we define the following maps:
\begin{equation}
\xymatrix{
   \mathcal{I}_{E,r}/{\sim}
   \ar@<.3em>[rr]^{\Psi} 
   & &
   \mathcal{R}_{E,r}/{\sim} \ar@<.3em>[ll]^{\Psi'} \ar@<.3em>[dl]^{\Phi}  \\
   & \mathcal{C}_{E,r}
   \ar@<.3em>[ul]^{\Theta}
   &
}    \label{eqn:triangle}
\end{equation}
The goal is to show that the diagram commutes, where every map is a bijection.  Then Theorem~\ref{thm:mainbij-nobase} will follow by `forgetting' the basepoint.

Disregarding conjugation for the moment, we first define maps 
\[
\xymatrix{
   \mathcal{I}_{E,r}
   \ar@<.3em>[r]^{\Psi} &
 \mathcal{R}_{E,r} \ar@<.3em>[l]^{\Psi'}
   }.
\]
In Lemma~\ref{lem:psi}, we show that these are inverse bijections and that they provide inverse bijections even after a quotient by the conjugation relation.  In Lemma~\ref{lem:theta} we will define $\Theta$ and show it is a bijection.  Lemma~\ref{lem:phi} will define $\Phi$ and show it is a bijection.  Finally, in Theorem~\ref{thm:mainbij}, we show that the triangle in \eqref{eqn:triangle} commutes.

\textbf{Definition of $\Psi'$.}
Consider a directed rim $R \in \mathcal{R}_{E,r}$ passing through $(E,\iota)$.  Denote by $\mathcal{O}$ the order such that $\iota$ is $\mathcal{O}$-primitive.  Then $\mathcal{O}$ is $\ell$-fundamental since $(E,\iota)$ is at the rim.   Moreover, the $\mathcal{O}$-ideal $(\ell)$ splits and the $\mathcal{O}$-ideals above $\ell$ have order $r$ in $\Cl(\mathcal{O})$ since $(E,\iota)$ belongs to a rim of size $r$.  The direction along the rim corresponds to the action of one of the ideals above $\ell$, which we denote $\mathfrak{l}$.  Then $\Psi'$ is defined as
\[
\Psi'\colon \mathcal{R}_{E,r} \rightarrow \mathcal{I}_{E,r},
\quad
R \mapsto (\iota, \mathfrak{l}).
\]

\textbf{Definition of $\Psi$.}
On the other hand, consider $(\iota,\mathfrak{l}) \in \mathcal{I}_{E,r}$.  The orientation $\iota$ on $E$ gives rise to a directed rim of size $r$ passing through $(E,\iota)$ by the consecutive action of $\mathfrak{l}$.  Let $\Psi$ send $(\iota, \mathfrak{l})$ to this directed rim in $\mathcal{R}_{E,r}$, i.e.
\[
\Psi \colon \mathcal{I}_{E,r} \rightarrow \mathcal{R}_{E,r}, 
 \quad
(\iota, \mathfrak{l}) \mapsto R = ( (E,\iota) \rightarrow_{\phi_{\mathfrak{l}}} \cdots ).
\]

\begin{lemma}
\label{lem:psi}
The maps $\Psi$ and $\Psi'$ defined above are inverse bijections.  They also provide bijections
\[
\xymatrix{
   \mathcal{I}_{E,r}/{\sim}
   \ar@<.3em>[r]^{\Psi} &
   \mathcal{R}_{E,r}/{\sim} \ar@<.3em>[l]^{\Psi'}
   }
\]
where $\sim$ denotes equivalence between conjugates.
\end{lemma}

\begin{proof}
The fact that $\Psi$ and $\Psi'$ are inverse to each other follows from their definitions. 
It is also clear that conjugate rims  correspond to conjugate orientations (while fixing $E$ and $K$).
\end{proof}

\textbf{Definition of $\Theta$.}
We now define a map $\Theta$ from $\mathcal{C}_{E,r}$ to $\mathcal{I}_{E,r}$. Let $C \in \mathcal{C}_{E,r}$.  Starting at the vertex $E$, this cycle composes into an endomorphism $\pm \theta \in \End(E)$. This endomorphism $\pm \theta$ induces an orientation $\iota_\theta \colon K \rightarrow \End^0(E)$ on $E$ (defined up to conjugation). 
 Fixing a choice for $\iota_\theta$, let $\mathcal{O}$ be the order for which $\iota_\theta$ is primitive.  Then $\ell$ must split or ramify in $\mathcal{O}$ because $\theta \in \mathcal{O}$ has norm $\ell^r$ and is not divisible by~$\ell$ (since $C \in \mathcal{C}_{E,r}$ has no backtracking).  If $\phi_\mathfrak{l}^{(E,\iota_\theta)}$  is equal to the first isogeny of the cycle $C$ in its indicated direction, for a prime $\mathfrak{l}$ above $\ell$, then define $\Theta(C) := (\iota_\theta, \mathfrak{l})$.

\begin{lemma}
\label{lem:theta}
The process described above gives a well-defined map $\Theta \colon \mathcal{C}_{E,r} \rightarrow \mathcal{I}_{E,r}/{\sim}$ (in particular, the ideal $\mathfrak{l}$ indicated in the definition above exists).
\end{lemma}
\begin{proof}
Let $C_{E,r} \in \mathcal{C}_{E,r}$ and let $\theta \in \End(E)$ be the composition of the edges in $C_{E,r}$ (unique up to sign). Let~$\alpha$ denote the preimage of $\theta$ in $K$ via $\iota_\theta$. The orientation $\iota_\theta$ is a primitive $\mathcal{O}$-orientation for exactly one order $\ZZ[\alpha] \subseteq \mathcal{O}\subseteq \QQ(\alpha)$.  If $\ZZ[\alpha]$ has conductor divisible by $\ell$, then $\alpha$ has discriminant and norm divisible by $\ell$, hence trace divisible by $\ell$, so it is itself divisible by $\ell$.  
But this is impossible because the isogeny cycle $C_{E,r}$ has no backtracking. 
Hence $\ZZ[\alpha]$, and therefore $\mathcal{O}$, are $\ell$-fundamental orders. 

Note that $\alpha$ has norm $\ell^r$.
We consider the factorization of $(\ell)$ in $\mathcal{O}$.  As mentioned in the definition of~$\Theta$, $\ell$ splits or ramifies, so $(\ell) = \mathfrak{l}\overline{\mathfrak{l}}$. Then we can factor $(\alpha)$, using its norm:  $(\alpha) = (\ell)^t \mathfrak{l}^s$ for some $t,s$ (up to possibly replacing $\mathfrak{l}$ with its conjugate).  However, $t=0$ since $\ell$ does not divide $\alpha$.  Hence $s=r$, by comparing norms.  That is, $(\alpha) = \mathfrak{l}^r$. 

This shows that $\ell$ is represented by a class $[\mathfrak{l}]$ of order dividing $r$ in $\Cl(\mathcal{O})$.  We will show it is of order exactly $r$.   If not, then $\mathfrak{l}^s = (\beta) \in \mathcal{O}$, so $\alpha =  u\beta^s$ for some unit $u \in \mathcal{O}$.  Then by Lemma~\ref{lem:primitive-cycle}, the cycle $C_{E,r}$ is a power of a smaller cycle, a contradiction to the definition of isogeny cycle.

Finally, since $(\alpha) = \mathfrak{l}^r$, the kernel $E[\mathfrak{l}]$ is contained in $\ker \theta$, and so $\phi_\mathfrak{l}$ defined on $(E, \iota_\theta)$ is equal to the first isogeny of $C$ in its direction of travel.

Note that our only arbitrary choice is that of $\iota_\theta$ in place of its conjugate, but as we map into $\mathcal{I}_{E,r}/{\sim}$, the other choice results in the same image element, and the map is well-defined.
\end{proof}

\textbf{Definition of $\Phi$.}
Next we consider the map $\Phi$.  Given a directed rim $R \in \mathcal{R}_{E,r}/{\sim}$, we get a walk in  $\mathcal{G}_\ell$ by forgetting orientations.  If there are no extra automorphisms on any curves of the rim, this is well-defined up to a sign on the isogenies.  If there are extra automorphisms, we must use the process in Section~\ref{sec:graphcover} to safely forget orientations.  The resulting cycle will be denoted $\Phi(R)$.  Note that if $R$ is conjugated, we obtain the same walk.  This is clear when there are no extra automorphisms.  When we use the process in Section~\ref{sec:graphcover}, this follows from the observation that any valid set of representatives is again valid for the conjugate rim, in the proof of Proposition~\ref{prop:rim-to-cycle}.

\begin{lemma}
\label{lem:phi}
Let $R \in \mathcal{R}_{E,r}$.  Then $\Phi(R) \in \mathcal{C}_{E,r}$.
\end{lemma}

\begin{proof}
Let $W$ be the walk obtained from $R$ by forgetting the orientations.  We need to show that $W$ doesn't have backtracking and is not repeating smaller cycles.  Let $(\iota, \mathfrak{l}) = \Psi'(R) \in \mathcal{I}_{E,r}/{\sim}$. 
First, we show that the action of $\mathfrak{l}$, which dictates the direction of $R$, cannot produce a backtracking path.  Suppose the action of $\mathfrak{l}$ produces a backtracking path $\phi_1, \phi_2$.  Then there is an endomorphism $\varphi$ of $E$ such that $\varphi [\ell] = \phi_2 \phi_1$ has the same kernel as $\phi_{\mathfrak{l}^2}$, namely $E[\ell]$.  This implies $\mathfrak{l}^2 = (\ell)$, so $r=2$, a contradiction.

Now suppose that the walk $W$ in $\mathcal{G}_\ell$ repeats a smaller closed walk $k > 1$ times.  Let $\phi \in \End(E)$ be the composition of the smaller closed walk.  Then $\phi = \phi_\mathfrak{a}$ for the ideal $\mathfrak{a} = \mathfrak{l}^{r/k}$, and $\phi^k = \phi_\mathfrak{b}$ where $\mathfrak{b} = \mathfrak{l}^r = (\beta)$ for some $\beta \in \mathcal{O}$.  Hence as an element of $\End(E)$, $\phi^k = \iota(\beta)$, which implies that $\phi \in \iota(K)$.  But then letting~$\alpha$ be such that $\iota(\alpha) = \phi$, we have $\alpha^k = \beta$ and $\mathfrak{l}^{r/k} = (\alpha)$.  This contradicts the rim being of size $r$.  Thus, the walk $W$ in $\mathcal{G}_{\ell}$  cannot repeat a smaller walk $k > 1$ times.
\end{proof}

\begin{theorem}\label{thm:mainbij}
Given a safe arbitrary assignment for $\mathcal{G}_\ell$, we define maps $\Phi$ and $\Phi' := \Psi \circ \Theta$ as above.  Then the sets $\mathcal{C}_{E,r}$ and $\mathcal{R}_{E,r}/{\sim}$ are in bijection under the inverse pair of maps $\Phi$ and $\Phi'$.
\end{theorem}
\begin{proof}
We make a safe arbitrary assignment (Definition~\ref{defn:arb2}), to define the maps above.
We will show that the two maps are inverse to one another.

Suppose $C \in \mathcal{C}_{E,r}$ is taken to $[R] \in \mathcal{R}_{E,r}/{\sim}$ by $\Phi'$.  If $\theta$ is the composition of $C$, then $R$ is the rim passing through $(E, \iota_\theta)$ ($\overline{R}$ is the rim passing through $(E, \overline{\iota_\theta})$, respectively).  In particular, $\mathfrak{l}^r = (\alpha)$ where $\iota_\theta(\alpha) = \theta$.  Therefore, composing the isogenies around $R$ results in the endomorphism $\pm \theta$.  By Lemma~\ref{lemma:unique-cycle} (and using  Proposition~\ref{prop:rim-to-cycle} in the case of extra automorphisms), the rim $R$ maps to $C$ when orientations are forgotten (i.e.\ under $\Phi$).

Next, suppose $[R] \in \mathcal{R}_{E,r}/{\sim}$ is taken to $(\iota_\theta, \mathfrak{l}) \in \mathcal{I}_{E,r}$ by $\Psi'$.  Let $C = \Phi(R)$.  Then $C$ composes to $\theta$ (using Proposition~\ref{prop:rim-to-cycle} in the case of extra automorphisms).  Thus $C$ maps to $(\iota_\theta, \mathfrak{l}) \in \mathcal{I}_{E,r}$ under $\Theta$.
\end{proof}

The discussions so far relies on a choice of a basepoint $E$.  We now prove the main theorem (Theorem~\ref{thm:mainbij-nobase}) as a corollary to the basepointed version (Theorem~\ref{thm:mainbij}).

\begin{proof}[Proof of Theorem~\ref{thm:mainbij-nobase}]
Make a safe arbitrary assignment for $\mathcal{G}_\ell$ and use it to define the maps of diagram \eqref{eqn:triangle} as described above.

There is a map $\mathcal{C}_{E,r} \rightarrow \mathcal{C}_r$ given by forgetting the basepoint of the isogeny cycle.  There is also a map $\mathcal{R}_{E,r}/{\sim} \ \rightarrow \mathcal{R}_r/{\sim}$ given by forgetting the basepoint of the rim.

Given $[R] \in \mathcal{R}_r/{\sim}$, a basepoint can be chosen in $r$ ways.  For each such basepoint $(E,\iota)$, we obtain $R_{E,\iota} \in \mathcal{R}_{E,r}$, which we map to $C_{E,\iota} := \Phi(R_{E,\iota}) \in \mathcal{C}_{E,r}$ via Theorem~\ref{thm:mainbij}.  We must show that the resulting~$C_{E,\iota}$ are all equal when basepoints are forgotten.  But the map $\Phi$ to isogeny cycles is agnostic to the basepoint:  this is clear for rims without extra automorphisms, as there are no choices to make, and is a consequence of Proposition~\ref{prop:rim-to-cycle} when there are extra automorphisms.

Similarly, a given $C \in \mathcal{C}_r$ corresponds to $C_{i} \in \mathcal{C}_{E_i,r}$ for $E_i$ the $i$-th curve in the cycle.  Each of these basepointed cycles gives a corresponding basepointed rim $R_{i} := \Phi'(C_{i}) \in \mathcal{R}_{E_i,r}$.  We must show that the resulting $R_{i}$ are all equal, when basepoints are forgotten.  Define $(\iota_i, \mathfrak{l}_i) = \Theta(C_i) \in \mathcal{I}_{E,r}$.  Define the composed endomorphism $\theta_i \in \End(E_i)$ for each $E_i$.  Then by construction, $\iota_i = \iota_{\theta_{i}}$ and $\phi_i = \phi^{(E_i, \iota_i)}_{\mathfrak{l}_i}$.  It suffices to show that all $(\iota_i, \mathfrak{l}_i)$ give rise to the same rim.   The $\theta_i$ are related by $\theta_{i+1} = \phi_i \theta_i \phi_i^{-1}$.  Thus, we obtain $\iota_{{i+1}} = (\phi_i)_* \iota_{{i}}$.  In particular, the orientations $\iota_i$ are all associated to the same field $K$, and they are all primitive for the same order $\mathcal{O}$ of $K$, since $\phi_i$ is horizontal.  Therefore $(E_i, \iota_i)$ and $(E_{i+1}, \iota_{i+1})$ are elements of $\SSO$ connected by a horizontal isogeny $\phi_i$ of degree $\ell$, i.e.\ they are on the same rim.  As for direction, by construction $\phi_i = \phi^{E_i}_{\mathfrak{l}_i}$ for all $i$.  Since $\phi_i$ all indicate the same direction, the $\mathfrak{l}_i$ all indicate the same direction.
\end{proof}

\subsection{Mapping from $\mathcal{G}_{K,\ell}$ to $\mathcal{G}_\ell$}
\label{sec:graphcover}
We now deal with the one remaining item pushed aside during the proof: intricacies introduced by extra automorphisms when defining a map from rims to closed walks by forgetting orientations.
One can hope to map any walk in $\mathcal{G}_{K,\ell}$ to a walk in $\mathcal{G}_\ell$ by forgetting orientations.  However, isogenies in the oriented isogeny graph are taken up to pre- and post-composition by automorphisms (see~\eqref{eqn:equivalent-isogenies}), so when mapping down to $\mathcal{G}_\ell$, there is a choice to be taken if the domain has extra automorphisms.  Even more to the point, if we obtain a closed walk from a rim, by making the choice arbitrarily, we may end up with a closed walk whose composition does not lie in $\iota(K)$.  (Recall that in order to obtain a bijection, we hope to recover the rim from the isogeny cycle using this composition.)  

In this section, we sidestep this issue, with the help of a special type of arbitrary assignment, by showing that it is nevertheless possible to accomplish something more limited: to define a map from rims to isogeny cycles such that the resulting isogeny cycle has a composition in $\iota(K)$.  We emphasize that this is a non-issue (and this entire section can be skipped) if $\mathcal{G}_\ell$ has no curves with extra automorphisms.
To be somewhat more precise, consider a rim $R$ in $\mathcal{G}_{K,\ell}$.  If we choose representatives of the vertices and edges (which are, a priori, only equivalence classes), then we can forget orientations to obtain a closed walk in $\mathcal{G}_\ell$.  However, the choice of representatives must be made with care.  First, the edges must belong to the arbitrary assignment:  if not, then when composing around the rim before and after forgetting orientations, we may differ by a post-composition, and hence may produce an endomorphism from the wrong quadratic field (which we will see momentarily in Example~\ref{ex:loops} below).  Second, the choice of representatives must be consistent:  when we traverse the entire rim, beginning at representative $(E,\iota)$, we must return to representative $(E,\iota)$, and not merely an isomorphic oriented curve.  In the proof that follows, we spend the vast majority of our energy in showing that such a choice of representative exists and is unique.

\begin{example}
\label{ex:loops}
An example demonstrates some of the problems the proof will need to overcome.  Consider the example $p=11$ and $\ell=3$ where $\mathcal{G}_3$ has two loops at $j=1728$.  Recall the endomorphism ring of the curve with $j$-invariant $j = 1728$ from Section~\ref{sec:back:ss}.
There are eight endomorphisms of degree $\ell=3$, namely
\[
\frac{\pm 1\pm \pi_p}{2} \quad \text{and}\quad \frac{\pm [i]\pm [i]\pi_p}{2} .
\]
The first group have trace $\pm 1$ and the second group have trace $0$.  The endomorphisms
\[
\pm \frac{1+\pi_p}{2}, \quad
\pm \frac{[i]+[i]\pi_p}{2}
\]
have the same kernel, which we denote $\kappa$.  This corresponds to one of the loops, and these four endomorphisms differ by post-composition by an automorphism (namely $\pm 1, \pm [i]$).  The endomorphisms
\[
\pm \frac{1-\pi_p}{2}, \quad
\pm \frac{[i]-[i]\pi_p}{2}
\]
all have the same kernel, namely the image of $\kappa$ under $[i]$.  This corresponds to the second loop, and these four endomorphisms also differ from one another only by post-composition by an automorphism.

If, in the arbitrary assignment, we choose $\frac{1+\pi_p}{2}$ for the first loop and $\frac{1 - \pi_p}{2}$ for the second, then these endomorphisms generate orientations by $\ZZ\left[ \frac{1+\sqrt{-11}}{2}\right]$.  The endomorphisms $\pm([i] + [i]\pi_p)/2$ of trace zero generate orientations by $\ZZ[ \sqrt{-3} ]$.  The graph $\mathcal{G}_{\QQ(\sqrt{-3}),\ell}$ has one rim in $\SS_{\ZZ[\sqrt{-3}]}$, namely the loop on the oriented curve $(E,\iota)$ where $E$ has $j$-invariant $1728$ and $\iota(\sqrt{-3}) = [\pm \sqrt{-3}]$ (case (i) of Remark~\ref{prop:gkl-auto}).  Under this arbitrary assignment, there is no place to map this rim in such a way that the composition of the resulting loop will match the order $\ZZ[\sqrt{-3}]$ as desired. 
\end{example}

This motivates modifying Definition~\ref{defn:arb} so that the two loops must compose to endomorphisms from different fields $K$.  
In particular, we must be careful to choose the arbitrary assignment `in tandem' for multiple edges between curves with extra automorphisms.  This issue arises in larger cycles also.  For an explicit example, consider $p=11$, and $\QQ(\sqrt{-23})$, which has class group of size $3$ generated by $\mathfrak{l}$ lying above $\ell=2$;  the rim for the maximal order is of size $3$ with $j$-invariants $1728,1728,0$.
We therefore make a special requirement on the arbitrary assignment which applies to this case.

  \begin{definition}[Safe arbitrary assignment]
  \label{defn:arb2}
    An arbitrary assignment is said to be \emph{safe} if it obeys the following rule: whenever there is a collection of edges $E_1 \rightarrow E_2$ related by precomposition by an automorphism of $E_1$ besides $[\pm 1]$, we choose their representatives to be of the form $\{  \pm \phi \varphi : \varphi \in \Aut(E_1) \}$ for a fixed choice~$\phi$.
  \end{definition}
  
  The rule in Definition~\ref{defn:arb2} is automatically satisfied at $E_1$ and $E_2$ by any arbitrary assignment unless both $E_1$ and $E_2$ have extra automorphisms, and there are multiple edges between them.  It is also clear that a safe arbitrary assignment is always possible.
  
  With regards to Example~\ref{ex:loops}, note that $\frac{1+\pi_p}{2} = -[i]\frac{1 - \pi_p}{2} [i]$, which violates Definition~\ref{defn:arb2}.  This is the reason the example `fails' the bijection.

\begin{proposition}
\label{prop:rim-to-cycle}
Suppose a safe arbitrary assignment has been made for $\mathcal{G}_\ell$. There is a map from rims $R$ of $\mathcal{G}_{K,\ell}$ to closed walks 
$C(R)$ in $\mathcal{G}_\ell$ with the following properties:
\begin{enumerate}
    \item \label{item:map-directed} The curves and isogenies  of $C(R)$ are representatives of the isomorphism classes of curves and isogenies of $R$, in the same order; in particular, there is a map $R \rightarrow C(R)$ of directed graphs.
    \item \label{item:re-compose} The composition $\pm \theta$ of $C(R)$ at vertex $E$ in $C(R)$ corresponding to $(E, \iota)$ in $R$, gives rise to the orientation $\pm \iota$ in the sense that $\theta = \iota(\alpha)$ for some $\alpha \in K$.
\end{enumerate}
\end{proposition}

\begin{proof}
\textbf{Defining the map.}
It is possible to define such a map as follows.  Consider a rim $R$ of length $r$, associated to an ideal $\mathfrak{l}$ above $\ell$ in the associated $\ell$-fundamental order $\mathcal{O}$, with basepoint $(E_0,\iota_0)$, written as follows:
\[
\xymatrix{
(E_0,\iota_0) \ar[r]^{\phi_0} &
(E_1,\iota_1) \ar[r]^{\phi_1} &
(E_2,\iota_2) \ \cdots \!\!\!\!\!\!\!\!\!\!\!\!\!\! & 
(E_{r-1}, \iota_{r-1}) \ar[r]^{\quad \phi_{r-1}} &
(E_0, \iota_0)
}.
\]
Write $\phi := \phi_{r-1} \phi_{r-2} \cdots \phi_1 \phi_0 \in \End(E)$.  The result $\phi$ depends on representative choices of the isomorphism classes of oriented curves $(E_i, \iota_i)$ and of the isogenies $\phi_i$ shown above.  We seek to choose such representatives so that
\begin{enumerate}[label=(\roman*)]
    \item \label{item:iota} $\phi_*\iota_0 = \pm \iota_0$,
    \item \label{item:arbitrary} each $\phi_i$ is chosen from the safe arbitrary assignment.
\end{enumerate}
Such a choice will be called a \emph{valid set of representatives} during the course of this proof.  With such a valid set of representatives, to define a map from rims to isogeny cycles, we then forget orientations to obtain an isogeny cycle in $\mathcal{G}_\ell$.  This is now well-defined because of \ref{item:arbitrary} and will define the map referred to in the statement of the proposition.  For any rim $R$, we will write $C(R)$ for the isogeny cycle which is its image.

\textbf{Facts about valid sets of representatives.}
We will need a few facts about valid sets of representatives.  First, we will show that if we have a valid set of representatives with respect to a basepoint $(E_0, \iota_0)$, then it is also a valid set of representatives for any other $(E_i,\iota_i)$ in the cycle.  In other words, for any $k$, we have $(\phi_{r-k} \cdots \phi_{r-k+1})_* \iota_{r-k+1} = \pm \iota_{r-k+1}$ (where indices are taken modulo $r$).  To see this for $(E_1, \iota_1)$, use $\iota_1 = (\phi_0)_* \iota_0$ to compute the left hand side, using the fact that $\phi_* \iota_0 = \pm \iota_0$; repeat for $E_2$, $E_3$ etc. 

Let $\alpha \in \mathcal{O}$ be such that $\mathfrak{l}^r = (\alpha)$.   Without causing confusion, we will write $\phi_{\mathfrak{l}} \colon  (E_i,\iota_i) \rightarrow (E_{i+1}, \iota_{i+1})$ for each $i$ (this isogeny is defined up to isomorphism, by the convention that its kernel is $E_i[\iota_i(\mathfrak{l})]$).
We will now prove that any valid set of representatives also satisfies
\begin{enumerate}[label=(\roman*)]
\setcounter{enumi}{2}
    \item \label{item:classgp} $\ker \phi_i = \ker \phi_{\mathfrak{l}} = E_i[\iota_i(\mathfrak{l})]$,
    \item \label{item:classgp2} $\phi = \pm \iota_0(\alpha)$.
\end{enumerate}
  To prove this, observe that $\ker \phi_{\mathfrak{l}^r} = E_i[\iota_0(\alpha)] = \ker \iota_0(\alpha)$.  In particular, $\iota_0(\alpha)$ is $K$-isomorphic to $\phi_{\mathfrak{l}^r}$.  Furthermore, $(\iota_0(\alpha))_* \iota_0 = \pm \iota_0$ because elements of $\iota_0(K)$ commute. 
Next, since $\phi_* \iota_0 = \pm \iota_0 = (\iota_0(\alpha))_* \iota_0$ (by \ref{item:iota}) it must be that $\phi = \pm \iota_0(\alpha)$, since this is the only pre-/post-composition $\phi$ of $\iota_0(\alpha)$ that continues to satisfy $\phi_* \iota_0 = \pm \iota_0$.  This implies that $\ker \phi = \ker \phi_{\mathfrak{l}^r}$ and hence $\ker \phi_0 = \ker \phi_{\mathfrak{l}}$.  Since the valid set of representatives is valid at each basepoint, we have $\ker \phi_i = \ker \phi_{\mathfrak{l}}$ by the same reasoning.

\textbf{Claim: The map defined on a valid set of representatives satisfies requirements \eqref{item:map-directed} and \eqref{item:re-compose} of Proposition \ref{prop:rim-to-cycle}.}
 By \ref{item:classgp2}, $\phi = \pm \iota_0(\alpha)$.
The composition of the isogeny cycle $C(R)$ is $\pm \phi$, which demonstrates \eqref{item:re-compose} at the vertex $(E_0, \iota_0)$. 
Then \eqref{item:re-compose} holds at the other vertices in $R$ by the fact that a valid set of representatives is valid at each of the vertices of $R$.  The fact that the map satisfies \eqref{item:map-directed} is by definition.

\textbf{Proof outline.}
It remains to show that there is a set of representatives satisfying \ref{item:iota}--\ref{item:arbitrary} (i.e.\ it is a valid set of representatives), and that any such set leads to the same isogeny cycle (i.e.\ the map $R\to C(R)$ given above is well-defined).  To do this, we will divide the possible rims into two cases:  exceptional and non-exceptional.  The exceptional case is the following:  first, we assume every vertex of the rim has the same $j$-invariant with extra automorphisms, of which $\varphi$ is a non-trivial ($\neq [\pm 1]$) element; then, defining $\phi$ as above for any basepoint, we assume that no matter which vertex we start with as basepoint, we have $\ker \phi = \ker \phi \varphi$.

\textbf{The map $R\mapsto C(R)$ is well-defined.}
Let us show that the map $R \mapsto C(R)$ is well-defined in the non-exceptional cases.  A priori, it may be that there are two choices of representatives satisfying \ref{item:iota}--\ref{item:arbitrary}.  Suppose so.  Use primes ($\iota', \phi_i'$ etc.) to denote the second valid set of representatives.  

First note that if $(E_i, \iota_i') = (E_i, \pm \iota_i)$ at any vertex, then the two choices of representatives must be the same.  To see this, note that the vertex representative $(E_i, \pm \iota_i)$ determines the kernel of $\phi_i$ (by \ref{item:classgp}), which in turn determines $\phi_i$ up to sign (by \ref{item:arbitrary}).  Therefore $\phi_i' = \pm \phi_i$.  But this implies that $(E_{i+1},\iota_{i+1}') = (E_{i+1},\pm \iota_{i+1})$.  Proceeding in this way, the two valid sets of representatives are the same.  Hence the proof is complete if any curve in the rim has no extra automorphisms (in which case there is no choice for $\iota_i$).

Suppose that we have $\iota_0' = \varphi_* \iota_0$ for some non-trivial (not $\pm 1$) automorphism $\varphi \in \End(E_0)$.  In particular, $E_0$ has extra automorphisms.  At the same time, $\iota_1' = \eta_* \iota_1$ for some $\eta \in \Aut(E_1)$ (possibly trivial this time).   These relationships entail that $\phi_0' = \eta \phi_0 \varphi^{-1}$.  By \ref{item:arbitrary}, $\phi_0$ and $\phi_0'$ both belong to the safe arbitrary assignment, so $\eta = \pm 1$ and hence $\iota_1' = \pm \iota_1$; thus, we have reduced to the case in the previous paragraph.

\textbf{Existence of a valid set of representatives in non-exceptional cases.}
We use the safeness of the safe arbitrary assignment (Definition~\ref{defn:arb2}).  Let $\phi_0$ be the map with kernel $E_0[\iota_0(\mathfrak{l})]$ and post-composition determined by the safe arbitrary assignment.  Then let $\iota_1  = (\phi_0)_* \iota_0$.  Continue in this manner around the entire rim.  When we close the rim (reaching $(E_0, \iota_0)$ again), we find that $\phi_* \iota_0 = \varphi \iota_0 \varphi^{-1}$ for some automorphism~$\varphi$ (since the rim closes, we must reach an oriented curve isomorphic to the initial oriented curve).  If $E_0$ has no extra automorphisms, then the proof of this case is complete. Hence we are done if the cycle contains even one curve without extra automorphisms.

Let us therefore restrict to the case that $E_0$ has extra automorphisms and so do all other curves in the rim.  If item \ref{item:iota} fails, then there is a nontrivial automorphism $\varphi \in \Aut(E_0)$ so that $\phi_* \iota_0 = \varphi_* \iota_0$.  Then
\[
(\phi \varphi^{-1})_* (\varphi_* \iota_0) 
= \phi_* \iota_0 = \varphi_* \iota_0.
\]
Thus, replacing isomorphism class representatives $(E_0, \iota_0)$ with $(E_0, \varphi_*\iota_0)$ and $\phi_0$ with $\phi_0 \varphi^{-1}$, we recover~\ref{item:iota}.  Provided that $\phi_0$ and $\phi_0\varphi^{-1}$ have different kernels, which they do when $E_1 \not\cong E_0$ (Lemma~\ref{lem:loopsandauts}), then using the conventions of the safe arbitrary assignment (Definition~\ref{defn:arb2}), we have the freedom to replace $\phi_0$ with~$\phi_0\varphi^{-1}$ without violating item \ref{item:arbitrary}.

Thus we are reduced to the case that all $E_i$ are isomorphic (assumed equal without loss of generality), and $\phi_0$ and $\phi_0\varphi$ have the same kernel (for every choice of basepointing the rim).  This is the exceptional case.

\textbf{Existence of a valid set of representatives in the exceptional case.}
Suppose we are in the exceptional situation described in the proof outline.  That is, all vertices are isomorphic curves with extra automorphisms, and $\ker \phi = \ker \phi \varphi$.
Since $\phi \neq \pm \phi \varphi$ (as in the proof of Lemma~\ref{lem:loopsandauts}), we have $\varphi^k \phi = \pm \phi \varphi$ for  some integer $k$ with $\varphi^k \neq \pm 1$.  This means either $\phi\varphi = \varphi \phi$ or $\phi \varphi = \varphi \overline{\phi}$.  Note that $\phi$ commutes with $\varphi$ if and only if $\varphi \phi$ does, so commuting is independent of the safe arbitrary assignment.

The above is true for $\phi_0$ as well as $\phi$.  But $\ell = N(\phi_0) < p$, which by Lemma~\ref{lem:normorthogonal} implies that~$\phi$ and~ $\varphi$ commute.  The same argument applies to each $\phi_i$ for which $\ker \phi_i = \ker \phi_i \varphi$.  But if there is a $\phi_i$ for which this is not true, we are out of the exceptional case.  Hence every $\phi_i$ commutes with $\varphi$.

Therefore we may suppose $\phi$ commutes with $\varphi$.  Then $\phi \in \QQ(\varphi)$.  Regardless of the safe arbitrary assignment, the commutativity implies $\phi_* \iota = \iota$ for the orientation $\iota(\zeta_3) = \varphi$ or $\iota(i) = \varphi$ for $j=0$ and $j=1728$, respectively.  So suppose $\phi$ is the arbitrary assignment.
Then if $R$ is a rim with $\phi_\mathfrak{l} = \phi$ (up to post-composition by an automorphism), then $R$ is associated to the field $K = \QQ(\zeta_3)$ or $\QQ(i)$ respectively,  and $\phi$ provides a set of representatives satisfying \ref{item:iota}--\ref{item:arbitrary}.  Hence there is a valid set of representatives.
\end{proof}

We observe the following immediate corollary to Theorem~\ref{thm:mainbij-nobase}.

\begin{corollary}
Let $(j_1, \ldots, j_n)$ be a finite sequence of $j$-invariants of supersingular elliptic curves in $\mathcal{G}_\ell$.  Then the number of isogeny cycles in $\mathcal{G}_\ell$ on this ordered sequence of $j$-invariants is equal to the number of rims of $\bigcup_{K} \mathcal{G}_{K,\ell}$ on this ordered sequence of $j$-invariants.
\end{corollary}

In fact, an alternate approach to proving Theorem~\ref{thm:mainbij-nobase} would be to observe that it is indeed equivalent to this corollary.  One could then break into cases based on the existence of extra automorphisms.  In most cases, there are no double edges, and this is a bijection of singleton sets.

\subsection{An undirected version of the main bijection}
\label{sec:mainbij-undir}

The arbitrary assignment also allows us to set up a bijection between walks \and their `duals' (walking the same $j$-invariants backwards), despite the ambiguity at $j=0$ and $j=1728$.  This means we can discuss the bijection between undirected rims and undirected isogeny cycles.

\begin{lemma}
\label{lemma:dual-reverse}
Assume we have made a safe arbitrary assignment for $\mathcal{G}_\ell$. Let $(j_0, j_1, \ldots, j_n)$ be a sequence of $j$-invariants for a closed walk in $\mathcal{G}_\ell$ (where $j_i$ is adjacent to $j_{i+1}$, and $j_n$ to $j_0$).  Then the closed walks through this ordered list of $j$-invariants can be paired with the closed walks through the list $(j_n, j_{n-1}, \ldots, j_1, j_0)$ (i.e.\ going backwards) in such a way that the pair of endomorphisms obtained from each pair of walks (by composition) is a pair of dual endomorphisms.
\end{lemma}

\begin{proof}
Each walk $\phi_0, \phi_1, \ldots, \phi_n$ with composition $\theta$ has a `dual walk' obtained by replacing the individual isogenies with their duals and reversing the order:  the walk $\widehat{\phi_n}, \widehat{\phi_{n-1}}, \ldots, \widehat{\phi_0}$.  This results in a walk on the same vertices backwards.  If the walk does not pass through $j=0$ or $j=1728$, this dual walk gives the dual endomorphism under composition (up to sign).  The difficulty arises when choosing a dual for an outgoing edge $\phi_i \colon E_i \rightarrow E_{i+1}$ where $E_i$ has $j$-invariant $0$ or $1728$, in which case the dual $\widehat{\phi_i} \colon  E_{i+1} \rightarrow E_i$ may not be in the safe arbitrary assignment, and we are forced to use instead $\varphi\widehat{\phi_i} \colon  E_{i+1} \rightarrow E_i$ for some $\varphi \in \Aut(E_i)$.  In this case the endomorphism obtained by composition is no longer $\pm \widehat{\theta}$, because of the intervening $\varphi$ between $\widehat{\phi_i}$ and $\widehat{\phi_{i-1}}$.
However, this can be adjusted by choosing an appropriate replacement for $\widehat{\phi_{i-1}} \colon  E_{i} \rightarrow E_{i-1}$ from among the available isogenies in the safe arbitrary assignment from $E_i$ to $E_{i-1}$, which differ by precomposition by an automorphism.  In other words, can obtain the composition $\widehat{\phi_{i-1}} \widehat{\phi_i}$ as $\widehat{\phi_{i-1}}\varphi^{-1} \varphi \widehat{\phi_i}$.  In this way, we can obtain a dual path whose composition according to the safe arbitrary assignment is $\pm \widehat{\theta}$.  Since this process can be performed in both directions (from clockwise paths to counterclockwise, and vice versa), inversely (by Lemma~\ref{lemma:unique-cycle}), the statement is proved.
\end{proof}

The preceding lemma allows us to define, for any closed walk, its \emph{backward walk}.

\begin{lemma}
If $\overline{R}$ is the conjugate of $R$, then $\Phi(\overline{R})$ and $\Phi(R)$ are backward isogeny cycles of each other.
\end{lemma}

\begin{proof}
The rim $\overline{R}$ consists of the set of edges and vertices traversed backward, where an isogeny is replaced with its dual.  From the proof of Lemma~\ref{lemma:dual-reverse}, this results in the backward isogeny cycle of $\Phi(R)$.
\end{proof}

\begin{corollary}\label{cor:mainbij-undir}
There is a bijection between undirected rims of length $r$ of $\mathcal{G}_{K,\ell}$ for all $K$, and undirected isogeny cycles of length $r$ in $\mathcal{G}_\ell$.
\end{corollary}

\begin{definition}
\label{defn:barbell}
A closed walk in $\mathcal{G}_\ell$ is called a \emph{barbell} if traversing it in the opposite direction as in Lemma~\ref{lemma:dual-reverse} results in the same walk, up to a choice of basepoint.
\end{definition}

An example is shown in Figure~\ref{fig:barbell}, justifying the name of `barbell.'

\begin{figure}
    \centering
    \includegraphics[width=0.8\textwidth]{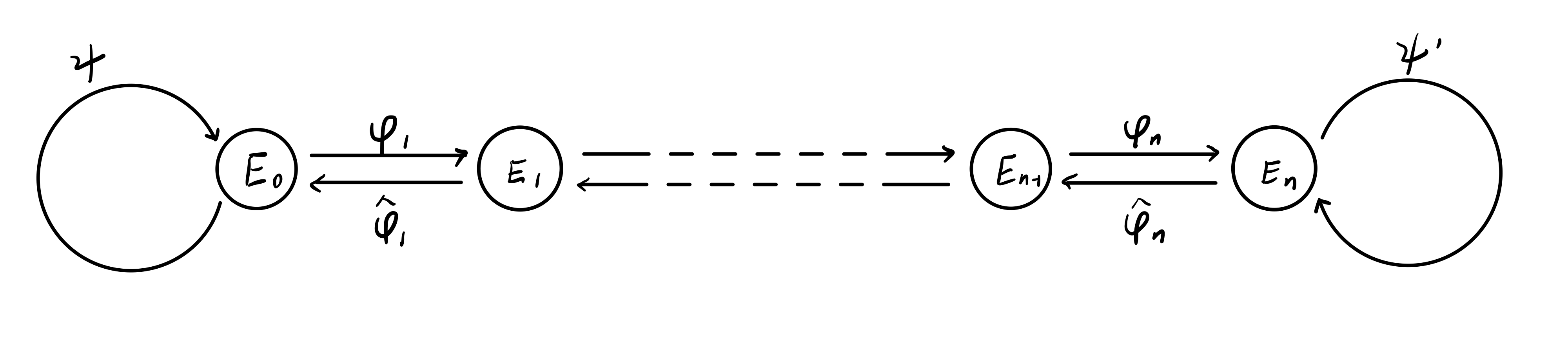}
    \caption{An example of an isogeny cycle that is a barbell (Definition~\ref{defn:barbell}).  If we walk in the direction of the arrows from $E_0$, we obtain $\psi \widehat{\varphi_1} \cdots \widehat{\varphi_n} \psi' \varphi_n \cdots \varphi_1$.  If we walk the cycle backwards (as defined in Lemma~\ref{lemma:dual-reverse}), we get $\widehat{\varphi_1} \cdots \widehat{\varphi_n} \widehat{\psi'} \varphi_n \cdots \varphi_1\widehat{\psi} $.  If we ignore the basepoint and have $\psi = \widehat{\psi}$ and $\psi' = \widehat{\psi'}$, then these are the same cyclic sequence of arrows, in which case this is a barbell.}
    \label{fig:barbell}
\end{figure}

\begin{lemma}
\label{lemma:barbellcount}
Barbells necessarily involve at least two loops and have an even number $r$ of vertices and edges.  The total number of barbells having $r$ edges is at most $\Np(\ell+1)\ell^{(r-2)/2}$.
\end{lemma}

\begin{proof}
Any barbell traverses a palindromic sequence of vertices of the following form:
\[
j_0, j_1, \ldots, j_{n-1}, j_n, j_n, j_{n-1}, \ldots, j_1, j_0.
\]
In particular, the `turnaround' points must be loops, because otherwise they would involve backtracking. This implies that the number of edges and vertices is even.  The count follows by counting the number of walks of length $(r-2)/2$ between two loops, and bounding the number of loops by the number of edges of the graph.
\end{proof}

The above bound on the number of barbells is far from tight.  There are congruence conditions modulo $p$ depending on $\ell$ which can rule out loops in $\mathcal{G}_\ell$ and hence rule out the existence of barbells.  (These conditions can be worked out by considering the specialization of the modular polynomial $\Phi_\ell(x,x) \bmod p$ and choosing~$p$ such that the roots are not supersingular.)
 
\begin{theorem}
\label{thm:barbells}
Under the bijection of Theorem~\ref{thm:mainbij-nobase}, the self-conjugate rims in $\mathcal{G}_{K,\ell}$ are in bijection with the barbells in $\mathcal{G}_\ell$.
\end{theorem}

\begin{proof}
If $R$ is a self-conjugate rim, then conjugation returns the same rim traversed in the opposite direction.  Hence $\Phi(R)$ must be a barbell.  Conversely, if $R$ is a rim such that $\Phi(R)$ is a barbell, then the rim $R'$ directed in the opposite direction maps to $\Phi(R)$, the cycle $\Phi(R)$ traversed in the opposite direction.  Hence $\Phi(R) = \Phi(R')$ and therefore $R$ is self-conjugate.
\end{proof}


\section{Supersingular curves oriented by $\mathcal{O}$}
\label{sec:classfieldtheory}

This section answers the question of when $\SSO = \rho(\Ell(\mathcal{O}))$.  Its purpose is to prepare for the next section, where we discuss the fibre sizes of the map taking a rim to its corresponding quadratic order.  In particular, we wish to understand the action of conjugation on $\SSO$.

Let $p$ be inert or ramified in $K$, and let $\mathcal{O}$ be an order of $K$. 
In this section, we study the set $\SSO$ in greater detail, using class field theory.   In particular, Theorem~\ref{thm:sso} characterizes when $\rho(\Ell(\mathcal{O})) = \SSO$, which holds precisely when $p$ ramifies in $K$.  This also characterizes the scenario when conjugation takes $\rho(\Ell(\mathcal{O}))$ back to itself.

Let $K = \QQ(\beta)$ be an imaginary quadratic field of discriminant $\Delta_K$ with a distinguished root $\beta$ of a fixed minimal polynomial $f_K(x)$ generating $K/\QQ$.  Let $\mathcal{O}$ be an order of $K$.
Let $L/K$ be the ring class field of $\mathcal{O}$.  By the general theory of class fields, the field $L$ has Galois group $\Gal(L/\QQ) \cong \Cl(\mathcal{O}) \rtimes \langle \tau \rangle$ where $\Cl(\mathcal{O}) \cong \Gal(L/K)$, and $\tau$ is the non-trivial automorphism of $K/\QQ$, and where $\sigma \tau = \tau \sigma^{-1}$ for any $\sigma \in \Cl(\mathcal{O})$ \cite[Lemma 9.3]{Cox_primesoftheform}.
Let $H_{\mathcal{O}}(x)$ be the Hilbert class polynomial (sometimes called the ring class polynomial) 
of the order $\mathcal{O}$.
Define
\[
\mathcal{L}_K := \{ j \in L : H_\mathcal{O}(j) = 0 \}, \quad
\mathcal{L}_\QQ := \{ (j, \alpha) \in L \times K : H_\mathcal{O}(j) = 0 = f_K(\alpha) \}.
\]

The Galois groups $\Gal(L/K)$ and $\Gal(L/\QQ)$ act freely and transitively on $\mathcal{L}_K$ and $\mathcal{L}_\QQ$, respectively, via $j \mapsto j^\sigma$ and $(j, \alpha) \mapsto (j,\alpha)^\sigma = (j^\sigma, \alpha^\sigma)$, respectively.  We have $\# \mathcal{L}_K = \hO = \#\Gal(L/K)$ and $\#\mathcal{L}_\QQ = 2h_\mathcal{O} = \#\Gal(L/\QQ)$.

It is well known that $\mathcal{L}_K$ is in bijection with $\Ell(\mathcal{O})$, the isomorphism classes of elliptic curves over $\CC$ with complex multiplication by $\mathcal{O}$, by the map $j \mapsto E(j)$ taking a $j$-invariant to its isomorphism class of curves.  (We will abuse notation by denoting an element of $\Ell(\mathcal{O})$ as a curve $E$ without equivalence class brackets $[E]$.)  This endows $\Ell(\mathcal{O})$ with an action of $\Gal(L/K)$.

We now give an analogous result for $\mathcal{L}_\QQ$.
 Given $E \in\Ell(\mathcal{O})$, there are two isomorphisms $\iota \colon \mathcal{O} \rightarrow \End_\CC(E)$ that differ by conjugation.  Denote these by $\iota_{E}$ and $\overline{\iota}_{E}$, where $\iota_{E}$ is normalized so that $\iota_E(\alpha)^* \omega_E = \alpha \omega_E$  for the invariant differential $\omega_E$ of $E$ and any $\alpha \in \mathcal{O}$.  Let $\Ell^*(\mathcal{O})$ denote the set of cardinality $2\hO$ of pairs consisting of a curve from $\Ell(\mathcal{O})$ together with one of the two orientations:
 \[
 \Ell^*(\mathcal{O}) := \{ (E, \iota)\colon E \in \Ell(\mathcal{O}), \, \iota\colon \mathcal{O} \xrightarrow{\sim} \End(E) \}.
 \]

 \begin{proposition}
 There is a bijection of $\mathcal{L}_\QQ$ with $\Ell^*(\mathcal{O})$ via
\[
(j,\delta) \mapsto 
\left(
E=E(j), \left\{
\begin{array}{ll}
\iota_E & \text{ if }\delta = \beta \\
\overline{\iota}_E & \text{ if }\delta = \overline{\beta}
\end{array} \right.
\right).
\]
This endows $\Ell^*(\mathcal{O})$ with an action of $\Gal(L/\QQ)$.  
\end{proposition}
The bijection is evidently non-canonical; we could make the opposite choice for $\iota_E$.

Choose a finite extension $L'/L$ such that there is a prime $\mathfrak{p}$ of $L'$ above $p$ for which all curves in $\Ell(\mathcal{O})$ have good reduction at $\mathfrak{p}$.
Let $\rho_\mathfrak{p}$ 
denote reduction modulo $\mathfrak{p}$, and put $\tilde{E} = \rho_{\mathfrak{p}}(E)$ for brevity.  Reduction provides a map
\[
\rho \colon \Ell^*(\mathcal{O}) \rightarrow \SSO
\text{ given by }
(E, \iota_E) \mapsto (\tilde{E}, \iota_{E} \bmod{ \mathfrak{p}} )\text{, and }
(E, \overline{\iota}_E) \mapsto (\tilde{E}, \overline{\iota}_{E} \bmod{\mathfrak{p}}  ).
\]
There is a natural embedding of $\Ell(\mathcal{O})$ in $\Ell^*(\mathcal{O})$ given by $E \mapsto (E, \iota_E)$ (i.e.\ choosing the normalized isomorphism), and to ease notation, we will consider $\Ell(\mathcal{O})$ a subset of $\Ell^*(\mathcal{O})$.  Consider the following diagram.
\[
\xymatrix@R=1em{
\Cl(\mathcal{O}) \cong \Gal(L/K) \ \ar@{^{(}->}[r] & \Gal(L/\QQ) \cong \Cl(\mathcal{O}) \rtimes \langle \tau \rangle \\
\mathcal{L}_K\ar@(ul,ur)[] \ar^{\text{\Large $\wr$}}[d] &  \mathcal{L}_\QQ\ar@(ul,ur)[] \ar^{\text{\Large $\wr$}}[d]
\\
 \Ell(\mathcal{O}) \ \ar@{^{(}->}[r] \ar[d]^{\rho_\mathfrak{p}} & \Ell^*(\mathcal{O})  \ \ar[d]^{\rho_\mathfrak{p}}  \\
\rho_\mathfrak{p}(\Ell(\mathcal{O})) \ \ar@{^{(}->}[r] & \SSO
}
\]
The actions of the Galois groups shown are free and transitive.  The inclusion map $\Ell(\mathcal{O}) \hookrightarrow \Ell^*(\mathcal{O})$ is equivariant with respect to the action of $\Gal(L/K) \subseteq \Gal(L/\QQ)$.

In the vertical arrows of the lower half of the diagram, we see reduction modulo $\mathfrak{p}$.  The square in this portion commutes.  The descending arrow on the left is a bijection.  

First, we consider the easier unramified case, and show that $\rho_\mathfrak{p}(\Ell(\mathcal{O})) \neq \SSO$.  The map $(E,\iota) \mapsto (E,\overline{\iota})$ on $\SSO$ will be called \emph{conjugation}.

\begin{proposition}
\label{prop:sso-unram}
Let $p$ be inert in $K$.  Then the two isomorphism classes of $\mathcal{O}$-oriented curves represented by $(E,\iota)$ and $(E,\overline{\iota})$ cannot both belong to $\rho_\mathfrak{p}(\Ell(\mathcal{O}))$, so $\rho_\mathfrak{p}(\Ell(\mathcal{O})) \neq \SSO$, in which case $2\# \rho_\mathfrak{p}(\Ell(\mathcal{O})) =\# \SSO$. 
\end{proposition}

\begin{proof}
Since $p$ is unramified, there exists $\alpha_0 \in \mathcal{O}$ such that $\iota(\alpha_0 - \overline{\alpha_0})$ is separable.  
Suppose $E_1, E_2 \in \Ell(\mathcal{O})$.  Then there are isomorphisms $[ \cdot ]_{E_i} \colon \mathcal{O} \rightarrow \End(E_i)$, where we use the normalization of Section~\ref{sec:back:sso}, namely we choose the isomorphisms so that $[\alpha]_{E_i}^* \omega_{E_i} = \alpha \omega_{E_i}$ for all $\alpha \in \mathcal{O}$.  
With notation as in Section~\ref{sec:back:sso}, suppose that the reduction $(\tilde{E}_1,\iota_{E_1})$ of $(E_1, [\cdot]_{E_1})$ modulo $\mathfrak{p}$ is in the isomorphism class of $(E, \iota)$ and the reduction $(\tilde{E}_2, \iota_{E_2})$ of $(E_2, [\cdot]_{E_2})$ is in the isomorphism class of $(E, \overline{\iota})$.  

Let us first assume for simplicity that $(\tilde{E_1},\iota_{E_1}) = (E,\iota)$ and $(\tilde{E_2}, \iota_{E_2}) = (E, \overline{\iota})$, in which case the proof is more transparent.  Then, as discussed in Section~\ref{sec:back:sso}, the normalization is preserved by reduction, which is to say,
\[
\iota(\alpha)^* \omega_E 
= \alpha \omega_E 
=  \overline{\iota}(\alpha)^* \omega_{E } 
\]
for all $\alpha \in \mathcal{O}$.  But then
\[
\iota(\alpha)^* \omega_E 
=  \overline{\iota}(\alpha)^* \omega_{E } 
= \iota(\overline{\alpha})^* \omega_E,
\]
for all $\alpha \in \mathcal{O}$, which implies $\iota(\alpha -\overline{\alpha})^* \omega_E = 0$ for all $\alpha$, a contradiction for $\alpha = \alpha_0$ since $\iota(\alpha_0 - \overline{\alpha_0})$ is separable.

In general, however, the map $\rho$ is only defined on isomorphism classes, prompting a certain amount of bookkeeping.  In particular, assume that $(\tilde{E_1}, \iota_{E_1}) = (\phi_1 E, \phi_1 \iota \phi_1^{-1})$, and $(\tilde{E_2}, \iota_{E_2}) = (\phi_2 E, \phi_2 \overline{\iota} \phi_2^{-1})$ where $\phi_i \colon E \rightarrow \phi_i E$, $i=1,2$, are isomorphisms.  
Then for all $\alpha \in \mathcal{O}$,
\[
\alpha \omega_E 
= \alpha \phi_2^* \omega_{\phi_2 E}
=  \phi_2^* \alpha \omega_{\phi_2 E}
= \phi_2^* (\phi_2 \overline{\iota}(\alpha) \phi_2^{-1})^* \omega_{\phi_2 E } 
= \overline{\iota}(\alpha)^* \phi_2^* \omega_{\phi_2 E}
= \overline{\iota}(\alpha)^* \omega_E
= \iota(\overline{\alpha})^* \omega_E
\]
and by a similar computation, $\iota(\alpha)^* \omega_E = \alpha \omega_E$.  The rest of the proof is as before.
\end{proof}

Next we consider the ramified case, where we show that $\rho_\mathfrak{p}(\Ell(\mathcal{O})) = \SSO$.

\begin{proposition}
\label{prop:tau-on-sso}
Suppose $p$ is ramified in $K$.
Then the action of $\Gal(L/\QQ)$ on $\mathcal{L}_\QQ$ induces an action on~$\SSO$.  Under this action, $\tau$ acts by conjugation, i.e.\ $\tau \cdot (E, \iota) = (E, \overline{\iota})$.
\end{proposition}

\begin{proof}
Define an action on $\SSO$ by $\Cl(\mathcal{O})$ as in Definition~\ref{def:idealaction}, and by $\langle \tau \rangle$ as $\tau \cdot (E, \iota) = (E, \overline{\iota})$.  Let $\mathfrak{P} = \mathfrak{p} \cap L$ be the restriction of $\mathfrak{p}$ to $L$. 
Since $p$ is ramified in $K$, $\tau$ is in the inertia group of $\mathfrak{P}$, which implies that $j^\tau \equiv j \pmod{\mathfrak{P}}$ and hence also modulo $\mathfrak{p}$, where $j = j(E)$.  Thus 
\begin{align*}
\rho_\mathfrak{p}( \tau \cdot(E_j,\iota) )
& = \rho_\mathfrak{p}( (E_{j^\tau}, \overline{\iota} ) )
= (E_{j^\tau \bmod{\mathfrak{p}}}, \overline{\iota} \bmod{\mathfrak{p}} )
= (E_{j \bmod{\mathfrak{p}}}, \overline{\iota} \bmod{\mathfrak{p}}) \\
&= \tau \cdot (E_{j \bmod{\mathfrak{p}}}, \iota \bmod{\mathfrak{p}})
= \tau \cdot \rho_\mathfrak{p}((E_j, \iota)).
\end{align*}
The map $\rho$ is also known to satisfy $\rho_\mathfrak{p}( \sigma \cdot (E,\iota)) = \sigma \cdot (\rho_\mathfrak{p}( (E,\iota)))$ for $\sigma \in \Cl(\mathcal{O})$ \cite[Proof of Proposition 3.6]{onuki2021}.   Therefore, the actions described on $\SSO$ give a well-defined action of $\Gal(L/\QQ)$ with respect to which~$\rho$ is equivariant.
\end{proof}

\begin{theorem}
\label{thm:sso}
The prime $p$ ramifies in $K$ if and only if $\rho_\mathfrak{p}(\Ell(\mathcal{O})) = \SSO$ if and only if conjugation takes $\rho_\mathfrak{p}(\Ell(\mathcal{O}))$ to itself. 
\end{theorem}

\begin{proof}
If $\rho_\mathfrak{p}(\Ell(\mathcal{O})) = \SSO$, then $p$ ramifies by Proposition~\ref{prop:sso-unram}.
Conversely, if $p$ ramifies in $K$, then there is some $\theta \in \iota(\mathcal{O}) \subseteq \End(E)$ such that $\theta^2 = -np$ for some integer $n$ coprime to $p$.  The endomorphism $\theta$ is a horizontal isogeny with respect to $\mathcal{O}$ (as an element of $\iota(\mathcal{O})$).  Furthermore, $\theta$ must factor as
\[
E \xrightarrow{\pi_p} E^{(p)} \xrightarrow{\phi} E,
\]
where $\phi$ is of degree $n$ and separable.  But since $\theta \in \iota(\mathcal{O})$, $\theta_* \iota = \iota$ and we have the following chain of $K$-oriented isogenies:
\[
(E,\iota) \xrightarrow{\pi_p} (E^{(p)}, \iota^{(p)}) \xrightarrow{\phi} (E, \iota).
\]
The maps shown above are horizontal isogenies ($\pi_p$ is horizontal and so is $\theta$, hence $\phi$ is also).  We wish to show that $\phi$ represents the action of an element of $\Cl(\mathcal{O})$, up to multiplication by an integer.  We can factor~$n$ into prime powers and prove it for each one, so without loss of generality, we assume that $n$ is a prime power~$q^k$, for $q \ne p$.  

 The map $\phi$ is given by a path in the oriented $q$-isogeny graph determined by successive kernels, and as it is horizontal, this path must have as many descending as ascending steps.  Using the volcano structure of the graph, all the ascending and descending edges can be paired up into consecutive pairs $\phi_2 \circ \phi_1$ where $\phi_1$ descends and $\phi_2$ ascends.  Furthermore, $\phi_2 = \varphi_1 \widehat{\phi_1} \varphi_2$ for some automorphisms $\varphi_1, \varphi_2$.  Then $\varphi_1^{-1} \phi_2 \phi_1 = \widehat{\phi_1} \varphi_2 \phi_1$ has trace and norm divisible by $q$, hence it is divisible by $[q]$.  So $\phi_2\phi_1 = \varphi [q]$ for some automorphism~$\varphi$.  Since we only care about $\phi$ up to multiplication by an integer, we can replace $\phi_2 \phi_1$ with $\varphi$, which is horizontal.  Repeating this process leaves a purely horizontal path, which must therefore represent an action of an element of $\Cl(\mathcal{O})$
(since $q \ne p$, $\Cl(\mathcal{O})$ acts on $\SSO$, so any horizontal $q$-isogeny is the action of an element of the class group).
It follows that $(E^{(p)}, \iota^{(p)}) \in \Cl(\mathcal{O}) \cdot (E,\iota)$, which in turn implies that $\rho_\mathfrak{p}(\Ell(\mathcal{O})) = \SSO$ by \cite[Propositions  3.3 and 3.4]{onuki2021}.

Finally, we show that $\rho_\mathfrak{p}(\Ell(\mathcal{O})) = \SSO$ if and only if conjugation takes $\rho_\mathfrak{p}(\Ell(\mathcal{O}))$ to itself.  Since $p$ is ramified, by Proposition~\ref{prop:tau-on-sso}, the action of $\Gal(L/\QQ)$ is defined on $\SSO$ and $\tau$ acts by conjugation, fixing the $j$-invariant $j(E)$ modulo $\mathfrak{p}$.
First, note that if $(E,\iota)$ and $(E,\overline{\iota})$ are related by an element of the class group, then the entire orbit $\Cl(\mathcal{O}) \cdot (E,\iota)$ consists of elements related to their conjugates by the class group, and therefore, conjugation takes $\rho_\mathfrak{p}(\Ell(\mathcal{O}))$ into itself.  Therefore conjugation takes $\rho_\mathfrak{p}(\Ell(\mathcal{O}))$ to itself, or else maps it to $\SSO \setminus \rho_\mathfrak{p}(\Ell(\mathcal{O}))$.
We know that $\# \rho_\mathfrak{p}(\Ell(\mathcal{O})) = h_\mathcal{O}$, and $\# \SSO \in \{ \hO, 2\hO \}$ (Section~\ref{sec:back:sso}).  But $\# \SSO = 2\hO$ if and only if $\rho_\mathfrak{p} \colon \Ell^*(\mathcal{O}) \rightarrow \SSO$ is an isomorphism if and only if $\Gal(L/\QQ)$ acts freely and transitively on $\SSO$ (calling upon Proposition~\ref{prop:tau-on-sso} again).  But $\rho_\mathfrak{p}(\Ell(\mathcal{O}))$ forms one $\Cl(\mathcal{O})$-orbit, so the only way the action can be free and transitive is if conjugation takes $\rho_\mathfrak{p}(\Ell(\mathcal{O}))$ outside itself.
\end{proof}

\section{Isogeny cycles and their associated orders}
\label{sec:cycle-order}

The following theorem describes a way to relate isogeny cycles in $\mathcal{C}_r$ to imaginary quadratic orders.  Define
\[
\mathcal{I}_{r} =\left\{ \text{imaginary quadratic orders $\mathcal{O}$} : \substack{ \mbox{$p$ does not split in the field containing $\mathcal{O}$} \\ 
\mbox{$p$ does not divide the conductor of $\mathcal{O}$} \\ \mbox{$\mathcal{O}$ is an $\ell$-fundamental order,} \\
\mbox{$(\ell) = \mathfrak{l}\bar{\mathfrak{l}}$ splits in $\mathcal{O}$,} \\ \mbox{and $[\mathfrak{l}]$ has order $r$ in $\Cl(\mathcal{O})$}. }
\right\}.
\]
These are exactly the orders for which $\SSO$ is non-empty and the permutation on the vertices arising from the action of a prime ideal in $\mathcal{O}$ above $\ell$ has a decomposition consisting of cycles of length $r$ only.
The following theorem describes how often each order in $\mathcal{I}_r$ is obtained from an isogeny cycle in $\mathcal{G}_\ell$ via the bijection of Theorem~\ref{thm:mainbij-nobase}.  If $p$ is inert, one expects to obtain $2\hO/r$ directed isogeny cycles of size $r$ which give rise to an order~$\mathcal{O}$, since $\#\SSO = 2\hO$ and the rims are identified in pairs by conjugation, and given a direction.  However, in the case that $p$ ramifies in the field $K$ containing $\mathcal{O}$, the exact count may differ from this.

Let $\mathcal{O}$ be a quadratic order with fraction field $K$.  Let $L/K$ be the ring class field of $\mathcal{O}$.  Let $g_\mathcal{O}$ be the number of genera in $\Cl(\mathcal{O})$.  Let $\sigma_\ell := \Frob_\ell \in \Gal(L/K)$ have order $r$.  Define $C_\ell := \langle \sigma_\ell \rangle$ (so $\# C_\ell = r$).

\begin{theorem}
\label{thm:mainbij-fibre}
Let $r > 2$.
The map $\mathcal{R}_r \rightarrow \mathcal{I}_r$ which takes a rim to the order $\mathcal{O}$ with respect to which all its vertices are primitively $\mathcal{O}$-oriented as in the bijection of Theorem~\ref{thm:mainbij-nobase} has fibres of size $2h_\mathcal{O}/r$ or $4\hO/r$, depending on whether or not $p$ ramifies in the field $K$ containing~$\mathcal{O}$, respectively.  The map factors through $\mathcal{R}_r/{\sim}$, so the induced map $\mathcal{C}_r \rightarrow \mathcal{I}_r$ has a fibre of size $\epsilon_{\mathcal{O},\ell} \hO/r$ over $\mathcal{O}$ for some $1 \le \epsilon_{\mathcal{O},\ell} \le 2$. 
Furthermore,
   
   \begin{enumerate}
       \item Suppose $p$ is inert $K$. Then $\epsilon_{\mathcal{O},\ell} = 2$.
       \item Suppose $p$ ramifies in $K$.  Then there are two possibilities:
       \begin{enumerate}
           \item There are no self-conjugate rims in $\rho(\Ell(\mathcal{O}))$. Then $\epsilon_{\mathcal{O},\ell} = 1$. 
           \item
       There are self-conjugate rims  in $\rho(\Ell(\mathcal{O}))$, in which case $r$ is even, and 
\[
\epsilon_{\mathcal{O},\ell} := 1 +
 \frac{ r g_\mathcal{O} } { 2 h_\mathcal{O} }.
\]
In this case, $1 < \epsilon_{\mathcal{O},\ell} \le 2$. 
   \end{enumerate}
   \end{enumerate}
\end{theorem}

The following commutative diagram  illustrates the theorem:
\begin{equation}
    \label{eqn:mainbij-fibre}
\xymatrix@C=5em{
    \mathcal{C}_r \ar[r]^\cong_{Thm.~\ref{thm:mainbij-nobase}}  & \mathcal{R}_r/{\sim} \ar[dr] & \mathcal{R}_r  \ar[l]_{{\substack{\text{fibres} \\ \text{of size} \\ \text{$1$ or $2$}}}} \ar[d]^{\substack{\text{fibres of size}\\ \text{$2h_\mathcal{O}/r$} \\ \text{or} \\ \text{$4\hO/r$} }} \\
                           &  & \mathcal{I}_r }
\end{equation}

\begin{proof}

Define a map $\mathcal{R}_r \rightarrow \mathcal{I}_r$ taking a rim to the associated order as described in the statement of the theorem. The size of the fibre above $\mathcal{O}$ is $2\#\SSO/r$ (the factor of $2$ arises because rims are directed), which is $2\hO/r$ or $4\hO/r$ in the ramified and inert cases, respectively, by Theorem~\ref{thm:sso} (the conditions on $p$ in the definition of $\mathcal{I}_r$ imply $\SSO$ is nonempty by Proposition~\ref{prop:sso-empty}).  Since a conjugate rim is associated to the same order as the original, the map factors through $\mathcal{R}_r/{\sim}$.  The fibres of $\mathcal{R}_r \rightarrow \mathcal{R}_r/{\sim}$ are of size $1$ or $2$ (depending upon whether the rim in question was self-conjugate).

In the inert case, $\#\SSO = 2\hO$ by Theorem~\ref{thm:sso}.  The fibres of $\mathcal{R}_r \rightarrow \mathcal{R}_r /{\sim}$ are uniformly of size $2$ for rims in $\SSO$ (since conjugation swaps $\rho(\Ell(\mathcal{O}))$ and its complement in $\SSO$ by Theorem~\ref{thm:sso}).  The map $\mathcal{R}_r \rightarrow \mathcal{I}_r$ has a fibre of size $4\hO/r$ over $\mathcal{O}$.  Hence the map $\mathcal{C}_r \rightarrow \mathcal{I}_r$ has a fibre of size $2\hO/r$ above $\mathcal{O}$.

Now we consider the ramified case.  Call upon the conventions and set-up of Section~\ref{sec:classfieldtheory}.   Let $G = \Gal(L/\QQ)$. We consider the stabilizers $\Stab_G( (E, \iota) )$ of elements of $\SSO$.  
All stabilizers are of size 2 since $\# G = 2\hO = 2\#\SSO$ by Theorem~\ref{thm:sso}, and no stabilizers are contained in $\Cl(\mathcal{O})$ (whose action is free).  Therefore, for some $(E,\iota)$, we have $\Stab_G( (E, \iota) ) = \{ id, \sigma \tau \}$, for some $\sigma \in \Cl(\mathcal{O}) \cong \Gal(L/K)$. The set of stabilizers is the orbit of $\Stab_G( (E, \iota) )$ under conjugation by elements of $G \cong \Cl(\mathcal{O}) \rtimes \langle \tau \rangle$.  For $\eta \in \Cl(\mathcal{O})$, we have
\[
\eta \Stab_G( (E, \iota) ) \eta^{-1} = \{ id, \eta \sigma \tau \eta^{-1} \} = \{ id, \eta^2 \sigma \tau \}.
\]

Also $\tau \Stab_G( (E, \iota) ) \tau^{-1} = \{ id, \tau \sigma \} $   $= \{ id, \sigma^{-2} \sigma \tau \}$. Thus, the nontrivial elements of the stabilizers, taken together, form a set $\Cl(\mathcal{O})^2 \sigma \tau$.  Set the notation $G_\sigma := \Cl(\mathcal{O})^2 \sigma$. Then $\sigma \in \Cl(\mathcal{O}) \setminus \Cl(\mathcal{O})^2$, as $\tau$ does not stabilize any pairs $(E, \iota)$. Hence $G_{\sigma}$ is a non-principal genus of $\Cl(\mathcal{O})$. 

An element $(E,\iota)$ is in the same $\ell$-isogeny rim as its conjugate $(E,\overline{\iota})$ if and only if $\sigma_\ell^k \tau \in \Stab_G( (E, \iota) )$ for some integer $k$ if and only if $C_\ell \tau \cap \Stab_G( (E,\iota) ) = C_\ell \tau \cap G_\sigma\tau \neq \emptyset$.  If $C_\ell \cap G_\sigma = \emptyset$, then there are no stabilizers intersecting $C_\ell \tau$, so there are no self-conjugate rims.  Since $\rho(\Ell(\mathcal{O}))$ is taken to itself by conjugation, in this case all rims pair up in conjugate pairs.  Therefore the fibres are of size $\hO/r$.

Suppose there are self-conjugate rims, in which case $\sigma_\ell^k \in G_\sigma$ for some integer $k$. Then $k$ must be odd, as otherwise $\sigma_\ell^k \in G_\sigma \cap \Cl(\mathcal{O})^2 = \emptyset$. It follows that $\sigma_\ell \in G_\sigma$. A similar argument forces $r$ to be even, as otherwise $\sigma_\ell = \sigma_\ell^{r+1} \in G_\sigma \cap \Cl(\mathcal{O})^2 = \emptyset$.

There are $\# G_\sigma = \# \Cl(\mathcal{O})^2 = \hO/g_\mathcal{O}$ distinct stabilizers, of which $\# G_\sigma \cap C_\ell$ belong to elements in the same rim as their conjugates.  By the same logic as before, using the fact that $r$ is even, the intersection of $G_\sigma$ with $C_\ell$ consists of the odd powers of $\sigma_\ell$, and is therefore of size $\# C_\ell/2 = r/2$. Since all rims are the same size, this means the proportion of rims which are self-conjugate is 
\[
c := \frac{ r/2 } { \# \Cl(\mathcal{O})^2 } = \frac{ rg_\mathcal{O} }{2 h_\mathcal{O} }.
\]
This represents the proportion of rims (equivalently, of vertices) of $\SSO$ that are self-conjugate.  This proportion is the same when we restrict to $\rho(\Ell(\mathcal{O}))$.

The proportion $1-c$ of non-self-conjugate rims is double-counted when we wish to count $\ell$-isogeny cycles in $\mathcal{G}_\ell$, so we have
\[
\epsilon_{\mathcal{O},\ell} = 2c + (1-c) = 1 + c.
\]
For the final inequality, note that in the ramified case, the fibres of $\mathcal{R}_r \rightarrow \mathcal{I}_r$ are of size $2\hO/r$, and reference diagram \eqref{eqn:mainbij-fibre}.
\end{proof}

\begin{example}
\label{ex:241}
The following example illustrates the most complicated type of fibre.  Consider $\mathcal{G}_{\ell}$ with $p=241$, $\ell=11$, and the quadratic order $\mathcal{O}$ of discriminant $-4p$, namely $\mathcal{O} = \ZZ[\sqrt{-p}]$.   Since $p$ ramifies, $\SSO = \rho(\Ell(\mathcal{O}))$ by Theorem~\ref{thm:sso}.  Let $r=4$.  The class group of $\mathcal{O}$ is cyclic of order $12$.  Therefore, there are $12/4 = 3$ rims of size $4$ in $\SSO$ under the action of a prime ideal $\mathfrak{l}$ in $\mathcal{O}$ above $\ell$.  That means $6$ directed rims.  Using an isomorphism $(\Cl(\mathcal{O}), \cdot) \cong (\ZZ/12\ZZ,+)$ which takes $\sigma_\ell$ of order $4$ to the element $3$, we have $\Cl(\mathcal{O})^2 = \{0,2,4,6,8,10\}$ and $C_\ell = \{0,3,6,9\}$.  So $1/3$ of the elements of $\Cl(\mathcal{O})^2$ are in $C_\ell$.  This implies that $1/3$ of the total directed rims are self-conjugate:  two of the six (one of the three undirected rims).  Using the formula of Theorem~\ref{thm:mainbij-fibre}, with $g_\mathcal{O} = 2$, we obtain
\[
\epsilon_{\mathcal{O},\ell} = 1+ \frac{4 \cdot 2}{2 \cdot 12} = \frac{4}{3},
\]
which gives a fibre size of $\frac{4}{3} \cdot \frac{12}{4} = 4$ for the map $\mathcal{C}_4 \rightarrow \mathcal{I}_4$ above $\mathcal{O}$. In other words, there should be four directed isogeny cycles of size $4$ (equivalently, two undirected isogeny cycles) in $\mathcal{G}_{11}$ which give rise to $\mathcal{O}$.

The set $\rho(\Ell(\mathcal{O}))$ has three undirected length $4$ isogeny cycles: $(64, 93, 216, 240)$ (twice, conjugate to one another) and $(8,8,28,28)$ (self-conjugate).  The subgraph of $\mathcal{G}_{11}$ of $\FF_p$-points is shown in Figure~\ref{fig:241}, where we observe that there is only one $4$-cycle through $(964,93,216,240)$, and that $(8,8,28,28)$ collapses to a `barbell shaped' isogeny cycle.  The commutative diagram \eqref{eqn:mainbij-fibre}, if labelled with the (undirected) elements above $\ZZ[\sqrt{-p}]$, looks something like the following (please compare to Figure~\ref{fig:241}):
\[
\xymatrix@C=5em{
   { {\Huge{\substack{ \; \\ \diamond \vspace{0.11em}\\ \vspace{-0.23em} \circ \\  \vspace{-0.13em} \shortmid \\ \circ }}}}
     \ar@<.5em>[r] \ar@<-.5em>[r]  & 
     { \Huge{\substack{ \square \\ \square }}}
     \ar[dr] & 
      { \Huge{\substack{ \square \; \square \\ \square }}}
     \ar@<.5em>[l] \ar@<-.5em>[l] \ar[d] \\
                           &  & \ZZ[\sqrt{-p}] }
\]
\end{example}

\begin{figure}
    \centering
    \includegraphics[width=0.4\textwidth]{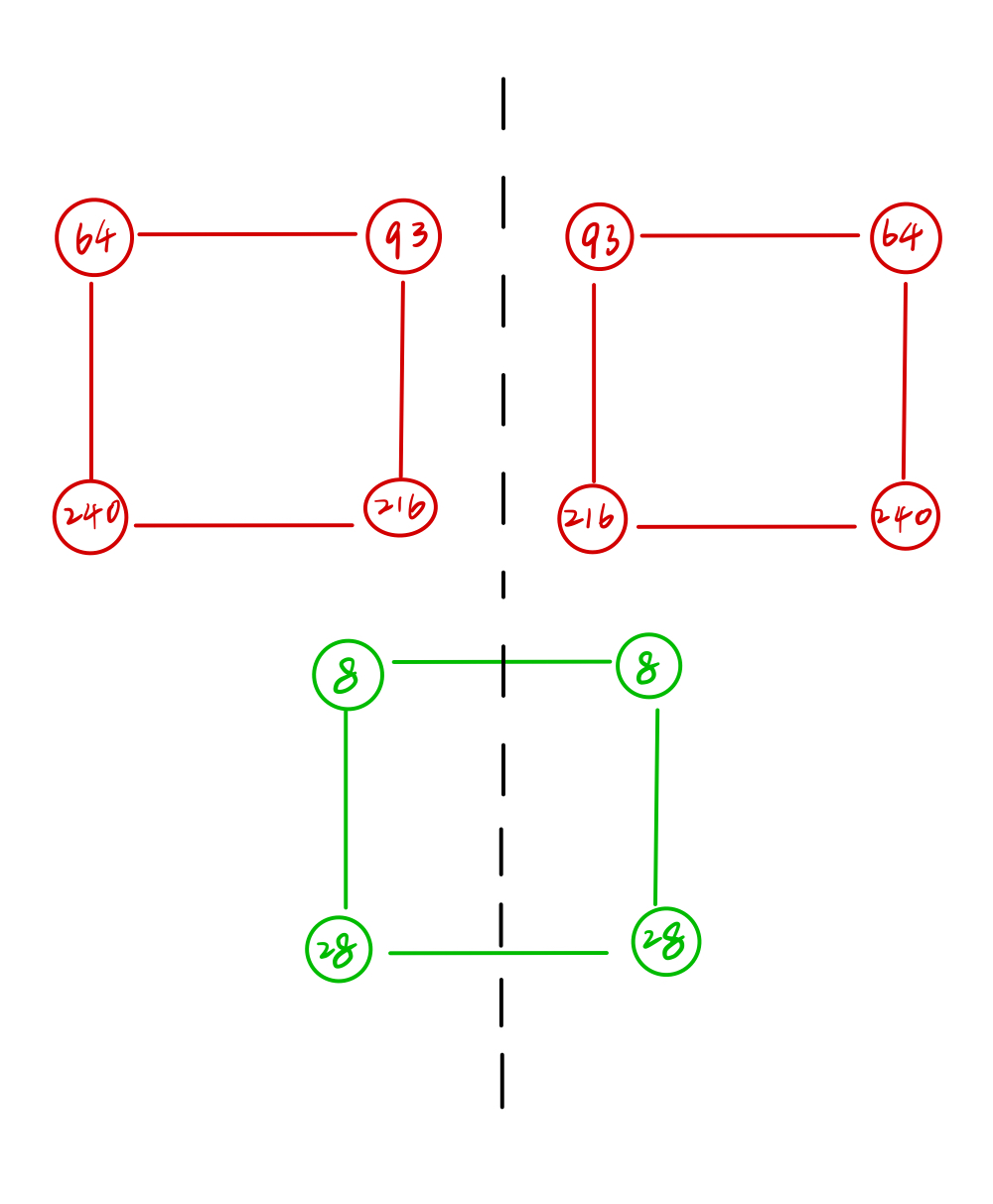}
    \includegraphics[width=0.4\textwidth]{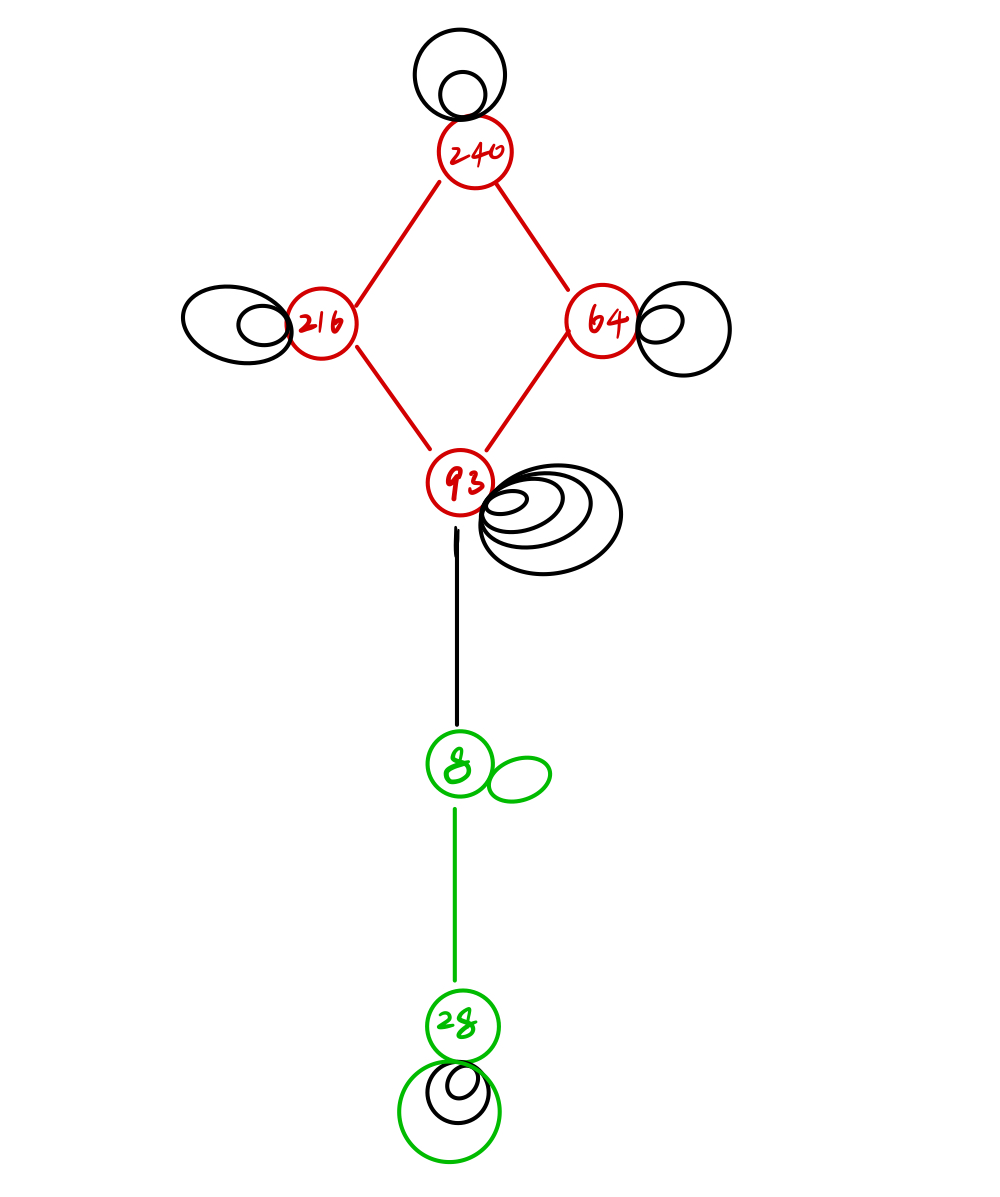}
    \caption{An illustration of Example~\ref{ex:241} with $p=241$, $\ell=11$ and $\mathcal{O} = \ZZ[\sqrt{-p}]$, for which $h_\mathcal{O}=12$.  On the left, the rims of $\mathcal{G}_{\QQ(\sqrt{-p}),11}$ drawn in such a way that conjugation is a mirror symmetry through the dotted line.  On the right, the subgraph of $\FF_p$-vertices of $\mathcal{G}_{11}$.}
    \label{fig:241}
\end{figure}

\section{An extended example}
\label{sec:examples}

In this section, we give an example that demonstrates our bijection (Theorem~\ref{thm:mainbij-nobase}). We choose $p = 179,\,\ell = 2$, and we consider $r = 3,\,4,\,5$ and $6$.  Computing the set $\mathcal{R}_r$ involves first finding all $\ell$-fundamental imaginary quadratic orders $\mathcal{O}$ such that $\mathcal{O}$ embeds into $B_{p,\infty}$, the ideal $(\ell) = \mathfrak{l}\bar{\mathfrak{l}}$ splits in $\mathcal{O}$ and $[\mathfrak{l}]$ has order~$r$ in $\Cl(\mathcal{O})$. A complete list of such orders $\mathcal{O}$ and prime ideals $\mathfrak{l},\,\bar{\mathfrak{l}}$ can be calculated as follows:
\begin{enumerate}
    \item list all representations of $\ell^r$ by norm forms of sufficiently small discriminant;
    \item each representation gives rise to an imaginary quadratic element $\alpha$ such that $N(\alpha) = \ell^r$ (up to sign and conjugation);
    \item check the splitting behaviour of $p$ in $\QQ(\alpha)$, discarding $\alpha$ if $p$ splits in $\QQ(\alpha)$;
    \item consider all $\ell$-fundamental orders in $\QQ(\alpha)$ that contain $\alpha$ and check whether the ideal class $[\mathfrak{l}]$ has order $r$ in $\Cl(\mathcal{O})$, obtaining a list $L_r$ of desirable orders $\mathcal{O}$.
\end{enumerate}

As an example, consider representing $2^3 = 8$.  We need only check discriminants $\Delta$ with $|\Delta| \le 4 \cdot 8 = 32$.  We find two representations:  $1^2 + 31 \cdot 1^2 = 4 \cdot 8$ and $3^2 + 23 \cdot 1^2 = 4 \cdot 8$.  These gives rise to two possible $\alpha$ up to conjugation and sign:
\[
\frac{\pm 1 \pm \sqrt{-31}}{2}, \frac{\pm 3 \pm\sqrt{-23}}{2}.
\]
However, $179$ splits in $\QQ(\sqrt{-23})$ and not in $\QQ(\sqrt{-31})$.  Hence we discard the second $\alpha$.  We find that $\mathfrak{l} = (2, \frac{1+\sqrt{-31}}{2})$ is not principal and $\mathfrak{l}^3 = (\alpha)$.  Hence $L_3 = \left\{ \ZZ\left[\frac{1+\sqrt{-31}}{2}  \right] \right\}$.

Each element $\alpha$ corresponds to an orientation $\iota_\alpha$ (up to conjugation), which gives a rim (up to conjugation and direction), as in the map $\Theta$ of Lemma~\ref{lem:theta}. 
Such a list therefore gives the number of undirected rims: $\#L_r = \frac{1}{2} \#(\mathcal{R}_r/{\sim})$. 

Theorem~\ref{thm:mainbij-nobase} concerns directed isogeny cycles and rims, but in this example we will ignore the issue of direction, which is allowed by Corollary~\ref{cor:mainbij-undir}.  Therefore, we use $L_r$ as a proxy for $\mathcal{R}_r/{\sim}$ for computational purposes.
We compare $L_r$ with $\mathcal{C}_r$, verifying the bijection theorem.

For each $r \in \{ 3,4,5,6\}$, we compute the set $\mathcal{C}_r$ on the supersingular 2-isogeny graph by reference to Figure~\ref{fig:241} showing $\mathcal{G}_2$.  We will actually consider $\mathcal{C}_r'$, the set of isogeny cycles of $\mathcal{C}_r$ but forgetting direction, allowed by Corollary~\ref{cor:mainbij-undir}.  Then we determine the associated endomorphisms by composing around the cycles as in the map $\Theta$; we do this by computing the trace of the composed endomorphism $\theta$ using the formulas $\theta + \widehat{\theta}$ or $1 + \deg(\theta) - \deg(1-\theta)$ in SageMath \cite{sagemath}.

For example, there is only one isogeny cycle of length $3$ (in the sense of Theorem~\ref{thm:mainbij-nobase}, but ignoring direction) in $\mathcal{C}_3'$ for $\mathcal{G}_2$,
namely $(j_3, \overline{j_3}, 171)$.  Composing around the triangle, one finds an endomorphism of trace $\pm 1$.  Since the endomorphism has norm $8$, it has minimal polynomial $x^2 \pm  x + 8$, and hence is of the form $\frac{ \pm 1 \pm \sqrt{-31}}{2}$.  This generates an order with class number $3$ (the class numbers were computed with SageMath).  

In Table~\ref{table:cycles}, we give the lists $\mathcal{C}_r'$, and $L_r$, and put them in correspondence. 
This represents the first row of the table which puts $\mathcal{C}_3' = \{ (j_3, \overline{j_3}, 171)\}$ and $L_3 = \left\{\ZZ\left[ \frac{1+\sqrt{-31}}{2} \right]\right\}$ (as a proxy for $\mathcal{R}_3/{\sim}$ modulo direction) in bijection.

The table also verifies Corollary~\ref{cor:h}.  In our case, we have observed no fields ramified at $p$, so $\epsilon_{\mathcal{O},\ell} = 2$.  We compute:
\begin{align*}
2 = 2 \# \mathcal{C}_3' = \# \mathcal{C}_3 &= \frac{2}{3} h\left( \ZZ\left[ \frac{1 + \sqrt{-31}}{2} \right] \right) = \frac{2}{3} \cdot 3 = 2. \\
2 = 2 \# \mathcal{C}_4' = \# \mathcal{C}_4 &= \frac{2}{4} h\left( \ZZ\left[ \frac{1 + \sqrt{-39}}{2} \right] \right) = \frac{2}{4} \cdot 4 = 2.\\
2 = 2 \# \mathcal{C}_5' = \# \mathcal{C}_5 &= \frac{2}{5} h\left( \ZZ\left[ \frac{1 + \sqrt{-47}}{2} \right] \right) = \frac{2}{5} \cdot 5 = 2.\\
14 = 2 \# \mathcal{C}_6' = \# \mathcal{C}_6 &= \frac{2}{6} \left(
h\left( \ZZ\left[ \frac{1 + \sqrt{-87}}{2} \right] \right)+
h\left( \ZZ\left[ \frac{1 + \sqrt{-231}}{2} \right] \right)+ 
\right. \\
&\quad\quad\quad\left.
h\left( \ZZ\left[ \frac{1 + \sqrt{-247}}{2} \right] \right)+
h\left( \ZZ\left[ \frac{1 + \sqrt{-255}}{2} \right] \right)+
h\left( \ZZ\left[ 3 \frac{1 + \sqrt{-15}}{2} \right] \right)
\right) \\
&= \frac{2}{6} \cdot \left( 6 + 12 + 6 + 12 + 6 \right) = 14.
\end{align*}
This verifies Corollary~\ref{cor:h} in this example.

\begin{table}
\begin{center}
\setlength{\tabcolsep}{0.5em}
\renewcommand{\arraystretch}{1.8}
\begin{tabular}{ |c|c|c|c|c|c|c| } 
\hline
 isogeny cycle & figure & length & endomorphism & order & class number \\
 \hline
 \hline
 $(j_3, \overline{j_3}, 171)$ & Fig.\ A & 3 &  $\frac{\pm 1 \pm \sqrt{-31}}{2}$ & $\ZZ\left[ \frac{1 + \sqrt{-31}}{2} \right]$ & 3  \\
 \hline
 \hline
 $(61, j_1, 140, \overline{j_1})$ & Fig.\ A & 4  & $\frac{\pm 5 \pm \sqrt{-39}}{2}$ & $\ZZ\left[ \frac{1 + \sqrt{-39}}{2} \right]$ & 4 \\
\hline
\hline
  $(22,\overline{j_2}, \overline{j_3}, j_3, j_2)$ & Fig.\ A & 5 & $\frac{\pm 9 \pm \sqrt{-47}}{2}$ & $\ZZ\left[ \frac{1 + \sqrt{-47}}{2} \right]$ & 5 \\
  \hline
  \hline
    $(22,\overline{j_2}, \overline{j_1}, 140, j_1, j_2)$ &Fig.\ B & 6 & $\frac{\pm 13 \pm \sqrt{-87}}{2}$ & $\ZZ\left[ \frac{1 + \sqrt{-87}}{2} \right]$ & 6 \\
    \hline
      $(140,j_1,j_2,j_3,171,120)$ &Fig.\ C  & \multirow{2}{*}{6} & \multirow{2}{*}{$\frac{\pm 5 \pm \sqrt{-231}}{2}$} & \multirow{2}{*}{$\ZZ\left[ \frac{1 + \sqrt{-231}}{2} \right]$} & \multirow{2}{*}{12} \\
      $(140,\overline{j_1},\overline{j_2},\overline{j_3},171,120)$ & Fig.\ D  & & &&\\
      \hline
        $(0,121,112,35,112,121)^*$ &Fig.\ E &6 & $\frac{\pm 3 \pm \sqrt{-247}}{2}$ & $\ZZ\left[ \frac{1 + \sqrt{-247}}{2} \right]$ & 6 \\
        \hline
                $(22, j_2, j_3, 171, \bar{j}_3,\bar{j}_2)$ & Fig.\ F   &\multirow{2}{*}{6}  & \multirow{2}{*}{$\frac{\pm 1 \pm \sqrt{-255}}{2}$} & \multirow{2}{*}{$\ZZ\left[ \frac{1 + \sqrt{-255}}{2} \right]$} & \multirow{2}{*}{12} \\
        $(0,121,112,35,112,121)^*$ & Fig.\ E &  & &&\\
                \hline
                $(61,j_1,j_2,22,\overline{j_2}, \overline{j_1})$ &Fig.\ G & 6 & $\frac{\pm 11 \pm 3\sqrt{-15}}{2}$ & $\ZZ\left[ 3\left( \frac{1 + \sqrt{-15}}{2}\right) \right]$ & 6 \\
 \hline
\end{tabular}

\caption{Cycles of lengths three through six, with the associated endomorphisms to which the cycles compose.  The two cycles labelled with an asterisk $*$ are not fully specified by giving their $j$-invariants.  However, as discussed in Section~\ref{sec:arb}, the bijection between these two cycles and the two associated $\alpha$ is not canonical, but if we choose a safe arbitrary assignment, we can make a non-canonical bijection.  So there is no reason to specify which cycle is which in the table here.}\label{table:cycles} 
\end{center}
\end{table}

\begin{figure}[h!]

\begin{center}
 \subfloat[]{\includegraphics[width = 4.5in]{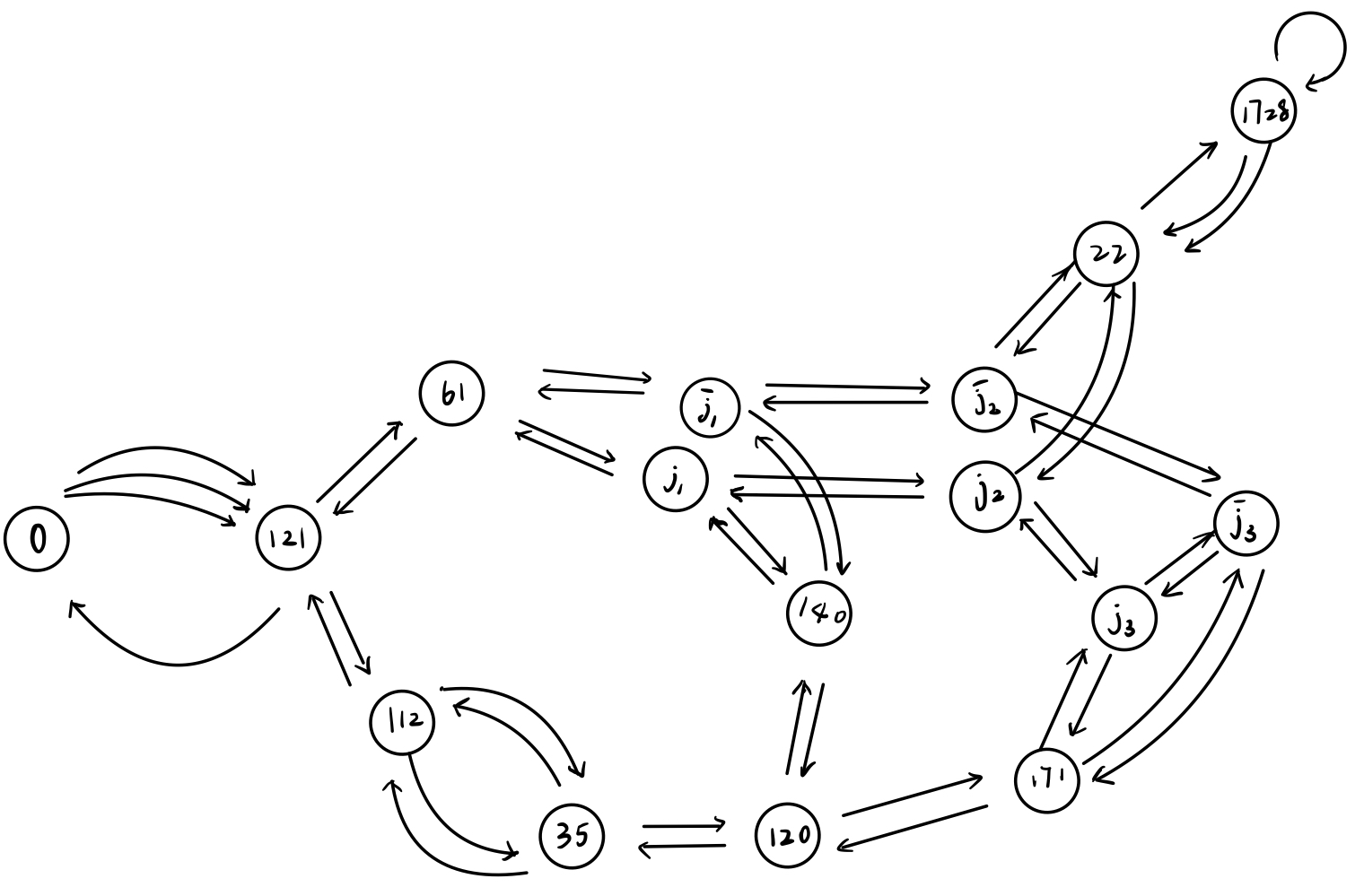}} \\
 \subfloat[]{\includegraphics[width = 2in]{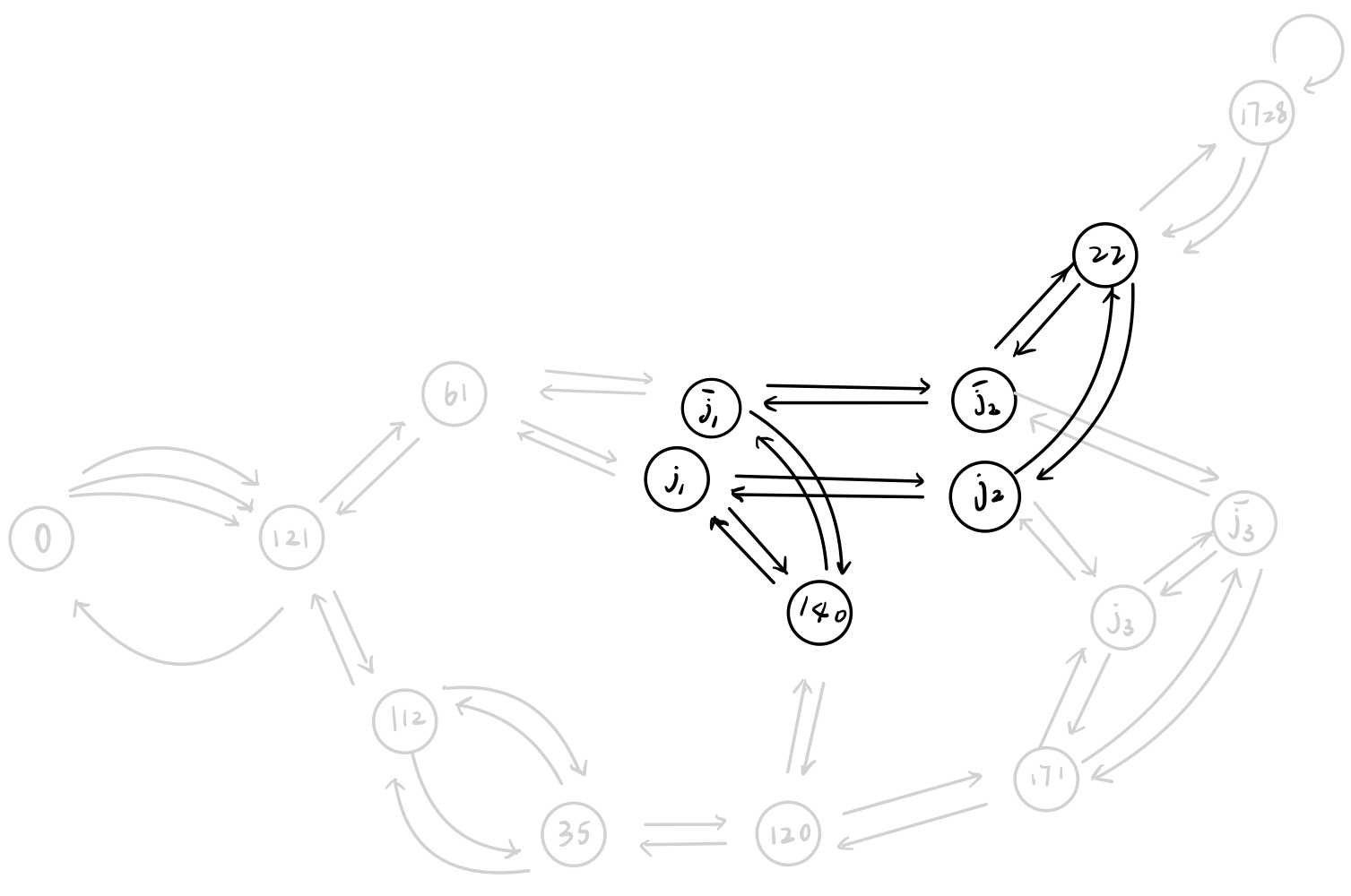}} 
\subfloat[]{\includegraphics[width = 2in]{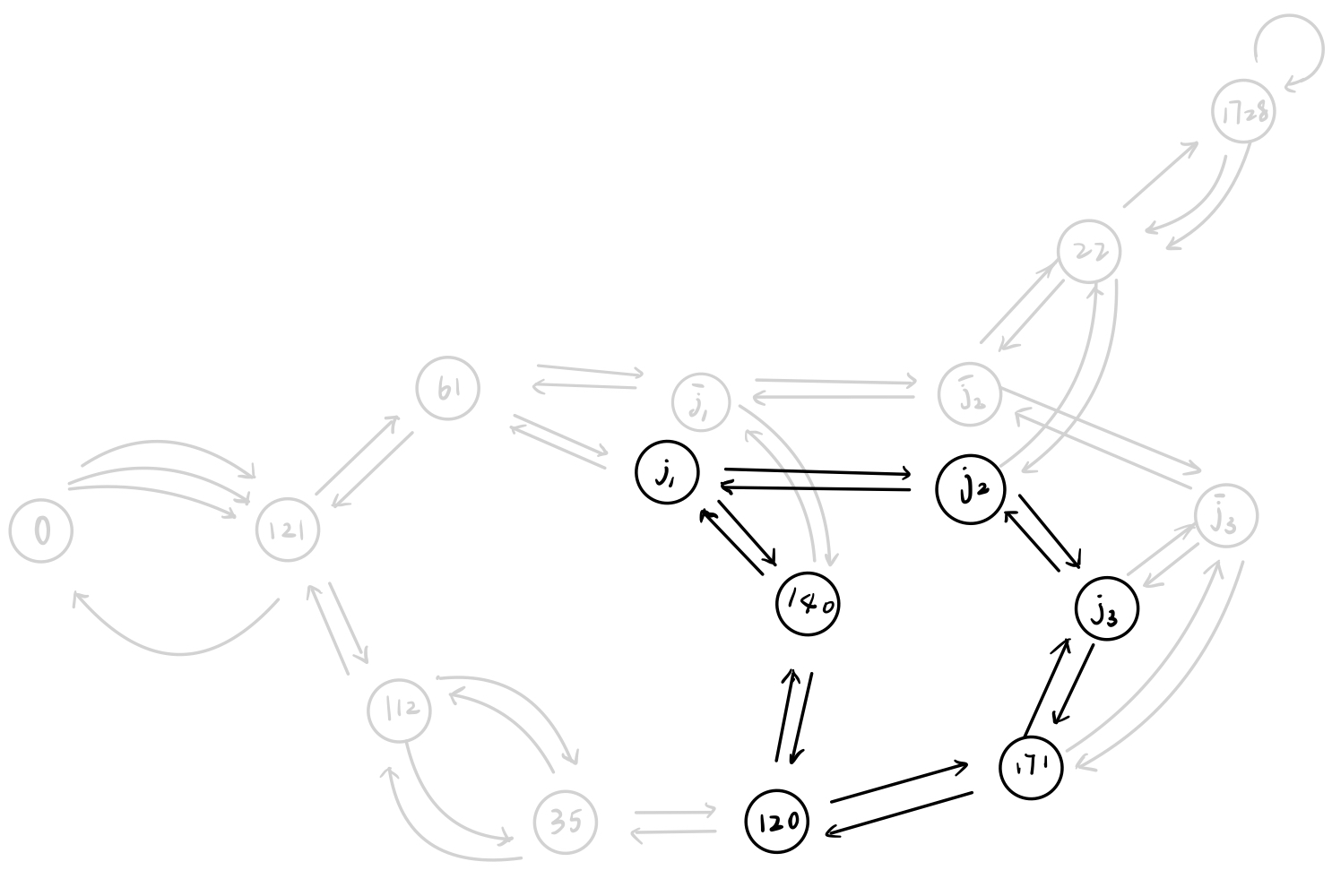}}
\subfloat[]{\includegraphics[width = 2in]{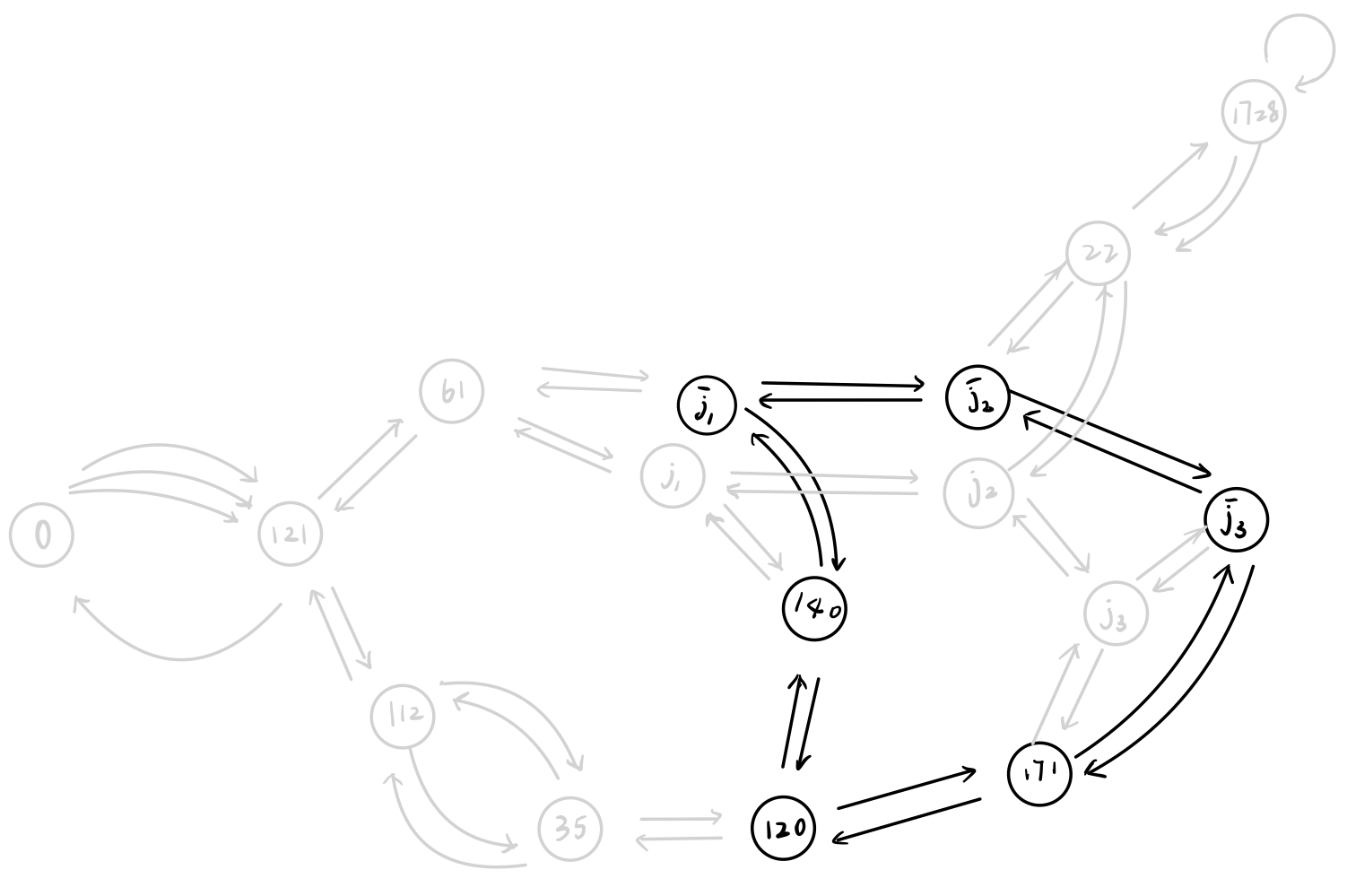}} \\
\subfloat[]{\includegraphics[width = 2in]{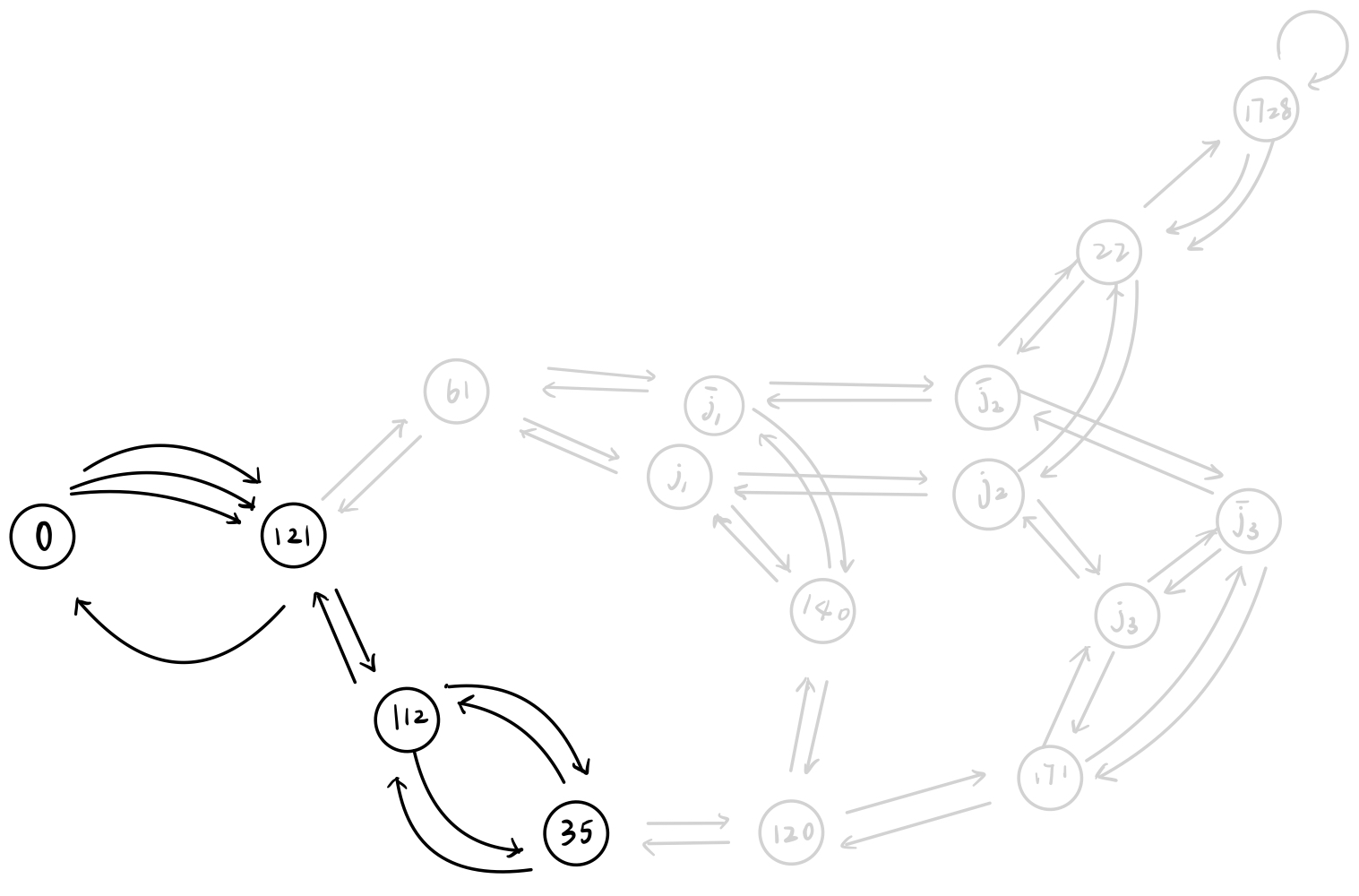}}
\subfloat[]{\includegraphics[width = 2in]{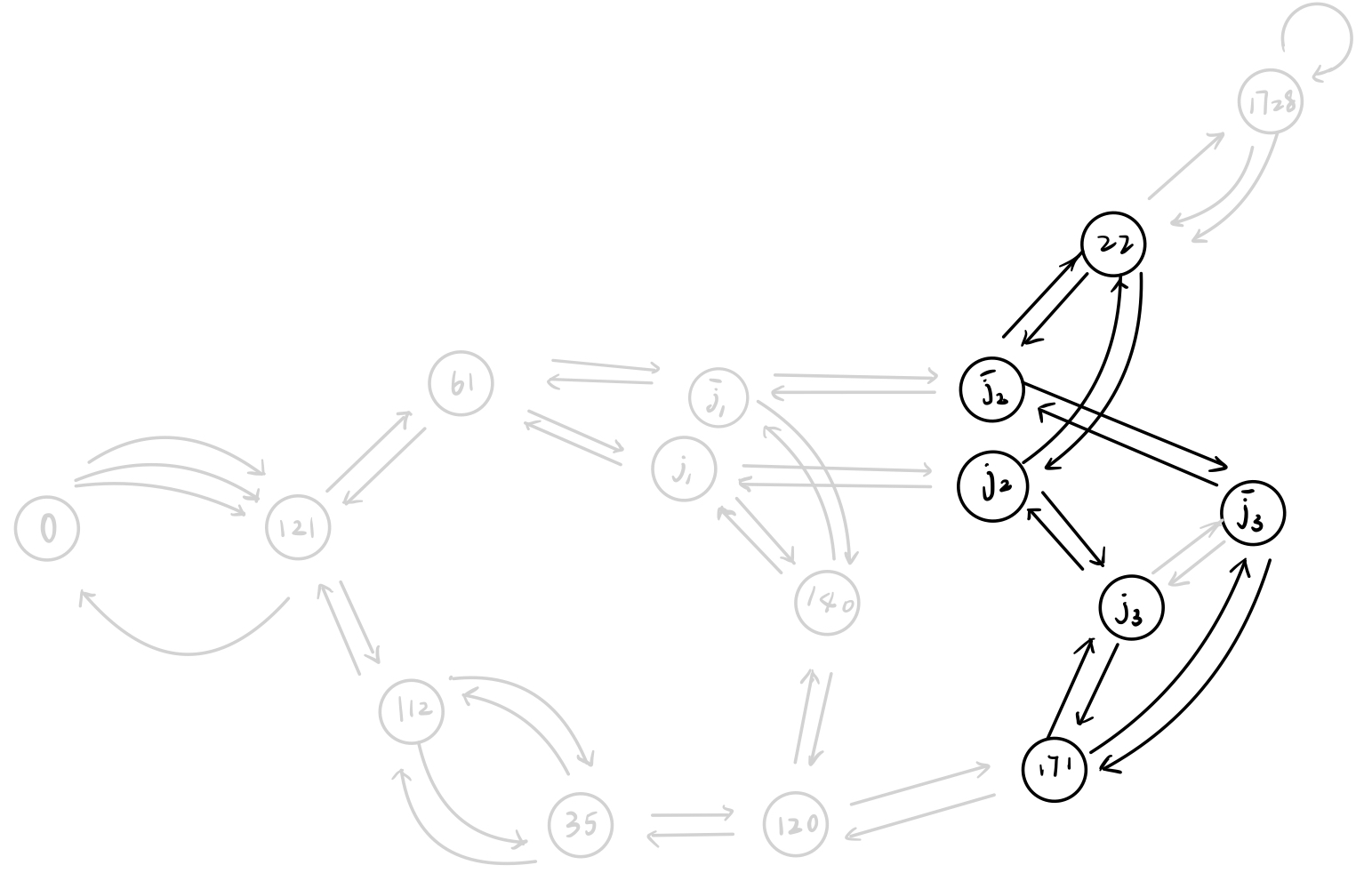}} 
\subfloat[]{\includegraphics[width = 2in]{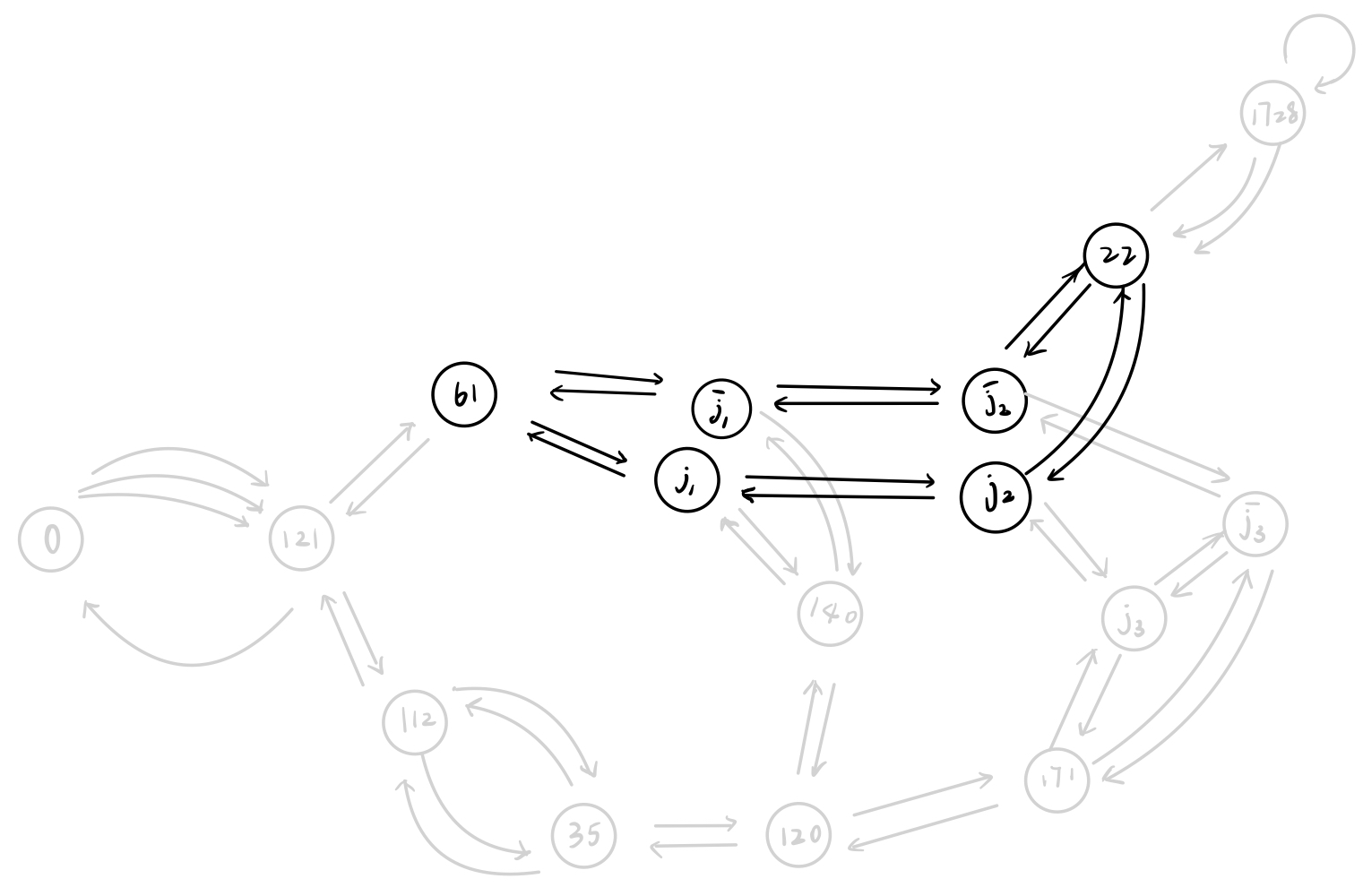}}\\
 \caption{Figure (A) is the supersingular $2$-isogeny graph $\mathcal{G}_2$ over $\FF_{p^2}$, where $p=179$. Here $j_1 = 64i+5 ,\,j_2 =99i+107 ,\,j_3 =5i+109$, where $i$ denotes a square root of $-1$ in $\FF_{p^2}$. Figures (B) -- (G) are the isogeny cycles of length 6 as listed in Table~\ref{table:cycles}. There are two distinct isogeny cycles included in (E), both sharing the same sequence of vertices.}\label{fig:2isogenygraphp179} 
\end{center}
\end{figure}

\textbf{The cycle $(0,121,112,35,112,121)$.}
The case of the $j$-invariant cycle $(0,121,112,35,112,121)$ in Figure~\ref{fig:2isogenygraphp179}~(E) is particularly interesting.  This sequence of $j$-invariants forms a $6$-cycle in 12 possible ways, as follows.  There are 4 ways to travel from $112$ to $35$ and back.  There are 3 ways to travel from $121$ to $0$ and back, producing three distinct endomorphisms of $j=121$ (but not canonically), each composing to one of $\pm 2$, $\pm \widehat{\varphi} \circ \zeta_3 \circ \varphi$, or $\pm \widehat{\varphi} \circ \zeta_3^2 \circ \varphi$, as described in Figure~\ref{fig:0121}.  Thus, one endomorphism is backtracking and the other two are duals of each other.
The isogenies from $112$ to $35$ each have a  dual:  hence two of the $2$-cycles from $112$ to $35$ and back are backtracking (obtained by pairing an isogeny with its dual).  There are therefore two non-backtracking $2$-cycles to choose from among the four ways to travel from $112$ to $35$ and back. These two non-backtracking $2$-cycles are actually duals of each other, and are obtained by simply changing the direction of traversal.
If the traces of the paths that include backtracking are computed, we obtain the associated endomorphisms $\pm 8$, $4(\pm 1 \pm  \sqrt{-3})$, and $2(\pm 1 \pm \sqrt{-15})$, all of which are divisible by~$2$.  This leaves exactly four non-backtracking $6$-cycles on the $j$-invariants $(0,121,112,35,112,121)$, coming in pairs obtained by changing the direction of traversal.  These two pairs give the two endomorphisms listed in Table~\ref{table:cycles}.

\textbf{The cycle $(61,j_1,j_2,22,\overline{j_2},\overline{j_1})$.}
Another interesting cycle is $(61,j_1,j_2,22,\overline{j_2},\overline{j_1})$ in Figure~\ref{fig:2isogenygraphp179}~(G).  This cycle corresponds to an element generating the non-maximal order $\mathcal{O}_3 :=\ZZ\left[ 3\left( \frac{1 + \sqrt{-15}}{2}\right) \right]$.  That element also exists in the order $\mathcal{O}_1 := \ZZ\left[\frac{1+\sqrt{-15}}{2}\right]$.  How do we know which order the volcano rim corresponds to?  The easiest answer is that only the class number of $\mathcal{O}_3$ is divisible by $6$.  However, even without that observation (sometimes several class numbers may be divisible by the correct integer $r$), we can observe that the orientation $\iota$ given by the endomorphism on any of the $j$-invariants on the cycle is $\mathcal{O}_3$-primitive, not $\mathcal{O}_1$-primitive.  This can be seen because if $\End(E)$ contained $\frac{1+\sqrt{-15}}{2}$, which has norm $16$, then this element would be realized as a $4$-cycle from~$E$, which it is not.  (The curves $j_2, \overline{j_2}$, and $22$ have no $4$-cycle.  The curves $61$, $j_1$ and $\overline{j_1}$ have the $4$-cycle associated to $\frac{\pm 5 \pm \sqrt{-39}}{2}$.)  The moral of this observation is that each cycle can only arise from a single quadratic order, as dictated by Theorem~\ref{thm:mainbij-nobase}.

\textbf{Repeating $j$-invariants.}
A final remark is the observation that repeating $j$-invariants in a cycle can only happen when the discriminant is large with respect to the prime.  This is a consequence of a result by Kaneko which gives the necessary condition $|\disc(\mathcal{O})|\geq p$~\cite[Theorem 2']{Kaneko}.  In our case, we do get a cycle with repeating $j$-invariants, but the associated discriminants are quite large at $-247$ and $-255$, both bigger in absolute value than the prime $p = 179$.

\textbf{Comparison with Gross \cite{Gross_Heights}}
Working with the same example, Gross counts cycles of length $\log_\ell(m)$ in the supersingular $\ell$-isogeny graph $\mathcal{G}_\ell$ via the formula provided in \cite[Proposition 1.9]{Gross_Heights}.  In the first row of Table \ref{table:cycles}, we count one isogeny cycle of length three in the 2-isogeny graph, namely $(j_3,\overline{j_3},171)$ arising from the order of discriminant $-31$. Gross counts nine via the formula \cite[Proposition 1.9]{Gross_Heights}. In particular, six of this count come from the order of discriminant $-31$. The three-cycle $(j_3,\overline{j_3},171)$, which we count once, is counted six times by Gross because it is a valid three-cycle starting at any of $j_3, 171, \overline{j_3}$, in either direction. Two more of Gross's count come the order of discriminant $-16$.  The cycle corresponding to the order of discriminant $-16$ is necessarily $(j_{22},j_{1728},j_{1728})$.  This cycle is only without backtracking when we start at 22, so Gross counts it once in each direction. We do not count this cycle, as it contains backtracking as soon as we forget the basepoint. The last one of Gross's count comes from the order of discriminant $-4$, as $j_{1728}$ has an endomorphism of degree-8 corresponding to composing the loop at $j_{1728}$ three times, with appropriate automorphisms in between to avoid backtracking.

\section{Counting isogeny cycles in $\mathcal{G}_\ell$}
\label{sec:counting}

Theorem~\ref{thm:mainbij-fibre} makes possible explicit formulas for the number of isogeny cycles of length $r$ in the $\ell$-isogeny graph $\mathcal{G}_\ell$.  In this section, we do the following:
\begin{enumerate}
    \item Theorem~\ref{thm:numcycles-randwalk} uses spectral graph theory to give an asymptotic for the expected number of length~$r$ isogeny cycles, and this estimate is compared to empirical data in Figure~\ref{fig:cycledata}. The use of spectral graph theory requires us to restrict to $p\equiv1\pmod{12}$, but for any large enough $p$ we still expect the estimates of this theorem to hold.

    \item Corollary~\ref{cor:h} gives an explicit formula for the number of length $r$ isogeny cycles in terms of class numbers of certain quadratic orders, and Theorem~\ref{thm:QNR} reformulates this as a sum over these class numbers with simpler conditions on which quadratic orders to include.
    
    \item Theorem~\ref{thm:QNR} enables an explicit upper bound on the length $r$ isogeny cycle count, given in Corollary~\ref{cor:cyclesbound}.  Unfortunately a lower bound would seem to be connected to an open problem in number theory (Remark~\ref{remark:burgess}).
\end{enumerate}

\subsection{Asymptotic number of isogeny cycles.}
\label{sec:asymptotic}

It is a standard fact that for $(\ell+1)$-regular Ramanujan graphs, the number of basepointed non-backtracking closed walks of length $r$ (where the non-backtracking condition does not rule out entering the final vertex along the initial edge of the walk) is
\begin{equation*}
(\ell+1) \ell^{r-1} +  \Np \,O\left(r (\ell+1)^{r/2} \right),
\end{equation*}
where the implied constant does not depend on $\Np$.
This follows, for example, from \cite[Theorem 1.4.6]{DavidoffSarnakValette}; some details are provided at 
\cite{mathoverflow}.  If we disallow entering the final vertex along the initial edge and then forget basepoints, then one would heuristically expect this to become
\begin{equation}
\label{eqn:ramanujancount}
\ell^{r}/2r +  \Np \,O\left((\ell+1)^{r/2} \right).
\end{equation}
For small $r$, estimating the number of graph-theoretic $r$-cycles is similar (recall that a cycle allows no repeated edges or vertices); this is because random $d$-regular graphs on $n$ vertices have on average $(d-1)^r/2r$ such $r$-cycles \cite[Theorem 2]{BOLLOBAS1980311}.  In our case, this would give $\ell^r/2r$.  The number of $r$-cycles in some Ramanujan graphs, including LPS graphs, but not supersingular graphs, has been analysed, and these graphs appear to behave as random graphs \cite{TillichZemor}.

In what follows, we give a brief proof of an asymptotic version of \eqref{eqn:ramanujancount}.

\begin{theorem}
\label{thm:numcycles-randwalk}
Let $p\equiv1\pmod{12}$. Let $\mathcal{G}_{\ell}$ be the supersingular $\ell$-isogeny graph over $\overline{\FF}_p$ and let $\Np$  denote its number of vertices. 
 The number of non-backtracking closed walks of length $r$ in $\mathcal{G}_{\ell}$, taken up to order of traversal and starting point, asymptotically approaches $\ell^r/2r$
as $r \rightarrow \infty$. Thus, the number of isogeny cycles in $\mathcal{G}_\ell$ is also asymptotically $\ell^{r}/2r$ as $r \rightarrow \infty$.
\end{theorem}

Before giving a formal proof, we give a heuristic description.  Choose one of $\Np$ initial vertices $j_0$.  Label its $\ell+1$ neighbours $j_1, j_2, \ldots, j_{\ell+1}$.  Choose one of these neighbours for the first step, say $j_1$ (relabelling without loss of generality).  Consider the $\ell^{r-1}$ non-backtracking length $r-1$ walks that could follow the edge $j_0 \rightarrow j_1$.  Such a walk provides a non-backtracking closed walk of length $r$ through $j_0 \rightarrow j_1$ if and only if the walk terminates at $j_0$ via one of $j_2, \ldots, j_{\ell+1}$.  The probability of this termination vertex $j_0$ should be expected to be $1/\Np$.  However, only $\ell$ of the $\ell+1$ incoming edges reach $j_0$ via one of $j_2, \ldots, j_{\ell+1}$.  Finally, we obtain a typical closed  non-backtracking walk of length $r$ exactly $2r$ times in this way, since there are $r$ choices of initial vertex and $2$ directions.  Hence the number of closed non-backtracking walks of length $r$ should be obtained by combining the steps in the outline above:
\[
 \Np (\ell+1) \ell^{r-1} (1/\Np) (\ell/(\ell+1))/2r = \ell^r/2r.
\]

\begin{proof}[Proof of Theorem~\ref{thm:numcycles-randwalk}]
The natural object for studying non-backtracking walks is a Markov chain on the \emph{non-backtracking matrix} $B$, which is the adjacency matrix of an auxiliary graph $\mathcal{G}'$ whose vertices are the directed edges of $\mathcal{G}_\ell$, and whose edges indicate adjacency between edges of $\mathcal{G}_\ell$ that respects the directedness and does not backtrack.  In other words, the matrix is a scaling of the transition probability matrix of a non-backtracking random walk on the directed edges of $\mathcal{G}_\ell$.  In particular, the entries of $B^r$ give the number of non-backtracking walks of length $r+1$ that start and end on certain edges.  This is studied in \cite{Kempton2016NonbacktrackingRW}, where it is shown that this Markov chain converges to a uniform stationary distribution (the mixing property).  There are $\Np(\ell+1)$ directed edges in $\mathcal{G}_\ell$.   Choose one of these as the initial directed edge of a walk, originating at a vertex which we may call $j_0$.  From that point, take a random non-backtracking walk of length $r-1$ (there are $\ell^{r-1}$ of these).  Then we have created a closed non-backtracking walk if and only if we end on one of the $\ell$ edges that are directed into $j_0$ without backtracking the initial edge.  The mixing property implies that the probability of ending on of these $\ell$ incoming edges is asymptotically $\ell/(\Np(\ell+1))$.  Therefore, asymptotically, we obtain 
\[
\Np(\ell+1)\ell^{r-1}\left( \frac{\ell}{\Np(\ell+1)} \right)  = \ell^r
\]
total non-backtracking closed walks with a direction and a basepoint.  Most of these can be directed in two ways and a basepoint can be chosen in $r$ ways.  The exceptions, namely the barbells (Definition~\ref{defn:barbell}) are those closed walks that involve the same edges when the direction is reversed.  These exceptions are rare enough (Lemma~\ref{lemma:barbellcount}) that they do not affect the asymptotic.  Finally, isogeny cycles are those non-backtracking closed walks which are not powers of smaller non-backtracking closed walks; this requirement clearly also does not affect the asymptotic.
\end{proof}

\begin{remark}
In Theorem~\ref{thm:numcycles-randwalk}, we restrict to $p\equiv 1\pmod{12}$ because the graph theory involved requires an $(\ell+1)$-regular graph. We have remarked at length the difficulties which arise when vertices $j=0,1728$ appear in $\mathcal{G}_\ell$. However, these issues are very localized -- the $(\ell+1)$-regularity of the graph may only be affected at vertices adjacent to $j=0,1728$. For $p\gg \ell$, this is a very small number of potential deviations from regularity, and we do not expect the estimates provided here to be greatly affected. 
\end{remark}

By reference to \cite{Kempton2016NonbacktrackingRW} and \cite{walk_mix}, the mixing rate of the Markov chain could be used to analyse the rate of convergence of the asymptotic and make this theorem more precise.

The proof of Theorem~\ref{thm:numcycles-randwalk} lends itself to computation:  using powers of the non-backtracking matrix, we can compute the number of isogeny cycles in $\mathcal{G}_\ell$, at least up to the issue of barbells.  Figure~\ref{fig:cycledata} shows some data collected in SageMath for the number of non-backtracking closed walks of length $r$ on some small supersingular $\ell$-isogeny graphs chosen to avoid loops (so $\mathcal{G}_\ell$ is $(\ell+1)$-regular without any barbells), and its agreement with the asymptotic.

\begin{figure}
    \centering
    \includegraphics[width=0.4\textwidth]{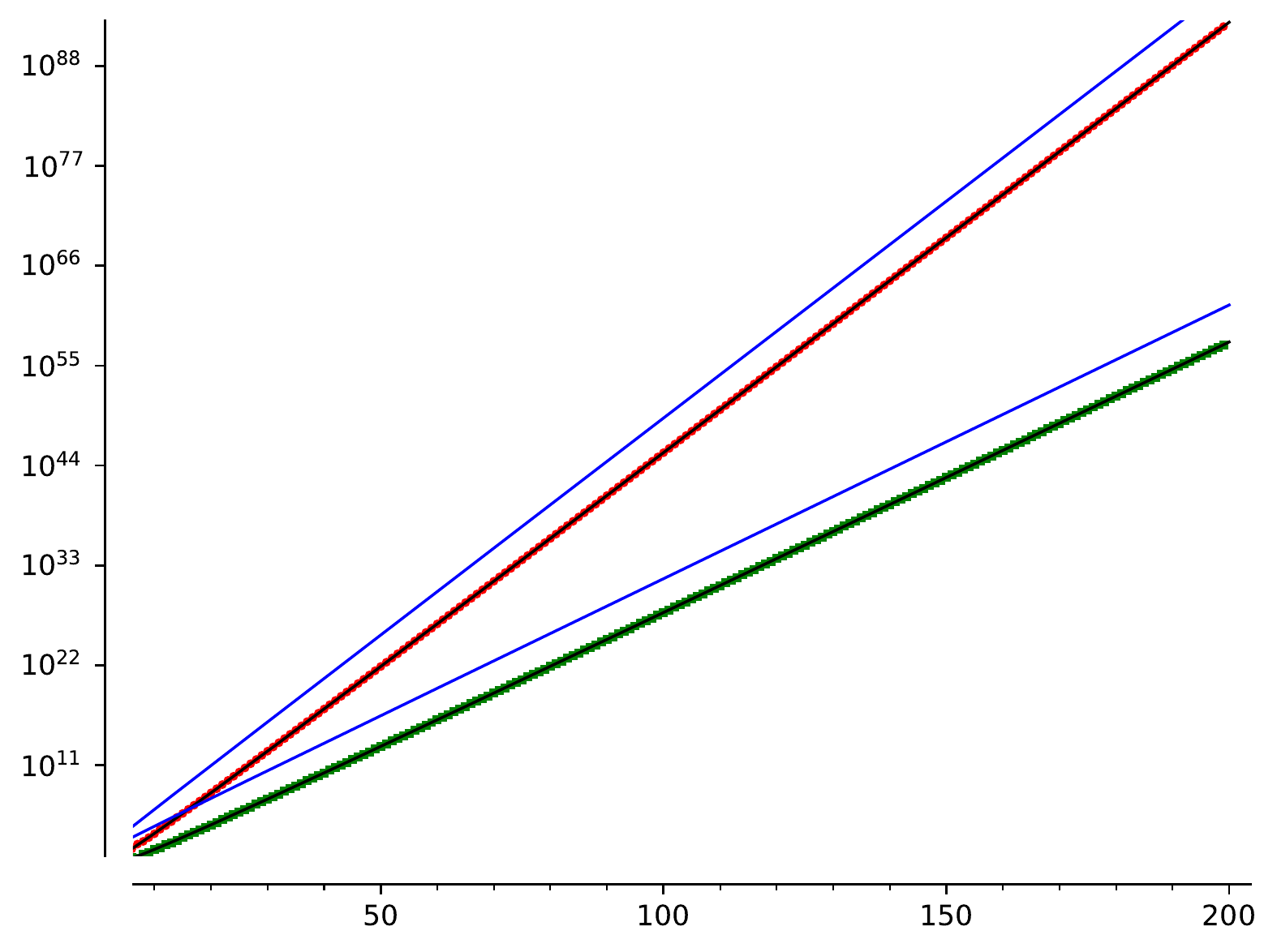}
    \includegraphics[width=0.4\textwidth]{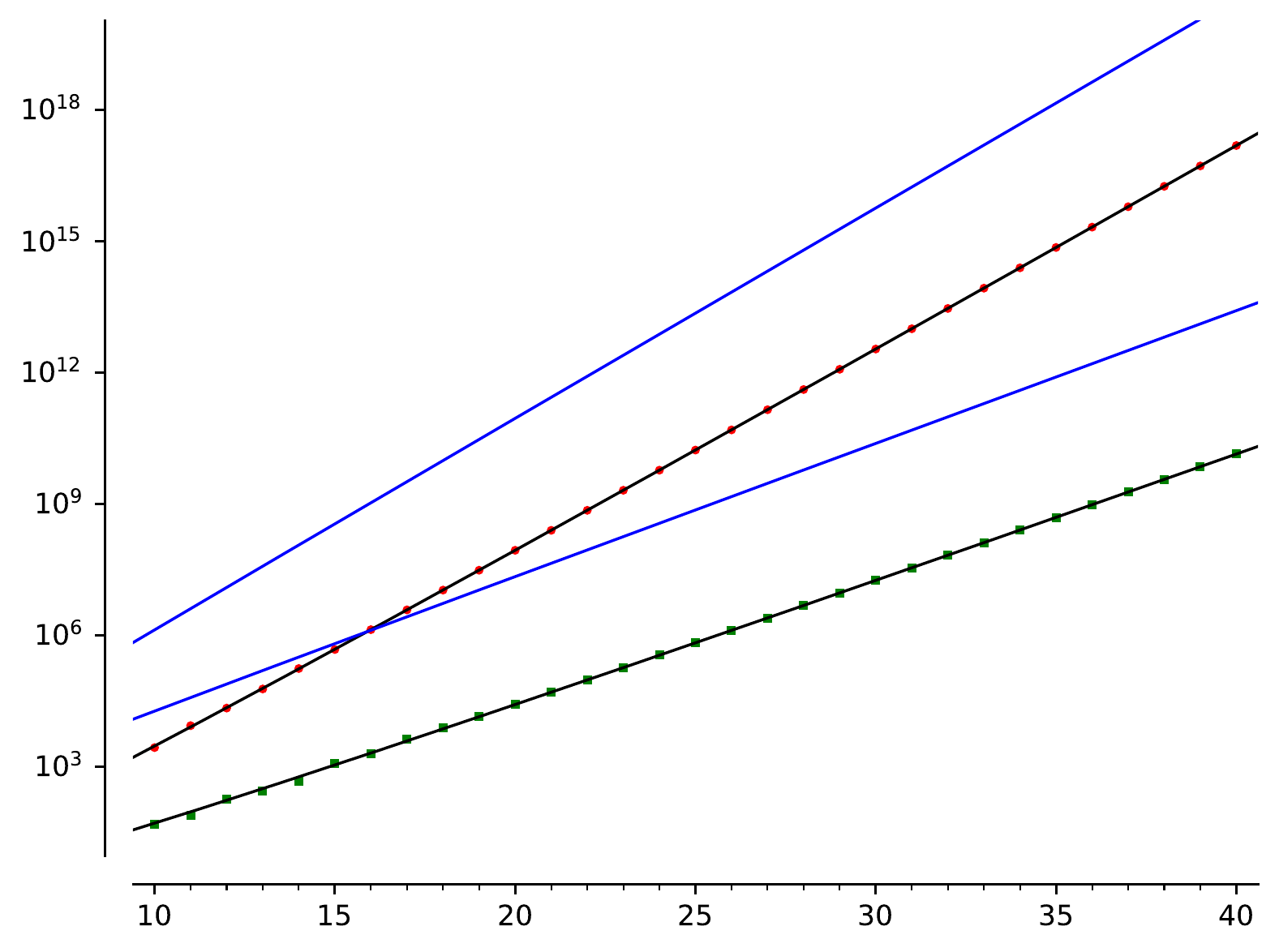}
    \caption{A semilog plot of cycle counts for $\mathcal{G}_\ell$ for $p=3361$ and $\ell=2$ (green squares, lower graph), and for $p =3229$ and $\ell=3$ (red circles, upper graph).  The black curves upon which the data points appear to sit are $\ell^x/(2 x)$.  The blue lines above the data in each case indicate the dominant term of Corollary~\ref{cor:cyclesbound}, indicating our proved upper bound.  On the left, the full data.  On the right, a close-up of the behaviour for small cycle sizes.  The graphs were generated in part using code from \cite{arpin2019adventures}.}
    \label{fig:cycledata}
\end{figure}

\subsection{Counting isogeny cycles with class numbers}

A corollary to Theorems~\ref{thm:mainbij-nobase} and \ref{thm:mainbij-fibre} counts cycles exactly as a sum of class numbers of certain quadratic orders.

\begin{corollary}
\label{cor:h}
Let $r > 2$ be fixed and consider primes $\ell \neq p$.  Let $\mathcal{I}_r$ be as in Theorem~\ref{thm:mainbij-fibre}.

Let $c_r = |\mathcal{C}_r|$ denote the number of directed isogeny cycles of length $r$.
Then
\[ c_r = \frac{1}{r} \sum_{\mathcal{O} \in \mathcal{I}_r} \epsilon_{\mathcal{O},\ell} \hO,
\]
where $\epsilon_{\mathcal{O},\ell}$ is as defined in Theorem~\ref{thm:mainbij-fibre}.  In particular, $1 \leq \epsilon_{\mathcal{O},\ell} \leq 2$ and $\epsilon_{\mathcal{O},\ell} = 2$ if $p$ is inert in the quadratic field containing $\mathcal{O}$.  If $\ell^r < p$, then $\epsilon_{\mathcal{O},\ell} = 2$ for all $\mathcal{O} \in \mathcal{I}_r$.
\end{corollary}

\begin{proof}
This follows from Theorems~\ref{thm:mainbij-nobase} and \ref{thm:mainbij-fibre}.  The cycles $\mathcal{C}_r$ of length $r$ in $\mathcal{G}_\ell$ are in bijection with $\mathcal{R}_r/{\sim}$.  Each $[R]$ then maps to some order $\mathcal{O} \in \mathcal{I}_r$.  Summing the elements in $\mathcal{I}_r$ weighted by the fibre sizes counts $\# \mathcal{C}_r$.  The fibre sizes are given in Theorem~\ref{thm:mainbij-fibre}.
\end{proof}

On the graph side, it is clear that this count is finite.  On the side of ideal classes, finiteness of $\mathcal{I}_r$ follows from the observation that if $\operatorname{ord}([\mathfrak{l}]) = r$, then $\ell^r$ is represented by the norm form of the quadratic field.  For sufficiently large discriminant, this representation is impossible except as $\ell^r = (\ell^{r/2})^2$ if $r$ is even; but in this case, $\ell$ is ramified.  This observation can be turned into a counting argument based on quadratic residues.

\begin{theorem}
\label{thm:QNR}
Assume $p > 2$. For any positive integer $N$, define
\[
\mathcal{Q}(N) =
\left\{ 0 < x < 2\ell^{N/2} : \substack{
\mbox{$x \not\equiv 0 \pmod \ell$}, \\
\mbox{$x^2 - 4\ell^N$ is not a quadratic residue modulo $p$}, \\[2pt]
\mbox{and has valuation $\le 1$ at $p$}
}\right\},
\]
and 
\[ Q_N = \sum_{x \in \mathcal{Q}(N)} \sum_{f^2 \mid (x^2 - 4\ell^N)} \epsilon_{\mathcal{O},\ell} \, h\left ( \frac{x^2-4\ell^N}{f^2} \right ) , \]
where $h(D)$ denotes the class number of the order of discriminant $D$.
 Let $c_N$, $N \ge 3$, be as in Corollary~\ref{cor:h}. Then 
there are values for $c_1$ and $c_2$ such that
 \[
 \sum_{r \mid N} r c_r = Q_N, \quad c_N = \frac{1}{N} \sum_{r \mid N} \mu(r) Q_{N/r}.
\]
\end{theorem}

\begin{proof}
Let $r \ge 3$. Recall from Theorems~\ref{thm:mainbij-nobase} and \ref{thm:mainbij-fibre} that we have maps
\[
\mathcal{C}_r \longrightarrow \mathcal{R}_r/{\sim}\; \longrightarrow \mathcal{I}_r.
\]
The first map is a bijection, and the second map has fibre sizes $\epsilon_{\mathcal{O},\ell} \hO/r$ above $\mathcal{O}$, as described in Theorem~\ref{thm:mainbij-fibre}.  Using the fibre sizes, we obtain
\[
c_N = \frac{1}{N} \sum_{\mathcal{O} \in \mathcal{I}_N} \epsilon_{\mathcal{O},\ell} \hO.
\]
Then
\begin{align*}
\sum_{r \mid N} r c_r
&= \sum_{\mathcal{O} \in \bigcup_{r \mid N} \mathcal{I}_r} \epsilon_{\mathcal{O},\ell} \hO.
\end{align*}
By definition, $\mathcal{O} \in \bigcup_{r \mid N} \mathcal{I}_r$ if and only if there exists an imaginary quadratic integer $\alpha \in \mathcal{O}$, not divisible by~$\ell$, such that $N(\alpha) = \ell^r$, $p$ does not split in $\QQ(\alpha)$, and $\ell$ is invertible and splits in $\mathcal{O}$.
Since these properties of $\alpha$ are invariant up to conjugation and sign, we can write  $\alpha = \frac{x + \sqrt{\Delta}}{2}$, where $\Delta$ is the discriminant of the order $\ZZ[\alpha]$, and $x \ge 0$. The condition that $\alpha$ has norm $\ell^N$ is equivalent to $x^2 - \Delta = 4\ell^N$.  The fact that $\alpha$ is imaginary is equivalent to $0 > \Delta = x^2 - 4 \ell^N$, or equivalently, $x < 2\ell^{N/2}$.  Next we show that the prime $\ell$ is invertible and splits in $\ZZ[\alpha]$ if and only if $x \not\equiv 0 \pmod \ell$.  To that end, invertibility is equivalent to the condition that $\ell$ does not divide the conductor of $\ZZ[\alpha]$, which is equivalent to $x \not\equiv 0 \pmod \ell$ (since $x^2 - \Delta = 4\ell^N$), and if these equivalent conditions hold, then $1 = \left( \frac{\Delta}{\ell} \right)=\left( \frac{x^2 - 4\ell^N}{\ell} \right)$, which holds if and only if  $\ell$ splits in the maximal order containing $\mathcal{O}$ (by Lemma~\ref{lemma:tower-split}).  
The prime $p$ does not split in this maximal order if and only if $x^2 - 4\ell^N$ is not a quadratic residue modulo $p$.  Finally, $\alpha$ is not divisible by $\ell$ if and only if $\ell$ does not divide $\Delta$.  To conclude then,
$\mathcal{O} \in \bigcup_{r \mid N} \mathcal{I}_r$
if and only if $\Delta_{\mathcal{O}} = x^2 - 4\ell^N$ for $x \in \mathcal{Q}(N)$ of the theorem statement.  Note that each $\mathcal{O}$ corresponds to one $\Delta_\mathcal{O}$, which in turn corresponds to one value of $x$. Therefore,
\[ \sum_{r \mid N} r c_r
= \sum_{\mathcal{O} \in \bigcup_{r \mid N} \mathcal{I}_r} \epsilon_{\mathcal{O},\ell} \hO  = \sum_{x \in \mathcal{Q}(N)} \sum_{f^2 \mid x^2 - 4\ell^N} \epsilon_{\mathcal{O},\ell} h\left( \frac{x^2 - 4\ell^N}{f^2} \right) 
= Q_N. \]
Calling upon M\"obius inversion, we have proved the theorem.
\end{proof}

\begin{remark} 
The key observation we make here -- that small cycles arise from a combination of splitting behaviour between $p$ and $\ell$ -- is used in \cite[Section 5.3.4]{CharlesGorenLauter} to eliminate small cycles by use of congruence conditions on $p$.
\end{remark}

The advantage of Theorem~\ref{thm:QNR} is that in contrast to the sum of class numbers in Corollary~\ref{cor:h}, we obtain a sum of class numbers that has no conditions on the class groups. The condition on the order of $[\mathfrak{l}]$ present in Corollary~\ref{cor:h} is removed when summing over $x \in \mathcal{Q}(N)$, since $\mathcal{Q}_N$ is entirely given in terms of `elementary' conditions, namely Legendre symbols and valuations at $p$ and $\ell$. 

\begin{example}
We illustrate the theorem and its proof with the example of Section~\ref{sec:examples}, where $\ell=2$ and $p=179$.  We have, from Table~\ref{table:cycles}, the directed isogeny cycle counts
\[
c_3 = 2, \ c_4 = 2, \ c_5 = 2, \ c_6 = 14.
\]
Using Sage to compute $\mathcal{Q}(N)$, we obtain, showing the data as tuples $(x, x^2 - 4\ell^n, h_{x^2-4\ell^n})$,
\[
\mathcal{Q}(1) = \emptyset,
\mathcal{Q}(2) = \{ (1,-15, 2) \},
\mathcal{Q}(3) = \{ (1,-31,3) \},
\mathcal{Q}(4) = \{ (5,-39,4), (7,-15,2) \},
\mathcal{Q}(5) = \{ (9,-47,5) \},
\]
\[
\mathcal{Q}(6) = \{ (1,-255,12), (3,-247,6), (5,-231,12), (11, -135,6), (13, -87,6), (15, -31,3) \}.
\]
Observe that the last tuple $(15,-31,3)$ in $\mathcal{Q}(6)$ corresponds to $\alpha = \frac{\pm 15 \pm \sqrt{-31}}{2} = \pm\left( \frac{1 \pm \sqrt{-31}}{2} \right)^2$.  The element $\beta = \frac{\pm 1 \pm \sqrt{-31}}{2}$ corresponds to the unique $3$-cycle in Table~\ref{table:cycles}.  Note that $\ZZ[\alpha] = \ZZ[\beta] = \mathcal{O}_{\QQ(\sqrt{-31})}$.  Accordingly, $\alpha$ corresponds to the $3$-cycle repeated twice.  

The fourth tuple $(11,-135,6)$ in $\mathcal{Q}(6)$ corresponds to $\alpha = \frac{\pm 11 \pm 3\sqrt{-15}}{2} = \left( \frac{\pm 1 \pm \sqrt{-15}}{2} \right)^3$.  The element $\frac{\pm 1 \pm \sqrt{-15}}{2}$ corresponds to the $2$-cycle $(112,35)$.  In this case, $\ZZ[\beta] = \mathcal{O}_{\QQ(\sqrt{-15})}$, but $\ZZ[\alpha]$ is a suborder of $\ZZ[\beta]$ of conductor $3$.  Thus, $\alpha$ corresponds to two different types of $6$-cycles:  the $2$-cycle repeated $3$ times (therefore not an isogeny cycle); and also an isogeny cycle of size $6$, as in Table~\ref{table:cycles}.

Therefore, 
\begin{align*}
Q_1 &= 0, \\
Q_2 &= 2h \left  (\ZZ\left[ \frac{1+\sqrt{-15}}{2} \right] \right ) = 4,\\
Q_3 &= 2h \left (\ZZ\left[ \frac{1+\sqrt{-31}}{2} \right] \right ) = 6,\\
Q_4 &= 2h \left (\ZZ\left[ \frac{1+\sqrt{-15}}{2} \right] \right ) + 2h \left (\ZZ\left[ \frac{1+\sqrt{-39}}{2} \right] \right ) = 4+8 = 12, \\
Q_5 &= 2h \left ( \ZZ\left[ \frac{1+\sqrt{-47}}{2} \right] \right ) = 10,\\
Q_6 &= 
2h \left (\ZZ\left[ \frac{1+\sqrt{-255}}{2} \right] \right )+
2h \left (\ZZ\left[ \frac{1+\sqrt{-247}}{2} \right] \right )+
2h \left (\ZZ\left[ \frac{1+\sqrt{-231}}{2} \right] \right )+ \\
& \quad 2h \left (\ZZ\left[ \frac{1+\sqrt{-15}}{2} \right] \right )+
2h \left (\ZZ\left[ 3 \frac{1+\sqrt{-15}}{2} \right] \right )+
2h \left (\ZZ\left[ \frac{1+\sqrt{-87}}{2} \right] \right )+ \\
&\quad 2h \left (\ZZ\left[ \frac{1+\sqrt{-31}}{2} \right] \right ) = 24 + 12 + 24 + 4 + 12 + 12 + 6
= 94.
\end{align*}

Recall that for the M\"obius function, $\mu(1) = 1$, $\mu(2) = -1$, $\mu(3) = -1$, $\mu(4) = 0$, $\mu(5) = -1$, $\mu(6) = 1$.   So the sum $S(N) := \frac{1}{N} \sum_{r \mid N} \mu(r) Q_{N/r}$ takes on the following values, which we can verify to match $c_N$ for $3 \le N \le 6$:
\[
S(1) = Q_1 = 0, \quad  
S(2) = \frac{1}{2}(Q_2 - Q_1) = 2, \quad
S(3) = \frac{1}{3}(Q_3 - Q_1) = 2 = c_3,
\]
\[
S(4) = \frac{1}{4}(Q_4 - Q_2) = 2 = c_4, \quad
S(5) = \frac{1}{5}(Q_5 - Q_1) = 2 = c_5, \quad
\]
\[
S(6) = \frac{1}{6}(Q_6 - Q_3 - Q_2 + Q_1) 
= 14 = c_6.
\]
\end{example}

We briefly confirm that the conclusion of Theorem~\ref{thm:QNR} is in heuristic agreement with Theorem~\ref{thm:numcycles-randwalk}.  Heuristically, we expect $\#\mathcal{Q}(N) \approx
\ell^{N/2}$.  Estimating the inner sum in $Q_N$ with a Hurwitz class number heuristically of size $\sqrt{4\ell^N-x^2}$ (ignoring log factors), we have
\[
Q_N \approx
\sum_{x=1}^{\lfloor 2\ell^{N/2} \rfloor} \sqrt{4\ell^N-x^2} 
\approx
\int_1^{2\ell^{N/2}} \sqrt{ 4\ell^N - x^2} \, dx 
\approx \ell^{N}, 
\]
from which follows $c_N \approx \ell^{N}/ N$ as expected.  In fact, a rigorous upper bound of magnitude $\ell^N \log N$ on $c_N$ is given in  Corollary \ref{cor:cyclesbound}, whose proof is based on this rough idea.

\begin{corollary} \label{cor:cyclesbound}
With the notation of Theorem \ref{thm:QNR}, 
\[ 
Q_N < B_N := \frac{2}{3} (e^{\gamma} \log \log (2\ell^{N/2}) + 7/3) \log(4\ell^N)(\pi\ell^N + 2\ell^{3N/4}) 
\]
and 
\begin{align*}
c_N &< \frac{B_N}{N} + \left (e^{\gamma} \log \log N + 7/3 - \frac{1}{N} \right ) B_{N/2} \\
&< \frac{2\pi e^{\gamma} \log (4\ell)}{3} \, \ell^N (\log N + \log \log (2\sqrt{\ell}) + 7/3) + O(\ell^{3N/4}\log N) \mbox{ as $N \rightarrow \infty$},
\end{align*}
where $\gamma =  0.577 \ldots$ is the Euler-Mascheroni constant.
\end{corollary}

\begin{proof}
For brevity, put $X = \lfloor 2\ell^{N/2} \rfloor$. For our bound, we simply ignore all the conditions on $x$ imposed by membership in the set $\mathcal{Q}_N$ of Theorem \ref{thm:QNR} and use $\epsilon_{\mathcal{O},\ell} \le 2$. That is, we upper-bound $Q_N$ by $Q_N' \geq Q_N$, where
\[ Q_N' = 2 \sum_{x=1}^X \sum_{f^2 \mid x^2 - 4\ell^N} h \left (\frac{x^2-4\ell^N}{f^2} \right ) .\]
Fix an integer $x$ with $0 \le x \le X$ and write $4\ell^N - x^2 = y^2d$  for positive integers $d, y$ with $d$ squarefree. For any imaginary quadratic order $\mathcal{O}$ of discriminant $\Delta$, we have $h_{\mathcal{O}} < (\sqrt{|\Delta|} \log|\Delta|)/3$ by \cite[Equation (8.11)]{Pell}, so
\begin{align*}
\sum_{f^2 \mid (x^2 - 4\ell^N)} h \left (\frac{x^2-4\ell^N}{f^2} \right )
    &< \frac{1}{3} \sum_{f^2 \mid (x^2-4\ell^N)} \frac{\sqrt{4\ell^N-x^2}}{f} \log \left ( \frac{4\ell^N-x^2}{f} \right ) \\
    & < \frac{1}{3}\log (4\ell^N) \sum_{f^2 \mid y^2d} \frac{\sqrt{y^2d}}{f}  \\
    & = \frac{1}{3} \log(4\ell^N) \sqrt{y^2d}\, \sum_{f \mid y} \frac{1}{f} \\
    &= \frac{1}{3} \sqrt{4\ell^N - x^2} \, \log(4\ell^N) \, \frac{\sigma(y)}{y} ,
\end{align*}
where $\sigma(\cdot)$ denotes the sum of divisors function. It is known (see \cite[Exercise 3.9 (a)]{Apostol} for example) that $\sigma(n)\varphi(n) \le n^2$ for all $n\ge 1$, where $\varphi(\cdot)$ is Euler's totient function. By \cite[Theorem 15]{RosserSchoenfeld}, we have
\[ \frac{n}{\varphi(n)} < e^{\gamma} \log \log n + \frac{2.50637}{\log \log n} \]
for all $n \ge 3$. If $y \ge 19$, then
\[ \frac{2.50637}{\log \log y} \leq \frac{2.50637}{\log \log (19)} < \frac{7}{3} , \]
so $y = \sqrt{(4\ell^N-x^2)/d} < 2\ell^{N/2}$ yields
\[ \frac{\sigma(y)}{y} \le \frac{y}{\varphi(y)} < e^{\gamma} \log \log y+ \frac{2.50637}{\log \log (y)}
     < e^{\gamma} \log \log (2\ell^{N/2}) + \frac{7}{3} .
\]
Simple numerical verification shows that $\sigma(y)/y \le 7/3$ for $1 \le y \le 18$ (the largest value is $\sigma(12)/12 = 7/3$), so this upper bound holds for all $y > 0$. It follows that 
\[ \sum_{f^2 \mid x^2 - 4\ell^N} h \left (\frac{x^2-4\ell^N}{f^2} \right )
    < \frac{1}{3} (e^{\gamma} \log \log (2\ell^{N/2}) + 7/3) \log(4\ell^N) \sqrt{4\ell^N - x^2} . \]
It remains to bound $\sum_{x=1}^X \sqrt{4\ell^N - x^2}$. Since $f(x) = 4\ell^N - x^2$ is monotonically decreasing, we have
\[ \sum_{x=1}^X \sqrt{4\ell^N - x^2} \le \int_0^X \sqrt{4\ell^N - x^2} \, dx
     = \frac{1}{2} X \sqrt{4\ell^N - X^2} + 2\ell^N \arcsin \left ( \frac{X}{2\ell^{N/2}} \right ) . \]
     
Now $2\arcsin(X/2\ell^{N/2}) \le 2\arcsin(1) = \pi$. Furthermore, $0 \le 2\ell^{N/2} - X < 1$ and
$0 < 2\ell^{N/2} + X < 4\ell^{N/2}$, so
\[ \frac{1}{2} X \sqrt{4\ell^N - X^2} = \frac{1}{2} X \sqrt{(2\ell^{N/2} - X)(2\ell^{N/2} + X)} < \frac{1}{2} \, 2\ell^{N/2} \sqrt{4\ell^{N/2}} =  2 \ell^{3N/4} . \]
It follows that $Q_N' < B_N$ where  
\[ B_N = \frac{2}{3} \big (e^{\gamma} \log \log (2\ell^{N/2}) + 7/3 \big ) \log(4\ell^N)(\pi\ell^N + 2\ell^{3N/4}) . \]
Next, we bound $c_N$. The quantity $B_N$ is monotonically increasing as a function of $N$, so $Q_{N/r} < B_{N/r} \le B_{N/2}$ for all divisors $r$ of $N$ with $r \geq 2$. Thus,
\[ c_N = \frac{1}{N} \left (Q_N + \sum_{1 < r \mid N}\mu(r) Q_{N/r}\right ) \le \frac{1}{N} \left (Q_N + \sum_{1 < r \mid N} Q_{N/r}\right )
    < \frac{1}{N} \left (B_N +(\sigma(N)-1) B_{N/2}\right ).  \]
Since $N \ge 3$, we can bound $\sigma(N)$ as before, i.e.\  $\sigma(N) < N (e^{\gamma} \log \log N + 7/3)$, to obtain
\[ c_N <  \frac{1}{N} \bigg (B_N + (N(e^{\gamma} \log \log N + 7/3)-1) B_{N/2} \bigg )
    =  \frac{B_N}{N} + \left (e^{\gamma} \log \log N + 7/3 - \frac{1}{N} \right ) B_{N/2}. \]
The asymptotic bound on $c_N$ is obtained by bounding the quantity $2\ell^{N/2}$ in the double log by $(2\sqrt{\ell})^N$.
\end{proof}

\begin{remark}
\label{remark:burgess}
The proof of Corollary~\ref{cor:cyclesbound} simply ignores the restriction that $p$ should not split and $\ell$ should split in any of the orders under consideration.  If we wish to prove a lower bound, these restrictions must be addressed.
There is a widely believed heuristic that a Legendre symbol $\left( \frac{f(x)}{p} \right)$ should take values $\pm 1$ with equal probability, independent of the size of $x$.  However, known bounds on the smallest $x$ for which a polynomial $f(x)$ is a non-residue modulo $p$ are comparatively weak.  There is a bound following from work of Weil; for example, Burgess mentioned this in 1967, observing there is a non-residue $x < H$ for $H \gg \sqrt{p} \log p$~\cite{Burgess_QuadChar}.  Improvements have only been made in general for polynomials in at least two variables (in contrast to the special case of $f(x)=x$, where this problem is known as the least non-residue problem); for a discussion and literature review, see \cite{PierceXu_BurgessBds}. 
\end{remark}

\section{Path finding with oriented isogeny volcanoes}
\label{sec:pathfinding}

 In this section, we review the previous work of the present authors \cite{paperone} which provides inspiration for the subject of this paper.
 
As discussed in the introduction, finding a path between two curves in a supersingular $\ell$-isogeny graph compromises the security of a variety of cryptographic systems. The bijection of Theorem~\ref{thm:mainbij-nobase} serves as inspiration for producing such a path, by highlighting the many  volcano structures present in the graph, hinting that one might navigate the graph by reference to these oriented volcanoes. In fact, we saw that every cycle in the supersingular $\ell$-isogeny graph is a rim of a volcano in some oriented isogeny graph. In previous work \cite{paperone}, the authors presented explicit classical and quantum algorithms for finding a path from any elliptic curve $E$ to an initial curve $E_{\text{init}}$. Here, $\Einit$ is generally taken to be the supersingular elliptic curve of $j$-invariant 1728, but other choices for $\Einit$ are also possible. The algorithms use the knowledge of one endomorphism $\theta$ on $E$, or equivalently, an explicit orientation of $E$ by some imaginary quadratic field~$K$, to navigate the oriented $\ell$-isogeny graph $\mathcal{G}_{K,\ell}$. Their complexity is subject to certain plausible heuristics, further detailed below, and is in part governed by the time it takes to evaluate $\theta$ on points on $E$.  The reader is encouraged to consult \cite{WesolowskiOrientations} for a different approach to solving the endomorphism ring problem using the data of an orientation.

The runtime of the classical algorithm additionally depends subexponentially on the degree $d$ of $\theta$ and linearly on a certain class number that can be significantly smaller than the class number of the quadratic order of discriminant~$\Delta$ defined by $\theta$. 
 This runtime can be substantially improved in many cases, leading in particular to certain new families of endomorphisms on \emph{every} supersingular elliptic curve, whose exposure gives rise to a classical polynomial-time algorithm for finding an $\ell$-isogeny path from $E$ to $E_{\text{init}}$. While this fact may on first glance seem alarming, the proof is non-constructive and thus does not provide an efficient strategy for finding such an offending endomorphism. So this result does not pose a threat to the security of isogeny based cryptosystems.

The quantum path finding algorithm relies on two different quantum subroutines:  solving an oriented variant of the \textsc{Vectorization} problem, which asks to find a class group element acting to take one oriented curve to another (see also \cite{chenu2021higherdegree, WesolowskiOrientations}), and a new problem called \textsc{PrimitiveOrientation}, which asks, for a given endomorphism, to determine the quadratic order for which the orientation it induces is primitive.  The resulting algorithm finds a smooth $\ell$-isogeny from $E$ to $\Einit$ such that the smoothness bound and the runtime are subexponential in $\log |\Delta|$.

For completeness, we reproduce the key results of \cite{paperone} here and give a brief road map of the path finding algorithms. Let $E/\Fpbar$ be a supersingular elliptic curve and $\theta$ an endomorphism on $E$ of degree $d$ and discriminant $\Delta$ not divisible by $p$. Denote by $T_{\theta}(k,p)$ the time it takes to evaluate $\theta$ on points on $E$ defined over $\mathbb{F}_{p^k}$. If $\theta$ is given by rational maps, then $T_{\theta}(k,p)$ is a polynomial in $d$, $k$ and $\log p$. But if $\theta$ is given as the product of endomorphisms of smaller degree, then the dependence of $T_{\theta}(k,p)$ on $d$ is much more favourable. All the algorithms of \cite{paperone} assume in practice one of these two representations, which are typical in applications involving elliptic curve endomorphisms. 

For any discriminant $\Delta$ of some quadratic order, the \emph{$\ell$-fundamental part} of $\Delta$ is the unique discriminant~$\Delta'$ such that $\Delta = \ell^{2m} \Delta'$, with $v_{\ell}(\Delta') \le 1$ when $\ell$ or $\Delta$ are odd and $v_{\ell}(\Delta'/4) \le 1$ when $\ell = 2$ and $\Delta$ is even. For brevity, write $L_x(1/2) = \exp(O(\sqrt{\log x \log \log x}))$.

\begin{theorem}[{\cite[Theorem 11.2]{paperone}}] \label{thm:paperoneThm1.1}
Let $E/\Fpbar$ be a supersingular elliptic curve and $\theta$ an endomorphism on $E$ of degree $d$ and discriminant $\Delta$ coprime to $p$. Let $\ell \ne p$ be a prime (assumed to be constant in the runtime estimate),  $\Delta'$ the $\ell$-fundamental part of $\Delta$, and $h_{\Delta'}$ the class number of the quadratic order of discriminant~$\Delta'$. If $L_d(1/2)$ is bounded below by a polynomial in $\log p$ and $|\Delta'| \le p^2$, then there is a classical algorithm that finds an $\ell$-isogeny path of length $O(\log p + h_{\Delta'})$  from $E$ to the curve $\Einit$ of $j$-invariant $j=1728$ in heuristic runtime $T_\theta(L_d(1/2),p) +h_{\Delta'} L_d(1/2)\poly(\log p)$. 
\end{theorem}

\begin{theorem}[{\cite[Theorem 11.3]{paperone}}] 
\label{thm:paperoneThm1.3}
Let $E/\Fpbar$ be a supersingular elliptic curve and $\theta$ an endomorphism on $E$ of discriminant $\Delta$ coprime to $p$ such that $T_{\theta}(k,p)$ is bounded below by a polynomial in $\log(p^k)$. If $|\Delta| \le p^2$, then there is a quantum algorithm that finds an $L_{|\Delta|}(1/2)$-smooth isogeny of norm $O(\sqrt{|\Delta|})$ from $E$ to $\Einit$ in heuristic runtime $T_\theta(O(\log^2d),p)L_{|\Delta|}(1/2)$.
\end{theorem}

The basic strategy of the algorithm of Theorem~\ref{thm:paperoneThm1.1} (\cite[Algorithm 8.1]{paperone}) is as follows. Use the orientation induced by $\theta$ as an ``orienteering tool'' to climb from $E$ to the rim of the oriented $\ell$-isogeny volcano containing~$E$. Generate this entire rim via the action of the class group of the associated~$\ell$-fundamental order $\mathcal{O}$ on the rim curves and store all their $j$-invariants in a list $L$. Orient the curve $\Einit$ (whose endomorphism ring is known) by the same orientation as $E$ and use it to climb its volcano. Hoping to reach the same rim as the path originating at $E$, i.e.\ a vertex in $L$, construct the desired path from $E$ to $\Einit$ from the two ascending paths and one of the two rim segments connecting them.  

Throughout the computation, it is imperative to keep the size of the endomorphisms and the discriminants under consideration reasonable. To that end, endomorphisms are maintained in factored form as a  product of endomorphisms of bounded prime power degree, and a technique akin to sieving produces such a representation for which the trace of the endomorphism is also of manageable size. Navigating the oriented $\ell$-isogeny graph requires effective subroutines for detecting whether an edge is ascending/descending/horizontal, traversing an edge and carrying along the orientation, dividing an endomorphism by $\ell$ (to traverse an ascending edge), evaluating an endomorphism on $\ell$-torsion points (required for the class group action), and factoring an endomorphism into a product of endomorphims of prime power degree while controlling the size of its reduced trace. Algorithms for performing all these computational tasks, accompanied by proofs of correctness and a complete complexity analysis, can be found in~\cite{paperone}. Code is available at \cite{github} and was used for a proof-of-concept example to generate an explicit $2$-isogeny path of length 5 from $E: y^2 = x^3 + (7i+86)x + (45i+174)$ to $\Einit: y^2 = x^3 - x$ over $\FF_{179^2}$, where $i^2 = -1$. 

A crucial ingredient in the path finding technique is orienting $\Einit$ appropriately. Given an $\ell$-fundamental order $\mathcal{O} \subset K$ of discriminant $\Delta$, this requires finding an endomorphism $\theta \in \End(\Einit)$ that generates an order $\mathcal{O}' \subseteq \mathcal{O}$.
The relative index $[\mathcal{O}:\mathcal{O}']$ is a power of $\ell$, say $\ell^r$ for some $r \geq 0$. For reasons of efficiency, $r$  should be small, situating $\mathcal{O}'$ as close to a rim as possible. At the same time, to guarantee the existence of a path from the starting curve $E$ to $\Einit$ in $\mathcal{G}_{K,\ell}$, the orientation induced by $\theta$ must place $\Einit$ in the same connected component (i.e.\ volcano) of $\mathcal{G}_{K,\ell}$ as the oriented starting curve. A natural strategy is thus to loop over $r = 0, 1, 2, \ldots$ until the desired endomorphism $\theta$ and order $\mathcal{O}'$ are found. This raises the key question of how large $r$ needs to be to ensure that a suitable orientation is encountered, at least with high probability. The answer is governed by how vertices corresponding to a fixed curve are spread around the cordillera. We will return to this issue in a moment. 

A number of the algorithms outlined above depend on heuristics. Quoting from the introductory text of~\cite{paperone}:

\begin{quote}
We rely on a number of heuristic assumptions: 
\begin{enumerate*}[label=(\roman*)]
\item The Generalized Riemann Hypothesis.
\item Powersmoothness in a quadratic sequence or form is as for random integers (a powersmooth analogue of the heuristic assumption underlying the quadratic sieve; see \cite[Heuristics~5.10 and 9.3]{paperone}).
\item \label{item:j} The orientations of a fixed $j$-invariant are distributed reasonably across all suitable volcanoes (\cite[Heuristic 3.7]{paperone}).
\item This distribution is independent of a certain integer factorization (\cite[Heuristic 6.7]{paperone}).
\item The aforementioned integer factorization is prime with the same probability as a random integer (\cite[Heuristic~6.4]{paperone}; this heuristic is similar to those used in \cite{deQuehenEtAl_ImprovedTorPt} and \cite{KLPT}).
\end{enumerate*}
\end{quote}

Among these, item (i) is standard, and items (ii) and (v) are purely number theoretical questions.  This leaves items (iii) and (iv) as the only heuristics which concern oriented isogeny graphs.  Unfortunately, item~(iv) seems out of reach at the moment, so we will focus on item (iii) which indeed addresses our earlier question. It asks us to consider, for a fixed $j$-invariant, how its various orientations appear in a cordillera of volcanoes:  does it prefer certain volcanoes over others? At the same time, it is natural to consider the `opposite' question:  for a fixed volcano, how do the various $j$-invariants lie upon it?  Do some $j$-invariants sit at greater depth than others?  This can be rephrased as a question about the covering map from an oriented volcano to the usual supersingular $\ell$-isogeny graph $\mathcal{G}_\ell$. In Section \ref{sec:randwalk}, we address some of these questions.

\section{Random walks and consequences for oriented isogeny volcanoes}
\label{sec:randwalk}

 In this section, we use a standard random walk result for Ramanujan graphs to address some questions about the appearance of $j$-invariants on oriented $\ell$-isogeny volcanoes.

\subsection{Random walks in an $\ell$-isogeny graph}

 In this entire section, as a matter of convenience, \textbf{we assume that $p \equiv 1 \pmod{12}$ to avoid issues of extra automorphisms}.  In particular, strictly speaking, Proposition~\ref{prop:randwalk2} depends upon results about regular undirected Ramanujan graphs.  Of course, the result can be expected to hold for all $\mathcal{G}_\ell$, regardless of $p$, which are at worst very nearly undirected regular graphs, having some perturbations around the two vertices $j=0$ and $j=1728$ if these belong to $\mathcal{G}_\ell$.
 
 We refer to the length (number of $\ell$-isogenies) of a shortest path between two curves in the $\ell$-isogeny graph $\mathcal{G}_\ell$ as the \emph{distance} between the two curves .   The diameter of $\mathcal{G}_\ell$ (the maximum distance between pairs of points) is known to be bounded by $2\log p$ \cite[Theorem 1]{Pizer_RamGraphsHecke}; see also \cite{LubetzkyPeres}. 
 One way to obtain such results is using the Ramanujan property and standard methods for expander graphs \cite{ExpanderGraphsAndApps}; our version and proof here closely follow \cite[Theorem 1]{GalbraithPetitSilva_IdProtocols}.

\begin{proposition}
\label{prop:randwalk2}
Let $p \equiv 1 \pmod{12}$.
Let $\ell$ be a prime and $\mathcal{G}_\ell$ be the supersingular $\ell$-isogeny graph over~$\overline{\FF}_p$, with $\Np$ vertices.  Let $\mathcal{S}$ represent a subset of $\Sp$ initial vertices.  Let $j_1$ represent a target vertex.  Let $j$ represent the vertex reached by a random walk of $t$ steps from a random vertex in $\mathcal{S}$.  We denote by $\operatorname{Pr}_t$ the probability of an event over all such random walks of length $t$.  Then we have
\[
\left|\operatorname{Pr}_t(j = j_1) - \frac{1}{\Np} \right| \le \frac{1}{\Sp}\left( \frac{2\sqrt{\ell}}{\ell+1} \right)^t.
\]
\end{proposition}

\begin{proof} 
Choose a basis of a vector space of dimension $\Np$ to correspond to the vertices of $\mathcal{G}_\ell$.
Let $A_{\ell}$ represent the $\Np \times \Np$ adjacency matrix of $\mathcal{G}_\ell$.
Let $\mathbf{v}_0$ represent an initial vector whose entry for vertex $v$ is $1/\Sp$ if $v \in \mathcal{S}$ and $0$ otherwise.  This is the probability vector for the initial vertex of the random walk.  Let $\lambda_{1}$ be the largest eigenvalue of $A_{\ell}$, and $\lambda_{2}$ the second largest.  Recursively define $\mathbf{v}_t = \frac{1}{\lambda_1} A_\ell \mathbf{v}_{t-1}$, so $\mathbf{v}_t$ is the probability vector for the $t$-th vertex of the random walk.

Let $\mathbf{u}$ have all entries $1/\Np$, so $\mathbf{u}$ is the  probability vector for the uniform distribution over the vertices of~$\mathcal{G}_\ell$.  Then $\mathbf{u}$ is an eigenvector for $\lambda_1$, which has an eigenspace of dimension $1$.  Choose an orthogonal basis of eigenvectors $\mathbf{e}_i$ for $\mathbb{R}^{\Np}$ such that $\mathbf{e}_1 = \mathbf{u}$. Then since $\mathbf{v}_{t-1}$ is a probability distribution, $\langle \mathbf{v}_{t-1}, \mathbf{1} \rangle = \langle \mathbf{u}, \mathbf{1} \rangle = 1$.  Writing $\mathbf{v}_{t-1} = \sum_{i} \alpha_i \mathbf{e}_i$, this implies that $\alpha_1 = 1$.  Hence $\mathbf{v}_{t-1} - \mathbf{u} = \sum_{i=2}^{\Np} \alpha_i \mathbf{e}_i$.  Therefore,
\[
A_\ell (\mathbf{v}_{t-1} - \mathbf{u})
= \sum_{i=2}^{\Np} \lambda_i \alpha_i \mathbf{e}_i.
\]
Consequently, we have
\begin{equation*}
    \|\mathbf{v}_t - \mathbf{u} \|_2 = \left|\left| \frac{1}{\lambda_1} A_\ell (\mathbf{v}_{t-1} - \mathbf{u}) \right|\right|_2 
    \le \frac{1}{\lambda_1} \lambda_2 \| \mathbf{v}_{t-1} - \mathbf{u} \|_2.
\end{equation*}
We also have
\begin{equation*}
    \|\mathbf{v}_0 - \mathbf{u}\|_2 =
    \Sp \left( \frac{1}{\Sp} - \frac{1}{\Np} \right)^2 + (\Np - \Sp) \left( \frac{1}{\Np} \right)^2 
    = \frac{1}{\Sp} \left( 1 - \frac{\Sp}{\Np} \right)
    \le \frac{1}{\Sp}.
\end{equation*}
Iterating, we have
\begin{equation*}
    \left| \Pr(j=j_1) - \frac{1}{\Np} \right|
    = \| \mathbf{v}_t - \mathbf{u} \|_\infty
    \le \| \mathbf{v}_t - \mathbf{u} \|_2 
    \le \frac{\lambda_2^t}{\lambda_1^t} || \mathbf{v}_0 - \mathbf{u} ||_2 
    \le \frac{\lambda_2^t}{\Sp \lambda_1^t} 
    \le \frac{1}{\Sp} \left( \frac{2 \sqrt{\ell} }{\ell+1} \right)^t.
\end{equation*}
The final inequality is the bound on the spectral gap \eqref{eqn:spectralgap} for these graphs.
\end{proof}

\begin{corollary}
\label{cor:allcurves}
Let $j_1$ be a $j$-invariant and $\mathcal{S}$ a subset of the vertices of $\mathcal{G}_\ell$. If 
\[ t > \frac{\log \Np - \log \Sp}{\log ((\ell + 1)/2 \sqrt{\ell})}  \]
then there exists a path of length $\le t$ in $\mathcal{G}_\ell$ starting in $\mathcal{S}$ and ending at $j_1$.  In particular, for
\[
t > \log_{\sqrt{\ell}/2} \Np - \log_{\sqrt{\ell}/2} \Sp,
\]
all curves are within distance $t$ of $\mathcal{S}$.
\end{corollary}

\begin{proof}
This follows from Proposition~\ref{prop:randwalk2} by multiplying through by $\Np$, to obtain
\[
| \Np \operatorname{Pr}_t(j=j_1)  -1 | \le \frac{\Np}{\Sp} \left( \frac{2 \sqrt{\ell} }{\ell+1} \right)^t
\]
and then letting $t$ be large enough so the right side is less than $1$. 
\end{proof}


\subsection{The depth of $j$-invariants on an oriented volcano}
\label{sec:depth-of-j}

Each volcano of the oriented $\ell$-isogeny graph includes every $j$-invariant of the supersingular graph (to see this, fix a $j$-invariant $j_0$ of the volcano, and then consider an $\ell^n$-isogeny to the desired $j$-invariant $j_1$; this generates a path in the oriented graph).  It is natural to ask how far down a volcano one must go to see all $j$-invariants.  We may use Corollary~\ref{cor:allcurves} to answer this.

\begin{corollary}
\label{cor:j-rim}
Let $\mathcal{V}$ be an oriented volcano, and set $\mathcal{S} := \{ j \in \mathcal{G}_\ell : (E_j,\iota) \text{ is on the rim of $\mathcal{V}$ for some $\iota$} \}$.  
Then all $j$-invariants appear on $\mathcal{V}$ at depth less than or equal to
\[
\lceil \log_{\sqrt{\ell}/2}\Np - \log_{\sqrt{\ell}/2}\Sp \rceil.
\]

\end{corollary}

\begin{remark}
The proof of \cite[Proposition 6.5]{paperone}, particularly equation (3), provides a sort of computational proof of this behaviour in the case of the Gaussian field.
\end{remark}

It is an interesting question to ask about the cardinality of $\mathcal{S}$, which is related to the question of the multiplicity of $j$-invariants in $\SSO$.  That is, how many $\mathcal{O}$-primitive orientations are there for a fixed curve $E$?   Kaneko showed that a curve cannot have more than one embedding when the discriminant is at most $p$ in absolute value \cite[Theorem 2']{Kaneko}.  On the other hand, \cite[Proof of Theorem 1.4]{EOY} shows that for any fixed curve $E$ and integer $t$, orders $\mathcal{O}$ of sufficiently large discriminant will always embed into $\End(E)$ at least $t$ times.  The number of curves with an $\mathcal{O}$-orientation is bounded below in \cite{Antonin_lowerbound2022}.

\subsection{The volcanoes upon which the orientations of a fixed $j$-invariant lie}
\label{sec:volcanoes-upon-which}

Every $j$-invariant occurs in every volcano (as described at the beginning of Section~\ref{sec:depth-of-j}).  In \cite{paperone}, the algorithms depended upon the following heuristic about the relative frequency of a fixed $j$-invariant across different volcanoes (quoted as item \eqref{item:j} in Section~\ref{sec:pathfinding} above).

 \begin{heuristic}[{\cite[Heuristic 3.7]{paperone}}]
 \label{heur:uniform-volcanoes}
 Let $\mathcal{O}$ be an $\ell$-fundamental order in a quadratic field $K$.
 Consider the finite union $\mathcal{S}$ of $\mathcal{O}'$-cordilleras in the oriented supersingular $\ell$-isogeny graph $\mathcal{G}_{K,\ell}$ for all $\mathcal{O'} \supseteq \mathcal{O}$.  Let $d(v)$ denote the depth of a vertex $v$ on its volcano.  Let $j(v)$ denote its $j$-invariant.  Define:
  \begin{itemize} \itemsep 0pt
     \item $R_{\mathcal{V}}$, the number of edges descending from the rim of the volcano $\mathcal{V} \in \mathcal{S}$;
     \item $R_{\mathcal{S}}$, the number of edges descending from all rims in $\mathcal{S}$.
\end{itemize}
 Then for any $j$-invariant $j_0$ and any volcano $\mathcal{V} \in \mathcal{S}$, the ratio
 \[
 \frac{
\#\{v \in \mathcal{V}: j(v)=j_0, d(v) \le t \}
}{
\#\{ v \in \mathcal{S}: j(v)=j_0, d(v) \le t \}
}
\] 
approaches $R_\mathcal{V}/R_\mathcal{S}$ as $t \rightarrow \infty$.
 \end{heuristic}

However, na\"ive methods can only prove a weighted version of this result, where $j$-invariants at smaller depth are weighted more heavily.  Our result also has a restriction on the rim sizes.

\begin{theorem} \label{thm:uniform-volcanoes}
  Consider a cordillera $\mathcal{C}$ in the oriented supersingular $\ell$-isogeny graph $\mathcal{G}_{K,\ell}$ whose volcanoes are all isomorphic as graphs and $(\ell+1)$-regular (e.g.\ a single $\mathcal{O}$-cordillera).
  Let $\mathcal{V}$ be one volcano of the cordillera $\mathcal{C}$ and $\mathcal{S}_{\mathcal{V}}$ the vertex set of its rim.
  Let $\mathcal{S}_\mathcal{C}$ be the set of vertices in the rims of the full cordillera.
  Let $d(v)$ denote the depth of the vertex $v$, and let $f_t(d)$ denote the function that counts the number of paths of length $t$ starting at the rim of any volcano and terminating at any fixed vertex at depth $d$ of that volcano (by symmetry, this is independent of the vertex but depends on the depth and rim size).
If we define
\begin{equation}
\label{eqn:s}
S_{\mathcal{V},j_1,t} =  \frac{ \displaystyle \sum_{v \in \mathcal{V}, \, j(v)=j_1} f_t(d(v)) }{ \displaystyle \sum_{v \in \mathcal{C}, \, j(v)=j_1} f_t(d(v))  } \, ,
\end{equation}
then
\[
\left| S_{\mathcal{V},j_1, t} - \frac{\#\mathcal{S}_\mathcal{V}}{\#\mathcal{S}_\mathcal{C}} \right| \ll \left( \frac{2 \sqrt{\ell}}{\ell+1} \right)^t,
\]
where the implied constant depends on $\mathcal{C}$ and $\ell$, but not $t$.
\end{theorem}

\begin{proof}
The numerator of \eqref{eqn:s} is the number of times we encounter $j=j_1$ among all paths of length $t$ originating at the rim of $\mathcal{V}$.  The denominator is the corresponding number for all paths of length $t$ originating at the rims of $\mathcal{C}$.  
The number of paths of length $t$ from the rim set of $\mathcal{V}$ is $(\ell+1)\ell^{t-1} \#\mathcal{S}_\mathcal{V}$.
   So, applying Proposition~\ref{prop:randwalk2} with $\mathcal{S} = \mathcal{S}_\mathcal{V}$, 
   we have
\[
\frac{1}{(\ell+1)\ell^{t-1}} \sum_{v \in \mathcal{V}, j(v)=j_0} f_t(d(v))
= Pr(j = j_0) \#\mathcal{S}_\mathcal{V} = \frac{ \#\mathcal{S}_\mathcal{V}}{\Np}
+ O\left(  \left( \frac{2\sqrt{\ell}}{\ell+1} \right)^t \right).
\]
Combining with the analogous fact for the denominator gives the result.
\end{proof}

The restriction to $(\ell+1)$-regularity rules out the rims having oriented curves with automorphisms (see Remark~\ref{prop:gkl-auto}); it is merely a convenience and the result could be adjusted to accommodate these cases.

This result is insufficient for the purposes of \cite{paperone} in that it doesn't rule out the possibility that some volcanoes have few occurrences of $j=j_0$ at small depth, while others have more such occurrences at greater depth.  If this is indeed the case, it may be that Algorithm 6.1  of \cite{paperone}, which proceeds by depth, preferentially finds certain volcanoes early on. (Nevertheless, Corollary~\ref{cor:j-rim} above provides some reassurance that all $j$-invariants begin to appear after a while.)

\bibliographystyle{abbrv}   
\bibliography{refs} 

\end{document}